\documentclass[11pt]{article}
\usepackage{enumerate}
\usepackage{bbm}
\usepackage{amsfonts}
\usepackage[T1]{fontenc}
\usepackage{latexsym,amssymb,amsmath,amsfonts,amsthm}
\usepackage{graphicx, color}
\usepackage{subfigure}
\usepackage{epsf}
\usepackage{ulem}
\usepackage{epstopdf}
\usepackage{amssymb}
\usepackage{mathrsfs}
\usepackage{diagbox}
\usepackage[ruled]{algorithm2e}

\newcommand{\R}{{\mathbb R}}
\newcommand{\Z}{{\mathbb Z}}

\newcommand{\Sp}{{\mathbb S}}

\newcommand{\be}{\begin{eqnarray}}
\newcommand{\ben}{\begin{eqnarray*}}
\newcommand{\en}{\end{eqnarray}}
\newcommand{\enn}{\end{eqnarray*}}
\newcommand{\ba}{\backslash}
\newcommand{\pa}{\partial}

\newcommand{\ov}{\overline}

\newcommand{\om}{\omega}

\newcommand{\ka}{\kappa}

\newcommand{\la}{\lambda}

\newcommand{\hth}{\theta}
\newcommand{\hx}{\hat{x}}
\newtheorem{theorem}{Theorem}[section]
\newtheorem{lemma}[theorem]{Lemma}
\newtheorem{corollary}[theorem]{Corollary}

\newtheorem{remark}[theorem]{Remark}
\newtheorem{proposition}[theorem]{Proposition}

\begin{document}
\title{\bf Identifying strictly convex obstacles from backscattering far field data}
\author{Jialei Li\thanks{School of Mathematical Sciences, University of Chinese Academy of Sciences, Beijing 100049, China, and Academy of Mathematics and Systems Science,
Chinese Academy of Sciences, Beijing 100190, China. Email: lijialei21@mails.ucas.ac.cn},\and
Xiaodong Liu\thanks{Academy of Mathematics and Systems Science,
Chinese Academy of Sciences, Beijing 100190, China. Email: xdliu@amt.ac.cn},\and
Qingxiang Shi\thanks{Corresponding author. School of Mathematical Sciences, Dalian University of Technology, Dalian 116024, China. Email: sqxsqx142857@126.com}}
\date{}
\maketitle

\begin{abstract}
The recovery of anomalies from backscattering far field data is a long-standing open problem in inverse scattering theory. We make a first step in this direction by establishing the unique identifiability of strictly convex impenetrable obstacles from backscattering far field measurements. Specifically, we prove that both the boundary and the boundary conditions of the strictly convex obstacle are uniquely determined by the far field pattern measured in backscattering directions for all frequencies. The key tool is Majda's asymptotic estimate of the far field patterns in the high-frequency regime. 
Furthermore, we introduce a fast and stable numerical algorithm for reconstructing the boundary and computing the boundary condition. A key feature of the algorithm is that the boundary condition can be computed even if the boundary is not known, and vice versa. Numerical experiments  demonstrate the validity and robustness of the proposed algorithm. 


\vspace{.2in}

\vspace{.2in} {\bf AMS subject classifications:}
35P25, 45Q05, 78A46

\end{abstract}

\section{Introduction}

Inverse obstacle scattering problems (ISPs) concern reconstructing unknown obstacles from measurements of scattered fields resulting from incident waves. These problems are of paramount importance in numerous applications, including radar, non-destructive testing, medical imaging, geophysical prospecting, and remote sensing. Substantial theoretical and computational advances have been made in the past decades for solving time harmonic ISPs, as comprehensively documented in the seminal monograph by Colton and Kress \cite{CK}.

However, ISPs can not be considered completely solved, especially from a numerical point of view, and remain the subject matter of much ongoing research.
The fundamental difficulties in solving ISPs originate from their nonlinearity and ill-posed nature, which present significant challenges for numerical computations.
In recent decades, non-iterative reconstruction methods have drawn considerable research interest due to their computational efficiency and robustness, including the linear sampling method \cite{ColtonKirsch}, the factorization method \cite{Kirsch98, KirschGrinberg}, and more recently developed direct sampling methods \cite{CCHuang, ItoJinZou, LiLiuZou, LiZou, LiuIP17, Potthast2010}. In particular, direct sampling methods inherit many advantages of the classical linear sampling method and factorization method, e.g., they are independent of any a priori information on the geometry and physical properties of the unknown objects.
Moreover, the major advantage of direct sampling methods is their computational simplicity. Specifically, only the inner product of the measurement with some suitably chosen function is involved in the computation of the indicators of the direct sampling methods. This formulation renders the methods computationally efficient and remarkably stable against measurement noise. 

A notable limitation of the sampling methods discussed above is their qualitative nature. While they can effectively provide an approximation of the locations and shapes of the objects, they generally cannot determine the physical properties.
This shortcoming significantly restricts their applicability in many practical scenarios, such as geophysical prospecting, where precise determination of boundary conditions is often essential to identify the targets. 
In particular, the determination of the impedance parameter has been a central focus in inverse scattering research. 
To the best of our knowledge, it is always assumed that the boundary is known exactly to non-iteratively solve the impedance parameter via either moment equation methods \cite{ChengLiuNaka2003, Colton1982, ColtonKirsch1981} or probe-type methods \cite{CaconiColton2004, CaconiColtonMonk2010, LiuNakaSini2007}. The moment equation methods involve solving an integral equation of the first kind with noisy kernel and noisy right-hand sides. Therefore, such methods are highly unstable, particularly for variable impedance parameters. 
Probe-type methods offer a direct reconstruction formula for the impedance parameter at any boundary point but require measurements in all observation directions from all incident directions, which may be expensive in applications.
For simultaneous reconstruction of both the boundary and the impedance parameter, we refer to gradient-based iterative methods, including the first regularized Newton method by Kress and Rundell \cite{KressRundell2001} and more \cite{HeKindSini2009, KressRundell2018, Lee2014, Serrenho2006}. 

The other limitation for direct sampling methods is that they are usually designed for the full aperture scenario. However, practical implementations often must deal with the more challenging limited-aperture configuration and even backscattering data, i.e., one receiver and one transmitter with a fixed location are used to collect data. 
Currently, there is still a lack of global uniqueness results for inverse backscattering problems. For the case of Sch\"{o}rdinger operator and acoustic medium, it is shown that the local uniqueness is valid in a dense subset of compactly supported smooth potential (or medium) \cite{EskinRalston3, EskinRalston2,Uhlmann2001}. However, the global uniqueness remains open, although researchers are making progress on this subject \cite{Lagergren, OlaPS, RakUnl, Serov}.
For inverse obstacle backscattering problems, there are even fewer results due to the high nonlinearity. Kress and Rundell \cite{KressRun98} proved the local uniqueness near the unit disk by analyzing the Fr\'{e}chet derivatives. Based on the high-frequency asymptotics of the far field, Christiansen \cite{Christiansen2013} proved that a strictly convex real-analytic obstacle with Dirichlet or Neumann boundary condition can be determined by the generalized backscattering data in two directions. Besides, Shin \cite{Shin2016} showed that the radius of a sound-soft ball can be uniquely determined via phaseless backscattering data and proposed a frozen Newton method for numerically reconstructing convex objects.
The numerical methods using the multi-frequency backscattering far field patterns date back to Bojarski \cite{Bojarski} via physical optics approximation. The well-known Bojarski identity is proposed for reconstructing sound-soft or sound-hard obstacles. In a similar spirit, Arens, Ji, and Liu \cite{ArensJiLiu2020} developed a direct sampling method to reconstruct perfect conductors in the electromagnetic scattering context. Numerical experiments indicate that the high-frequency requirement could be relaxed to a certain extent, but a rigorous proof is not known. 
Additionally, there are some numerical results using backscattering aperture data. 
We refer to \cite{ColtonMonk} and \cite{DouLiuMengZhang} for the linear sampling method and the direct sampling method, respectively, for limited-aperture single frequency "backscattering" data (in the sense that both the incident and observation directions are made on the same aperture). The size is well reconstructed, but the reconstruction is highly elongated down range. Li, Liu, and Wang \cite{LiLiu2015, LiLiuWang2017} developed a scheme for convex polygonal obstacles based on a certain local maximum behavior of high-frequency far field pattern in the backscattering aperture.

This paper introduces a novel approach to ISPs through the analysis of point-wise backscattering data. 
We show that both the boundary geometry and the physical property of a convex obstacle can be uniquely determined from the  multi-frequency backscattering far field patterns. 
To the best of our knowledge, this is the first rigorous uniqueness result for identifying all the parameters of an obstacle from the backscattering far field patterns. 
Furthermore, we design an efficient and robust algorithm for identifying the obstacles, which is the first solution for the simultaneous determination of both shape and impedance parameters from the backscattering far field data. In particular, we develop direct sampling methods that are capable of reconstructing the obstacle's boundary independent of boundary conditions, and vice versa. 

The remainder of this paper is structured as follows: Section \ref{ProblemSetting} outlines the problem setup, while Section \ref{sec-InvP} presents the proof of uniqueness. Numerical algorithms based on this proof are introduced in Section \ref{sec-Numeric}, and numerical examples are provided in Section \ref{NumExamples} to demonstrate the effectiveness of the proposed methods.

\section{Problem setting}\label{ProblemSetting}

We begin with the formulation of the acoustic scattering problem. Let $k=\om/c>0$ be the wave number of a time harmonic wave, where $\om>0$ and $c>0$ denote the frequency and sound speed, respectively. In the whole paper, we consider multiple frequencies in a bounded band, i.e.,
\ben
k\in (k_{-}, k^{+}),
\enn
with two positive wave numbers $k_{-}$ and $k^{+}$.
Furthermore, let the incident field $u^{in}$ be a plane wave of the form
\be\label{incidenwave}
u^{in}(x)\ =\ u^{in}(x,\hth,k) = e^{ikx\cdot \theta},\quad x\in\R^n,
\en
where $\theta\in\Sp^{n-1}:=\{x\in\R^n:|x|=1\}$ ($n=2,3$) denotes the direction of the incident wave.

The scattering of plane waves by impenetrable obstacle $D$ involves finding the total field $u=u^{in}+u^s$ such that
\be
\label{HemEquobstacle}\Delta u + k^2 u = 0\quad \mbox{in }\R^n\setminus\ov{D},\\
\label{Bc}\mathcal{B}(u) = 0\quad\mbox{on }\pa D,\\
\label{Srcobstacle}\lim_{r:=|x|\rightarrow\infty}r^{\frac{n-1}{2}}\left(\frac{\pa u^{s}}{\pa r}-iku^{s}\right) =\,0,
\en
where $\mathcal{B}$ denotes one of the following three boundary conditions
\ben
(1)\,\mathcal{B}(u):=u\quad\mbox{on}\, \pa D;\qquad
(2)\,\mathcal{B}(u):=\frac{\pa u}{\pa\nu}\quad\mbox{on}\ \pa D;\qquad
(3)\,\mathcal{B}(u):=\frac{\pa u}{\pa\nu}+ik\la u\quad\mbox{on}\ \pa D
\enn
corresponding the cases in which the scatterer $D$ is sound-soft, sound-hard, and of impedance type, respectively.
Here, $\nu$ is the unit outward normal on $\pa D$ and $\la\in L^\infty(\partial D), \inf\la(x)>0$, is a positive real-value function on $\pa D$.

The well-posedness of the direct scattering problems \eqref{HemEquobstacle}--\eqref{Srcobstacle}
have been established and can be found in \cite{CK,Mclean}.
Every radiating solution of the Helmholtz equation has the asymptotic behavior at infinity 
\be\label{0asyrep}
u^s(x,\theta,k)
=\frac{e^{ikr}}{r^{\frac{n-1}{2}}}\left\{u^{\infty}(\hx,\theta,k)
+\mathcal{O}\left(\frac{1}{r}\right)\right\}\quad\mbox{as }\,r:=|x|\rightarrow\infty,
\en
uniformly in all directions $\hx:=x/|x|\in\Sp^{n-1}$ where the function $u^{\infty}(\hx,\theta,k)$ defined on the unit sphere $\Sp^{n-1}$
is known as the far field pattern of $u^s$ with $\hx\in\Sp^{n-1}$ denoting the observation direction.
For $l\in\Z_+\cup \{+\infty\}$, define
\ben
\Theta_l:=\{ \theta_j|j = 1,2,\cdots, l \}\subset \Sp^{n-1}.
\enn

In this paper, we consider the following two types of generalized backscattering direction configurations:
\begin{itemize}
  \item $\mathcal {A}_1:=\{(\hx,\hth)\in \Sp^{n-1}\times\Sp^{n-1} \,| \, \hx=Q\hth,\,\forall \hth\in\Theta_l \}$.
        The first data set $\mathcal{A}_1$ corresponds to the general backscattering experiment. Here, $Q$ is a fixed rotation in $\R^n$. In particular, the data set $\mathcal {A}_1$ reduces to the classical backscattering experiment if $Q=-I$, where $I$ is the identity matrix. The two sensors serve as the source and the receiver. The measurements are then taken by moving these two sensors around the scatterer simultaneously. 
  \item $\mathcal{A}_2:= \mathcal{A}_1^{(1)}\bigcup \mathcal{A}_1^{(2)}\bigcup \mathcal{A}_1^{(3)}$. The second data set $\mathcal{A}_2$ is a union of three generalized data sets $\mathcal{A}_1^{(j)}, j=1,2,3$ with different rotations $Q_j$ and the same direction set $\Theta_l$.
\end{itemize}

The inverse backscattering problem consists of the determination of the boundary $\pa D$ and the boundary condition $\mathcal{B}$ from the far field patterns $u^{\infty}(\hx,\theta,k)$ for all $(\hx,\hth)\in \mathcal {A}_j$ with $j=1\text{ or }2$ and all wave numbers $k\in (k_{-}, k^{+})$.

\section{Uniqueness}\label{sec-InvP}
In this section, we use high-frequency expansion of the far field pattern to derive uniqueness results from multi-frequency data. 
First, we recall Majda's asymptotics (Theorem \ref{Asymp-farfield}) and extend it to the case of an impedance obstacle in $\mathbb{R}^2$. We then prove our main results based on the high-frequency asymptotics.
\subsection{Asymptotics of far field patterns}

\begin{figure}[htbp]
\centering
\includegraphics[width=2in]{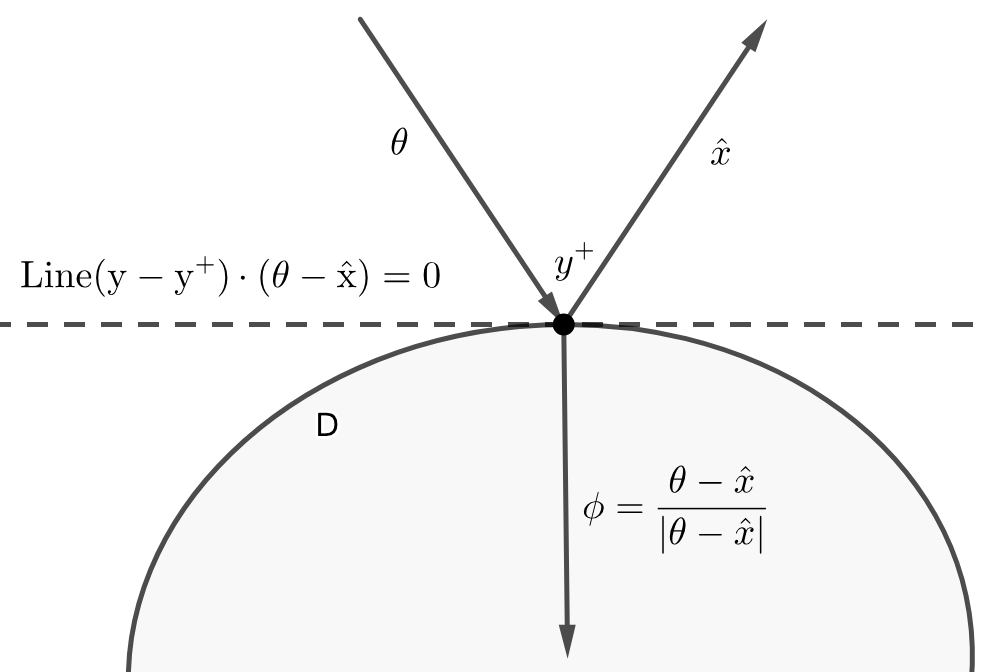}
\caption{
Physical optics diagram for incident direction $\hth$ and observation direction $\hx$.}
\label{pointline}
\end{figure}

As shown in Figure \ref{pointline}, for any fixed $(\hx,\theta)\in\mathcal {A}_j, \,j=1, 2$ satisfying $\hx\neq\hth$,
define the normal vector of the reflecting plane as
\be\label{phi}
\phi(\hx, \hth):=\frac{\theta-\hx}{|\theta-\hx|}.
\en
For strictly convex obstacles, the Gauss map $G: \partial D\to S^{n-1}$ defined by $G(y) = \nu(y)$ is bijective. In this case, we denote $\mathcal{G}:\mathbb S^{n-1}\times\mathbb S^{n-1}\to\partial D$ by
\begin{equation}
	y^+(\hx, \hth)= \mathcal{G}(\hx, \hth):= G^{-1}\left(-\phi(\hx, \hth)\right).
\end{equation}
In other words, $y^+(\hx, \hth)$ is the point at $\partial D$ such that $\nu(y^+(\hx, \hth)) = -\phi(\hx, \hth)$. 
For clarity and brevity, we will simply denote them as $\phi$ and $y^+$ when the context is unambiguous.

A key tool of the uniqueness proof is stated in the following theorem on the high-frequency asymptotic behavior of far field patterns.
\begin{theorem}\label{Asymp-farfield}
    Suppose that $D\subset \mathbb{R}^n, n=2,3$ is a smooth, strictly convex obstacle. Assume that any one of Dirichlet, Neumann, or impedance boundary conditions is satisfied on $\partial D$. Then, for any $\hx\neq \theta$, 
\begin{equation}\label{thm-majda}
\left|u^\infty(\hx, \theta, k)-e^{iky^+\cdot(\hx-\theta)}\kappa^{-1/2}(y^+) R^\lambda(\hx, \theta)\right|=O\left(\frac{1}{k}\right),
\end{equation}
as $k\to \infty$.
Here, $\kappa(y^+)>0$ is the Gauss curvature at $y^+$ and 
\begin{equation*}
R^\lambda(\hx, \theta)=
-|\hx-\theta|^{-\frac{n-1}{2}} \phi\cdot \hx \frac{\lambda(y^+)+\phi\cdot \hx}{\lambda(y^+)-\phi\cdot \hx}
\end{equation*}
is the reflection coefficient concerned with the boundary condition. For the Dirichlet and Neumann boundary conditions, we set $\lambda =\infty$ and $0$, respectively.
\end{theorem}

Theorem \ref{Asymp-farfield} was established by Majda \cite{Majda} except for the impedance obstacles in two dimensions due to technical reasons.  
The results concerning such asymptotic expansion of the radiating solution of the Helmholtz equation can be traced back to Ludwig \cite{Ludwig}. 
Majda \cite{Majda} significantly extended these results through the construction of approximate solutions using the parametrix method developed by Taylor \cite{Taylor}.
Majda specified the high-frequency asymptotics of the approximate solution, as well as the corresponding far field patterns, which is expected to dominate in \eqref{thm-majda}. 

When the discrepancy between the approximate and exact solutions becomes sufficiently small as $k\to\infty$ (see \cite[Lemma 2.2]{Majda}), one can rigorously establish the asymptotic behavior of the true solution.
Majda demonstrated this argument in the three-dimensional scenario in his work. For Dirichlet and Neumann boundary conditions in two dimensions, he referred to earlier results from Morawetz \cite{Morawetz}. 
The remaining challenge lies in establishing the estimates for the impedance boundary condition in two-dimensional settings, which we formally state as follows.
\begin{lemma}\label{lem-impe}
Consider the two-dimensional case, where $u$ is a radiating solution in the exterior of obstacle $D$ with a Lipschitz boundary,  subject to the impedance boundary condition ${\pa u}/{\pa\nu}+ik\la u=g$. Then the following estimate holds:
\begin{equation*}
    \|u\|_{L^2\left(B_R\backslash\ov{D}\right)}\leq C k^\alpha \|g\|_{L^2(\partial D)}.
\end{equation*}
Here, $B_R$ is a disk of radius $R$ centered at the origin such that $\overline{D}\subset B_R$, $\alpha>0$, and $C$ is a constant depending on $R, \alpha, \la$.
\end{lemma}

Spence \cite{Spence} proved this result in 2014 when $\la>0$ is a positive constant. Now we extend Spence's work to establish Lemma \ref{lem-impe}.

The fundamental solution to the Helmholtz equation in two dimensions is given by
\ben
\Phi_k(x,y):= \frac{i}{4} H_0^{(1)}(k|x-y|), x\neq y,
\enn
where $H_0^{(1)}$ denotes the Hankel function of the first kind of order zero.
For any integrable function $\phi\in L^2(\partial D)$, we recall the acoustic single-layer potential $S_k$ and acoustic double-layer potential $D_k$, which are defined respectively by  
\ben
S_k \phi:= \int_{\partial D} \Phi_k(x,y) \phi(y) ds(y), \quad x\in\R^2\ba\pa D,
\enn
and
\ben
D_k\phi := \int_{\partial D} \frac{\partial \Phi_k(x,y)}{\partial \nu(y)} \phi(y) ds(y)\quad x\in\R^2\ba\pa D.
\enn
We employ the notation $a\lesssim b$ if there is a constant $C$ independent of $k$ such that $a\leq C b$.
Now we state two key lemmas that are essential for the proof of Lemma \ref{lem-impe}.
\begin{lemma}\label{lem-layerpotential}
Let $D$ be an obstacle with Lipschitz boundary. For any cutoff function $\chi\in C_c^\infty(\R^2)$, given $k_0>0$, the following estimates hold:
\be
\|\chi S_k\|_{L^2(\partial D)\to L^2(\R^2)}\lesssim k^{-\frac{1}{2}} \quad \text{  and  }\quad
\|\chi D_k\|_{L^2(\partial D)\to L^2(\R^2)}\lesssim k^{\frac{1}{2}}
\en
for all $k>k_0$.
\end{lemma}
We refer the reader to either \cite[Theorem 2.15 and Theorem 2.16]{ChanGLS2012} or \cite[Lemma 4.3]{Spence} for the proof.
Note that the results in Lemma \ref{lem-layerpotential} are independent of the boundary conditions of the obstacle. 

\begin{lemma}\label{lem-cauchydata}
    Let $u\in H^1_{loc}(\R^2\backslash \overline{D})$ be a radiating solution to the Helmholtz equation $\Delta u+k^2 u = 0$ in $\R^2\backslash \overline{D}$ subject to the impedance boundary condition
\ben
\frac{\pa u}{\pa \nu}+ik\la u=g
\enn 
where $\la\in L^\infty(\partial D), \inf \la >0$ and $g\in L^2(\partial D)$. Then we have the following estimates:
\be\label{eq-lem-cauchy}
\left\|\frac{\pa u}{\pa \nu}\right\|_{L^2(\partial D)}\leq \left(1+\frac{\sup \la}{\inf \la}\right)\|g\|_{L^2(\partial D)}
\text{ and }\ 
\|u\|_{L^2(\partial D)}\leq \frac{1}{k\inf \la}\|g\|_{L^2(\partial D)}
\en
\end{lemma}
\begin{proof}
By Green's identity, we obtain
\be
\begin{aligned}
    0&= \int_{B_R\backslash\overline{D}}\ov{u}(\Delta + k^2) u dy\\
    &=k^2\|u\|^2_{L^2(B_R\backslash\overline{D})} - \|\nabla u\|^2_{L^2(B_R\backslash\overline{D})} + \int_{\pa B_R}\frac{\pa u}{\pa \nu} \ov{u} ds(y)-\int_{\partial D}\frac{\pa u}{\pa \nu} \ov{u} ds(y).
\end{aligned}
\en
Taking the imaginary part of both sides and letting $R\to\infty$, it turns out that
\be\label{eq-rellich}
\mathfrak{Im}\int_{\partial D}\frac{\pa u}{\pa \nu} \ov{u} ds(y)
 = \lim_{R\to\infty} \mathfrak{Im}\int_{\pa B_R}\frac{\pa u}{\pa \nu} \ov{u} ds(y)
 = k\|u^\infty\|^2_{L^2(S^1)}>0
\en
by the radiation condition \eqref{Srcobstacle}. Substituting the impedance boundary condition into \eqref{eq-rellich}, we get
\be
-k\int_{\partial D} \la(y) |u(y)|^2ds(y) + \mathfrak{Im} \int_{\partial D} g \ov{u}>0.
\en
Applying the Cauchy-Schwarz inequality, we obtain
\be
\|g\|_{L^2({\partial D})}\|u\|_{L^2({\partial D})} \geq k\inf \la \|u\|^2_{L^2({\partial D})},
\en
which proves the second inequality in \eqref{eq-lem-cauchy}.
The first inequality follows directly from the boundary condition and the triangle inequality.
\end{proof}

We are ready to prove Lemma \ref{lem-impe}.

\begin{proof}[Proof of Lemma \ref{lem-impe}]
By Green's representation formula \cite[Theorem 2.5]{CK}, the solution $u$ can be expressed as
\ben
u = -S_k \frac{\pa u}{\pa \nu} + D_k u, \text{ in } \R^2\backslash \overline{D}.
\enn
Applying the operator norm estimates from Lemma \ref{lem-layerpotential} and the Cauchy data estimates from Lemma \ref{lem-cauchydata}, we obtain the following estimate:
\ben
\begin{aligned}
\|u\|_{L^2(B_R\backslash \overline{D})} &\leq \left\|S_k \frac{\pa u}{\pa \nu}\right\|_{L^2(B_R\backslash \overline{D})} +\|D_k u\|_{L^2(B_R\backslash \overline{D})}\\
&\lesssim k^{-\frac{1}{2}} \left\|\frac{\pa u}{\pa \nu}\right\|_{L^2(\partial D)} + k^{\frac{1}{2}}\|u\|_{L^2(\partial D)}\\
&\leq k^{-\frac{1}{2}}\left(1+\frac{\sup \la+1}{\inf \la}\right)\|g\|_{L^2(\partial D)}.
\end{aligned}
\enn 
This completes the proof of the lemma.
\end{proof}

Theorem \ref{Asymp-farfield} holds universally for all boundary conditions in both two- and three-dimensional settings.

\subsection{Uniqueness results}

Based on Theorem \ref{Asymp-farfield}, we establish the following uniqueness theorem using multi-frequency data in sparse directions.
\begin{theorem}\label{UniqObs}
Let $D\subset \mathbb{R}^n, n=2,3$ be a smooth, strictly convex obstacle with an arbitrary boundary condition given by \eqref{Bc}. Consider an incident-observation direction pair $(\hx, \hth)$ where $\hx\neq \theta$, and assume the reflection coefficient $R^\lambda\neq 0$, i.e., $\lambda(y^+)\neq -\phi\cdot \hx$. Then the following results hold:
    \begin{itemize}
        \item[1.] The line $\{y\in \R^n|y\cdot(\hx-\hth)=y^+\cdot(\hx-\hth)\}$ and the combined parameter $\kappa^{-1/2}(y^+) R^\lambda(\hx, \theta)$ are determined by the far field data $\{u^\infty(\hx, \theta, k)| k\in(0,\infty)\}$.
        \item[2.] Both the curvature $\kappa(y^+)$ and impedance parameter $\lambda(y^+)$ at the reflection point $y^+$ are determined by $\{u^\infty(\hx_j, \theta_j, k)| k\in(0,\infty), j=1,2\}$ for $(\hx_j,\theta_j), j=1,2$ such that $-\phi(\hx_1,\theta_1)=-\phi(\hx_2,\theta_2)$ is the normal vector at $y^+$.
    \end{itemize} 
\end{theorem}
\begin{proof}
\begin{itemize}
    \item[1.] By \eqref{thm-majda} in Theorem \ref{Asymp-farfield},
\begin{equation}\label{eq-asym-int}
\lim_{K\to\infty} \frac{1}{K}\int_K^{2K} u^\infty(\hx, \theta, k)e^{-ikt}dk=
\left\{
\begin{aligned}
    &0, &\text{if }t\neq y^+\cdot(\hx-\theta);\\
    &\kappa^{-1/2}(y^+) R^\lambda(\hx, \theta), &\text{if }t= y^+\cdot(\hx-\theta).
\end{aligned}
\right.
\end{equation}
Obviously, $y^+\cdot(\hx-\theta)$ and $\kappa^{-1/2}(y^+) R^\lambda(\hx, \theta)$ are determined if $R^\lambda(\hx, \theta)\neq 0$, which is satisfied by the assumption. Hence, the line $\{y\in \R^n|y\cdot(\hx-\hth)=y^+\cdot(\hx-\hth)\}$ is determined.
    \item[2.] It suffices to prove that $\la(y^+)$ is determined by
\begin{equation*}
\kappa^{-1/2}(y^+) R^\lambda(\hx_j, \theta_j),\quad  j=1,2	
\end{equation*}
First, $\la(y^+)\neq 0$ if $\kappa^{-1/2}(y^+) R^\lambda(\hx_j, \theta_j)>0$ and $\la(y^+)\neq \infty$ if $\kappa^{-1/2}(y^+) R^\lambda(\hx_j, \theta_j)<0$.
By the assumptions on $\la(y^+)$, we can define the quotient
\begin{equation*}
	L:=\frac{-|\hx_2-\theta_2|^{-\frac{n-1}{2}} \phi_2\cdot \hx_2}{-|\hx_1-\theta_1|^{-\frac{n-1}{2}} \phi_1\cdot \hx_1}\frac{R^\lambda(\hx_1, \theta_1)}{R^\lambda(\hx_2, \theta_2)}=
\frac{\left(\la(y^+) + \phi_1\cdot \hx_1\right)\left(\la(y^+) - \phi_2\cdot \hx_2\right)}{\left(\la(y^+) - \phi_1\cdot \hx_1\right)\left(\la(y^+) + \phi_2\cdot \hx_2\right)},
\end{equation*}
where $\phi_j=\phi(\hx_j, \theta_j), j=1,2$. If $\la(y^+) =0$ or $\infty$, then $L=1$. Conversely, if $L=1$ and $\lambda\neq \infty$, then $\lambda=0$. 
In consequence, $\la(y^+)$ is determined if the boundary condition is Dirichlet or Neumann.

For the impedance case, $\la(y^+)$ is the solution of a quadratic equation
\begin{equation*}
\lambda^2 + \frac{L+1}{L-1}(\phi_2\cdot\hx_2-\phi_1\cdot\hx_1)\lambda - \phi_1\cdot\hx_1\phi_2\cdot\hx_2=0.
\end{equation*}
The left-hand side takes the value $-\phi_1\cdot\hx_1\phi_2\cdot\hx_2<0$ at $\lambda=0$ and goes to $+\infty$ as $|\lambda| \to \infty$. Therefore, the equation has exactly one negative solution and one positive solution, that is, $\la(y^+)$ is determined.
\end{itemize}
\end{proof}

\begin{remark}{\quad }

\begin{itemize}
    \item While the condition $\la(y^+)\neq -\phi\cdot \hx$ is fundamental in Theorem \ref{UniqObs},  the underlying arguments remain valid when supplementary data becomes available.
     Notably, this condition fails for a single direction pair. By incorporating far-field patterns $u^\infty(\hx^*,\theta^*, k)$ corresponding to an additional direction pair $(\hx^*,\theta^*)$ satisfying $-\phi(\hx^*,\theta^*)=\nu(y^+)$, the method can be successfully implemented without relying on his initial assumption. This indicates that the direction set $\mathcal{A}_2$ can reliably identify the obstacle without requiring prior knowledge of $\la$.
    \item In the case when $\lambda(y^+) = -\phi\cdot\hat{x}$, equation \eqref{eq-asym-int} produces zero for all $t\in \mathbb{R}$ - a result that differs markedly from the case when $\lambda(y^+)\neq -\phi\cdot\hat{x}$. Consequently, we can still establish that $\kappa^{-1/2}(y^+) R^\lambda(\hat{x}, \theta) = 0$ under these conditions.
    \item For strictly convex obstacles ($\kappa > 0$), the boundary condition can be partially classified based on the sign of $\kappa^{-1/2}(y^+) R^\lambda(\hat{x}, \theta)$:
\begin{itemize}
    \item When $\kappa^{-1/2}(y^+) R^\lambda(\hat{x}, \theta) \geq 0$, the boundary must be either Dirichlet or impedance with $\lambda(y^+)\geq\phi\cdot\hat{x}$.
    \item When $\kappa^{-1/2}(y^+) R^\lambda(\hat{x}, \theta) \leq 0$, the boundary must be either Neumann or impedance with $\lambda(y^+)\leq\phi\cdot\hat{x}$.
\end{itemize}
Thus, the sign of this quantity can eliminate either Dirichlet or Neumann boundary conditions for fixed directions $\hat{x}$ and $\theta$.
\end{itemize}
\end{remark}

In conclusion, the line tangent to the obstacle at $y^+$ is determined except for the critical condition $\la(y^+)=-\phi\cdot \hx$, in which the main term in the asymptotics \eqref{thm-majda} vanishes. Applying Theorem \ref{UniqObs} to the data set $\mathcal{A}_1$, we can derive the following uniqueness results. 

\begin{corollary} Let $D$ and $\lambda$ satisfy the assumptions in Theorem \ref{UniqObs}. Given far field data $\{u^\infty(\hx, \theta, k)|(\hx, \theta)\in \mathcal{A}_1, k\in(k_-, k^+)\}$, the following results hold:
\begin{itemize}
    \item[1.] The $\mathcal{A}_1$-hull of $D$, defined as
\begin{equation}
D_{\mathcal{A}_1}:=\bigcap_{(\hat{x},\theta)\in\mathcal{A}_1}\left\{ z\in \R^n|z\cdot\phi(\hx, \theta)> y^+\cdot\phi(\hx, \theta)
\right\},
\end{equation}
is uniquely determined.
    \item[2.] For the case of infinitely many measurement directions $(l=\infty)$, both the obstacle boundary $\partial D$ and its boundary condition are uniquely determined.
\end{itemize}
\end{corollary}
\begin{proof}
\begin{itemize}
    \item[1.] This is a direct consequence of Theorem \ref{UniqObs}. 
    \item[2.] The generalized backscattering data $u^\infty(Qd, d,k)$ exhibits analytic dependence on the incident direction $d\in S^{n-1}$, as established in \cite{GriesHS2013}. This analytic property demonstrates the equivalence between our backscattering measurements and full aperture backscattering data. Consequently, we can uniquely determine the convex hull of $D$, which coincides with $D$ itself.
    With the determined obstacle $D$, the Gauss curvature $\ka(x)$ at all boundary points $x\in \pa D$ becomes uniquely specified. This consequently leads to the determination of $R^\la(x)$ through Theorem \ref{UniqObs}, from which the function $\la(x)$ naturally follows.
\end{itemize}
\end{proof}

\section{Numerical Algorithm}\label{sec-Numeric}
This section is devoted to introducing the reconstruction algorithm. For simplicity, we describe our algorithm for the two-dimensional case. The three-dimensional case can be obtained by proper modification.
The far field patterns are obtained via the boundary integral equation method \cite{CK,KR95}.
We further perturb these synthetic data with relative random noise as follows:
 \begin{align}
 \label{dataerror}
     u^{\infty,\delta}(\hx,\theta,\omega)=u^{\infty}(\hx,\theta,\omega)\Big(1+\delta\big[X_{\theta,\hat{x},\omega}+iY_{\theta,\hat{x},\omega}\big]\Big),
 \end{align}
 where $X_{\theta,\hat{x},\omega},Y_{\theta,\hat{x},\omega}\sim N(0,1)$ are independent random variables, $N(0,1)$ is the normal distribution with mean zero and standard deviation of one. If not stated otherwise, we take $\delta=0.1$.
 
\subsection{Determine the boundary condition}
In this subsection, we introduce two methods to determine the boundary conditions. 

\subsubsection{Determine the boundary condition without knowing $\pa D$}
We begin with a method for computing the impedance coefficient $\lambda$ without knowing the boundary $\partial D$. To do so, we use the multi-frequency far field patterns with direction set $\mathcal{A}_2$.
The theoretical basis is established in the following theorem.
\begin{theorem}\label{determine-lambda-noD}
For every $\hx_0\in\mathbb S^{n-1}$, we select three pairs $(\hx_j,\theta_j)$ satisfying $-\phi(\hx_j,\theta_j)=\hx_0,\ j=1,2,3$. Using $\left\{|u^{\infty}(\hx_j,\theta_j,k)|\Big|k\in(k_-, k^+),j=1,2,3\right\}$,
we can distinguish between the following two cases:
\begin{itemize}
    \item The scatterer $D$ is either sound-soft or sound-hard.
    \item The scatterer $D$ is of impedance type, in which case  $\lambda\left(\mathcal{G}(-\hx_0,\hx_0)\right)$ can be additionally determined.
\end{itemize}
\end{theorem}
\begin{proof}
According to \eqref{thm-majda}, we define the following equations
     \begin{equation}\label{sovleLambda-noD}
     \begin{aligned}
         L_{j}(\hx_0):&=\lim_{k_-\to\infty} \frac{|\hx_1-\theta_1|^{-\frac{n-1}{2}} \hx_0\cdot \hx_1}{|\hx_j-\theta_j|^{-\frac{n-1}{2}} \hx_0\cdot \hx_j}\frac{\int_{k_-}^{k^+} \left|u^\infty(\hx_j, \theta_j, k)\right|dk}{\int_{k_-}^{k^+} \left|u^\infty(\hx_1, \theta_1, k)\right|dk}\\
         &=\left|\frac{(\lambda - \hx_0\cdot \hx_j)(\lambda + \hx_0\cdot \hx_1)}{(\lambda + \hx_0\cdot \hx_j)(\lambda - \hx_0\cdot \hx_1)}\right|\quad j=2,3.
     \end{aligned}
     \end{equation}

Note that we regard $\lambda$ as $\infty$ and $0$ for the Dirichlet and Neumann boundary conditions. This implies that we know that $\mathcal{B}(u)=u\ {\rm or}\ \frac{\partial u}{\partial \nu}$ if $L_2(\hx_0)=L_3(\hx_0)=1,\ \forall\ \hx_0\in\mathbb S^{n-1}$.
 
Otherwise, the scatterer $D$ is of impedance type. Then, it is sufficient to consider the case that $\lambda\left(\mathcal{G}(-\hx_0,\hx_0)\right) - \hx_0\cdot \hx_1\neq 0$. Similar to the arguments in Theorem \ref{UniqObs}, for $j=2,3$ every equation \eqref{sovleLambda-noD} has two positive solutions, denoted by $\lambda\left(\mathcal{G}(-\hx_0,\hx_0)\right)$ and $\lambda^*_j$, respectively. Without loss of generality, we assume that
\begin{align*}
\frac{(\lambda\left(\mathcal{G}(-\hx_0,\hx_0)\right) - \hx_0\cdot \hx_j)(\lambda\left(\mathcal{G}(-\hx_0,\hx_0)\right) + \hx_0\cdot \hx_1)}{(\lambda\left(\mathcal{G}(-\hx_0,\hx_0)\right) + \hx_0\cdot \hx_j)(\lambda\left(\mathcal{G}(-\hx_0,\hx_0)\right) - \hx_0\cdot \hx_1)}
=L_j(\hx_0)\quad j=2,3.
\end{align*}
Suppose that $\lambda^*_2=\lambda^*_3=\lambda^*>0$ satisfies
\begin{align*}
    \frac{(\lambda^* - \hx_0\cdot \hx_j)(\lambda^* + \hx_0\cdot \hx_1)}{(\lambda^* + \hx_0\cdot \hx_j)(\lambda^* - \hx_0\cdot \hx_1)}=-L_j(\hx_0)\quad j=2,3.
\end{align*}
Then, the equation 
\begin{align*}
    \frac{(\lambda - \hx_0\cdot \hx_2)(\lambda + \hx_0\cdot \hx_3)}{(\lambda + \hx_0\cdot \hx_2)(\lambda - \hx_0\cdot \hx_3)}=\frac{L_2(\hx_0)}{L_3(\hx_0)}.
\end{align*}
has two positive solutions, which leads to a contradiction to the arguments in Theorem \ref{UniqObs}. Hence $\la_2^*\neq \la_3^*$. Then the impedance coefficient $\lambda\left(\mathcal{G}(-\hx_0,\hx_0)\right)$ can be obtained because it appears twice among the four above solutions. 
\end{proof}
Numerically, we take $\alpha_1=8$ and $\alpha_2=10$, set 
\begin{align*}
     \delta k:=\frac{k^+-k_-}{M} ,\quad\mbox{and}\quad     k_m:=k_-+m\delta k,\quad m=0,1,2,\cdots, M.
 \end{align*}
and compute
\begin{align}\label{sovle-lambda-num-noD}
    \mathcal{L}(\hx,\alpha_j):= \frac{|\hx-\theta|^{-\frac{1}{2}} \hx\cdot \hx}{|\hx_j-\theta_j|^{-\frac{1}{2}} \hx\cdot \hx_j}\frac{\sum\limits_{m=0}^M \left|u^{\infty,\delta}(\hx_j, \theta_j, k_m)\right|}{\sum\limits_{m=0}^M \left|u^{\infty,\delta}(\hx, \theta, k_m)\right|},
\end{align}
where $\hx=-\theta$, $\hx_j:=Q_{\alpha_j}\theta_j$ and $\theta_j:=R_{\alpha_j}\theta$
with
\be \label{QandR}
Q_{\alpha_j}:=
\left [
            \begin{array}{ll}
              -\cos{\frac{\alpha_j\pi}{16}}, & \quad \sin{\frac{\alpha_j\pi}{16}}  \\
               -\sin{\frac{\alpha_j\pi}{16}}, &  -\cos{\frac{\alpha_j\pi}{16}}
            \end{array}
          \right]
          \quad{\rm and}\quad 
          R_{\alpha_j}:=
          \left [
            \begin{array}{ll}
              \quad \cos{\frac{\alpha_j\pi}{32}}, & \sin{\frac{\alpha_j\pi}{32}}  \\
               -\sin{\frac{\alpha_j\pi}{32}}, &  \cos{\frac{\alpha_j\pi}{32}}
            \end{array}
          \right],\quad j=1,2.
\en

First, we regard the boundary condition as a Dirichlet or Neumann boundary condition if 
\begin{align}\label{judgement-num-DorN}
    |\mathcal{L}(-\theta,\alpha_j)-1|<\frac{\delta}{2}=0.05,\ \forall\theta\in\Theta_l, \,j=1\, ({\rm or}\, 2).
\end{align}
The threshold values established in \eqref{judgement-num-DorN} for the backscattering scenario induce a dual truncation of the impedance parameter at critical values $\lambda=12.06$ and $\lambda=0.06$. This formulation effectively establishes a classification criterion where impedance values below $0.06$ are treated as Neumann-type boundaries, while those exceeding $12.06$ are interpreted as Dirichlet-type boundaries. As demonstrated in Table  \ref{table-reasonable-lambda}, these carefully selected threshold values yield backscattering far-field data that deviates by approximately $2\%-10\%$ from pure Neumann or Dirichlet reference solutions. 
Therefore, criterion \eqref{judgement-num-DorN} provides a valid means to determine whether the boundary condition is of impedance type. 

\begin{table}[h]
    \centering
    \begin{tabular}{|c|c|c|c|c|c|}
    \hline
    $\lambda$&0\ (Neumann)&0.06&12.06&$\infty\ ({\rm Dirichlet})$\\
    \hline
    $u^{\infty}(d,d,20)$&  -3.3288 + 3.8856i&  -3.5254 + 3.9288i& -4.3081 + 3.6514i&-4.3184 + 3.6405i\\
    \hline
    $u^{\infty}(d,-d,20)$&  -0.8189 + 0.2814i&  -0.7255 + 0.2493i&   0.7007 - 0.2172i&0.8278 - 0.2555i\\
    \hline
    $u^{\infty}(d,d,50)$& -5.5900 + 6.0797i&-5.8112 + 6.1183i& -6.4356 + 5.8676i&  -6.4422 + 5.8608i\\
     \hline
    $u^{\infty}(d,-d,50)$& 0.6111 + 0.6135i&0.5418 + 0.5442i& -0.5109 - 0.5262i& -0.6030 - 0.6217i\\
    \hline
    \end{tabular}
    \caption{The noise-free far field pattern at $d=(1,0)^T$ for a disk $D=\{x\in\mathbb R^2, |x|\leq1.5\}$ with different boundary conditions.}
    \label{table-reasonable-lambda}
\end{table} 

Second, for the impedance boundary condition, combining \eqref{sovleLambda-noD} and \eqref{sovle-lambda-num-noD}, we can obtain two positive solutions for each $\hx$ and $\alpha_j$ by solving 
\begin{align}\label{sovlelambda-num-equation}
    \left|\frac{(\lambda - \hx\cdot \hx_j)(\lambda + 1)}{(\lambda + \hx\cdot \hx_j)(\lambda - 1)}\right|=\mathcal{L}(\hx,\alpha_j).
\end{align}
 Specifically, for $j=1,2$, we have 
\begin{align*}
    \lambda_{\alpha_j,\hx}^{(1)}&:=\frac{1}{2}\left(\sqrt{\left[\frac{1+\mathcal{L}(\hx,\alpha_j)}{1-\mathcal{L}(\hx,\alpha_j)}(1-\hx\cdot\hx_j)\right]^2+4\hx\cdot\hx_j}-\frac{1+\mathcal{L}(\hx,\alpha_j)}{1-\mathcal{L}(\hx,\alpha_j)}(1-\hx\cdot\hx_j)\right),\\
    \lambda_{\alpha_j,\hx}^{(2)}&:=\frac{1}{2}\left(\sqrt{\left[\frac{1-\mathcal{L}(\hx,\alpha_j)}{1+\mathcal{L}(\hx,\alpha_j)}(1-\hx\cdot\hx_j)\right]^2+4\hx\cdot\hx_j}-\frac{1-\mathcal{L}(\hx,\alpha_j)}{1+\mathcal{L}(\hx,\alpha_j)}(1-\hx\cdot\hx_j)\right).
\end{align*}
With the help of Theorem \ref{determine-lambda-noD}, $\lambda(\mathcal{G}(-\hx,\hx))$ is approximated by 
\begin{align}\label{reconstruct-lambda}
    \Tilde{\lambda}(\hx):=\frac{1}{2}\left(\lambda_{\alpha_1,\hx}^{(\mathbbm{i})}+\lambda_{\alpha_2,\hx}^{(\mathbbm{j})}\right),
\end{align}
where
\begin{align*}
    (\mathbbm{i},\mathbbm{j})=\underset{(s,t)\in\{1,2\}\times\{1,2\}}{\arg\min}\frac{\left|\lambda_{\alpha_1,\hx}^{(s)}-\lambda_{\alpha_2,\hx}^{(t)}\right|}{\sqrt{\left(\lambda_{\alpha_1,\hx}^{(s)}\right)^2+\left(\lambda_{\alpha_2,\hx}^{(t)}\right)^2}}.
\end{align*}
Note that 
\begin{align}\label{stab-lambda-L}
\left|\frac{d\lambda_{\alpha_j,\hx}^{(1)}}{d\mathcal{L}}\right|&=\left|\frac{(1-\hx\cdot\hx_j)\frac{1+\mathcal{L}}{1-\mathcal{L}}}{\sqrt{[\frac{1+\mathcal{L}}{1-\mathcal{L}}(1-\hx\cdot\hx_j)]^2+4\hx\cdot\hx_j}}-1\right|\frac{(1-\hx\cdot\hx_j)}{(1-\mathcal{L})^2}
    \leq\frac{2(1-\hx\cdot\hx_j)}{(1-\mathcal{L})^2},
\end{align}
which implies that under the perturbation of noise, equation \eqref{sovlelambda-num-equation} remains stable when $\mathcal{L}$ is far from $1$, whereas the equation becomes ill-posed when $\mathcal{L}\approx1$ ($\mathcal{L}=1$ is equivalent to the case in which $\mathcal{B}u=u$ or $\mathcal{B}u={\partial u}/{\partial \nu}$). This phenomenon also indicates that we need to select a larger $\alpha_j$ to keep $\mathcal{L}$ away from $1$ (note also that too large $\alpha_j$ will strengthen the residual term in \eqref{thm-majda}, thereby making the algorithm less effective) and our criterion for determining the boundary conditions is
reasonable.

\subsubsection{Determine the boundary condition with given $\pa D$}
Given $\partial D$, following Theorem \ref{UniqObs}, one may compute the impedance coefficient $\lambda\left(\mathcal{G}(-\hx,\hx)\right)$ for $-\hx\in\Theta_l$ via data on the direction set $\mathcal{A}_1$. Specifically, for large $k_-$ and $k^+$, the impedance coefficient can be approximately computed by
\begin{align}\label{sovlelambda-num-NeD}
    \lambda\left(\mathcal{G}(-\hx,\hx)\right)\approx\frac{\mathcal{H}+1}{\mathcal{H}-1},
\end{align}
where
 \begin{align*}
     \mathcal{H}(\hx):=\frac{\sqrt{2\kappa(y^+)}}{M+1}\sum\limits_{m=0}^Mu^{\infty,\delta}(\hx,-\hx,k_m)e^{-2ik_my^+\cdot\hx}
 \end{align*}
 with $y^+=\mathcal{G}(-\hx,\hx)$. 

\subsection{Direct sampling methods for the boundary}
We introduce two direct sampling methods for shape reconstructions. Precisely, two indicators with and without boundary conditions, respectively, are proposed.

\subsubsection{Shape reconstruction based on the computed boundary condition}
Before introducing our direct sampling method, we may recall the Kirchhoff's approximation.
For any convex obstacle $D$, let
\be\label{illuminatedShadow}
\pa D_{-}(\theta):=\left\{x\in\pa D |\, \nu(x)\cdot \theta<0\right\} \quad\mbox{and}\quad \pa D_{+}(\theta):= \left\{x\in\pa D |\, \nu(x)\cdot \theta\geq0\right\}
\en
be the illuminated region and shadow region, respectively, with respect to the incident direction $\theta$.
For a large wave number $k$, i.e., for a small wavelength, an obstacle $D$ is locally almost a hyperplane with normal $\nu(x)$ at each point $x\in\pa D$. Consequently, the scattered field can be locally viewed as a reflected plane wave. For the near field, this implies that \cite{CK,LiLiu2015}
\be\label{Kirchhoff-soft}
u =  0\  \hbox{on $\pa D$,}\quad
\frac{\pa u}{\pa\nu} \approx \left\{
                            \begin{array}{ll}
                              2\frac{\pa u^{in}}{\pa\nu}  & \hbox{on $\pa D_{-}(\theta)$;} \\
                              0 & \hbox{on $\pa D_{+}(\theta)$}
                            \end{array}
                          \right.
\en
if $D$ is sound soft, and
\be\label{Kirchhoff-hard}
u \approx \left\{
            \begin{array}{ll}
              2 u^{in} & \hbox{on $\pa D_{-}(\theta)$;} \\
              0 & \hbox{on $\pa D_{+}(\theta)$.}
            \end{array}
          \right.\quad
\frac{\pa u}{\pa \nu} = 0 \,\hbox{on $\pa D$}
\en
if $D$ is sound hard. 

For the impedance case, under the same methodology, 
we assume that 
\begin{align*}
    u^s(y)\approx T(\theta,y)e^{ik\Tilde{\theta}(y)\cdot x},\ y\in\partial D_-,
\end{align*}
where $\Tilde{\theta}(y)\in\mathbb S^{n-1}$ such that $\nu(y)=\mathcal{G}(\theta,\Tilde{\theta})$.
In addition, ignoring the derivative of $T$ and $\Tilde{\theta}$, we approximate the boundary condition as follows
\begin{align*}
    ik\lambda(y)(e^{ikx\cdot \theta}+Te^{ik\Tilde{\theta}\cdot x})+ik(\nu(y)\cdot\theta e^{ikx\cdot \theta}+T\nu(y)\cdot\Tilde{\theta} e^{ikx\cdot \Tilde{\theta}})=0,\, \forall\ y\in \partial D_-.
\end{align*}
Note that $\nu(y)\cdot\theta=-\nu(y)\cdot\Tilde{\theta}(y)$, thus the scattered field has the following approximation
\be\label{Kirchhoff-imp}
u \approx \left\{
            \begin{array}{ll}
              u^{in}-\frac{\lambda+\nu\cdot \theta}{\lambda-\nu\cdot\theta} u^{in} & \hbox{on $\pa D_{-}(\theta)$;} \\
              0 & \hbox{on $\pa D_{+}(\theta)$,}
            \end{array}
          \right.
          \quad{\rm and}\quad
\frac{\pa u}{\pa \nu} \approx \left\{
            \begin{array}{ll}
              \frac{\pa u^{in}}{\pa \nu}+\frac{\lambda+\nu\cdot \theta}{\lambda-\nu\cdot\theta} \frac{\pa u^{in}}{\pa \nu} & \hbox{on $\pa D_{-}(\theta)$;} \\
             0 & \hbox{on $\pa D_{+}(\theta)$.}
            \end{array}
          \right.\quad
\en

Inspired by the famous Bojarski identity \cite{Bojarski}, we deduce that
\ben
\begin{aligned}
u^\infty(\hx, -\hx, k) &= \frac{C_n k^{(n-3)/2}}{2}\int_{\pa D} \left(\frac{\pa u}{\pa \nu} e^{-ik\hx\cdot y}-u\frac{\partial e^{-ik\hx\cdot y}}{\partial\nu} \right)ds(y) \\
&= \frac{C_n k^{(n-3)/2}}{2}\int_{\pa D} \left(-ik\lambda(y)+ik\hx\cdot\nu(y)\right)u e^{-ik\hx\cdot y} ds(y)\\
&\approx \frac{C_n k^{(n-3)/2}}{2}\int_{\pa D_-} \left(-ik\lambda(y)+ik\hx\cdot\nu(y)\right)\left(1-\frac{\lambda(y)-\nu(y)\cdot\hx}{\lambda(y)+\nu(y)\cdot\hx}\right)u^{in} e^{-ik\hx\cdot y} ds(y)\\
& = \frac{C_n k^{(n-3)/2}}{2}\int_{\pa D_-} \frac{\lambda(y)-\nu(y)\cdot\hx}{\lambda(y)+\nu(y)\cdot\hx}(-2ik\hx\cdot\nu(y))  e^{-2ik\hx\cdot y}ds(y)\\
&\approx\frac{C_n k^{(n-3)/2}}{2}\gamma(\hx)\int_{\pa D_-} \frac{e^{-2ik\hx\cdot y}}{\partial\nu}ds(y),\quad  \hx \in S^{n-1}.
\end{aligned}
\enn
 Here, $C_2 = -e^{i\frac{\pi}{4}}/\sqrt{2\pi}, C_3 = -1/2\pi$ and $\gamma(\hx)=\frac{\lambda\left(\mathcal{G}(-\hx,\hx)\right)-1}{\lambda\left(\mathcal{G}(-\hx,\hx)\right)+1}$. The final high-frequency approximation is derived by the stationary phase method (Theorem 7.7.5 in \cite{LH-SPM}). Consequently, by Green's theorem, we have
\ben
\begin{aligned}
\gamma(\hx)^{-1}u^\infty(\hx, -\hx, k)& + i^{3-n}\gamma(-\hx)^{-1}\ov{u^\infty(-\hx, \hx, k)} \\
&\approx \frac{C_n k^{(n-3)/2}}{2}\int_{\pa D} \frac{\pa  e^{-2ik\hx\cdot y}}{\pa \nu} ds(y)\\
&= \frac{C_n k^{(n-3)/2}}{2} \int_{D} \Delta e^{-2ik\hx\cdot y}dy\\
&=-2C_n  k^{(n+1)/2}\int_D e^{-2ik\hx\cdot y}dy.
\end{aligned}
\enn
We rewrite it as
\be\label{eq-Bojarski}
\int_{\mathbb{R}^n} \chi_D e^{-2ik\hx\cdot y}dy
\approx -\frac{\gamma(\hx)^{-1}u^\infty(\hx, -\hx, k) + i^{3-n}\gamma(-\hx)^{-1}\ov{u^\infty(-\hx, \hx, k)}}{2C_n  k^{(n+1)/2}}=: V(\hx, k).
\en
Here, $\chi_D$ is the characteristic function of the domain $D$. $V(\hx,k)$ can be viewed as the far field pattern of source $\chi_D$ in the direction $\hx$ and wave number $2k$. 

Motivated by the indicators in \cite{LiuShi} for the inverse source problem, we define the following indicator using the backscattering data:
\be\label{Iz-bj}
I(z):=\frac{1}{l}\sum_{ -\hx\in \Theta_l} I_{\hx}(z),
\en
where
\ben
I_{\hx}(z) &:=& \int_{k_-}^{k^+} k^{n-1} V(\hx, k)e^{2ik\hx\cdot z} dk\\
    &=&  -\int_{k_-}^{k^+}\frac{\gamma(\hx)^{-1}u^\infty(\hx, -\hx, k)e^{2ik\hx\cdot z} + i^{3-n}\gamma(-\hx)^{-1}\ov{u^\infty(-\hx, \hx, k)e^{-2ik\hx\cdot z}}}{2C_n  k^{(3-n)/2}} dk.
\enn
Letting $l$ tend to infinity, we find that
\begin{align*}
   I(z)&\approx\frac{1}{2^{n-1}\pi}\int_{\mathbb S^{n-1}} I_{\hx}(z)ds_{\hx}\\
   &\approx\frac{1}{2^{n-1}\pi}\int_{k_-<|\xi|<k^+}\int_{\mathbb{R}^n} \chi_D e^{-2i\xi\cdot y}dy\ e^{2i\xi\cdot z}d\xi\\
   &=\frac{1}{2}\int_{k_-<|\xi|<k^+}\mathcal{F}[\chi_D](2\xi)\ e^{2i\xi\cdot z}d\xi.
\end{align*}
This implies that $I(z)$ is again expected to capture the boundary $\pa D$ when $k_-$ is large because $\pa D$ carries the high-frequency Fourier pattern of $\chi_D$.

Numerically, we use the computed $\lambda$ in \eqref{reconstruct-lambda}, define 
\begin{equation}\label{num-set-gamma}
    \Tilde{\gamma}(\hx)=
    \left\{
    \begin{aligned}
        \frac{\Tilde{\lambda}(\hx)-1}{\Tilde{\lambda}(\hx)+1},&\ {\rm if}\ \eqref{judgement-num-DorN} \ {\rm is \ established},\\
        1\qquad\quad,&\ {\rm if}\ \eqref{judgement-num-DorN} \ {\rm is \ not\  established},
    \end{aligned}
    \right.
\end{equation}
and plot the following indicator function on the domain of interest,
\begin{align}\label{Iz-num-bj}
\mathcal{I}(z)=\frac{1}{l}\sum_{ -\hx\in \Theta_l} \mathcal{I}_{\hx}(z),
\end{align}
where 
\begin{align*}
\mathcal{I}_{\hx}(z) :=-\sum\limits_{m=0}^M\frac{\Tilde{\gamma}(\hx)^{-1}u^{\infty,\delta}(\hx, -\hx, k_m)e^{2ik_m\hx\cdot z} + i\Tilde{\gamma}(-\hx)^{-1}\ov{u^{\infty,\delta}(-\hx, \hx, k_m)e^{-2ik_m\hx\cdot z}}}{2C_n  k_m^{1/2}} \delta k.
\end{align*}
Comparing \eqref{Iz-bj} and \eqref{Iz-num-bj}, we derive that 
\begin{equation*}
   \mathcal{I}(z)\approx\left\{
   \begin{aligned}
        &I(z),\ {\rm for\ Dirichlet\ condition},\\
       -&I(z),\ {\rm for\ Neumann\ condition},\\
        &I(z),\ {\rm for\ Impedance\ condition},
   \end{aligned}\right.
\end{equation*}
which means that the indicator function $\mathcal{I}(z)$ is expected to catch the boundary $\partial D$. Moreover, $\mathcal{I}(z)$  can also be used to distinguish whether the boundary condition is the Dirichlet condition or the Neumann condition. We regard the boundary condition as
\begin{equation}\label{I+I-DN}
    \mathcal{B}u=\left\{
    \begin{aligned}
        u,&\ {\rm if}\  \mathcal{I}|_{\partial D^+}<0<\mathcal{I}|_{\partial D^-},\\
        \frac{\partial u}{\partial \nu},&\ {\rm if}\  \mathcal{I}|_{\partial D^+}>0>\mathcal{I}|_{\partial D^-}.
    \end{aligned}\right.
\end{equation}
Here,  $\partial D^{\pm}$ represents the outer and inner neighborhoods of $\partial D$. The explanation for criterion \eqref{I+I-DN} can be obtained from the Gibbs phenomenon. Gibbs phenomenon is a classic phenomenon in signal processing and Fourier analysis, mainly manifested as fixed amplitude overshoot peak and oscillation near discontinuous points when approximating discontinuous periodic signals with finite $M$ term Fourier series, details are shown in Figure \ref{gibbs}. Note that $I$ can be approximated by the difference between the $\lfloor k^+\rfloor$-term  Fourier series of $\chi_D$ and the $\lfloor k_-\rfloor$-term  Fourier series of $\chi_D$.
\begin{figure}[htbp]
    \centering
    \includegraphics[width=0.5\linewidth]{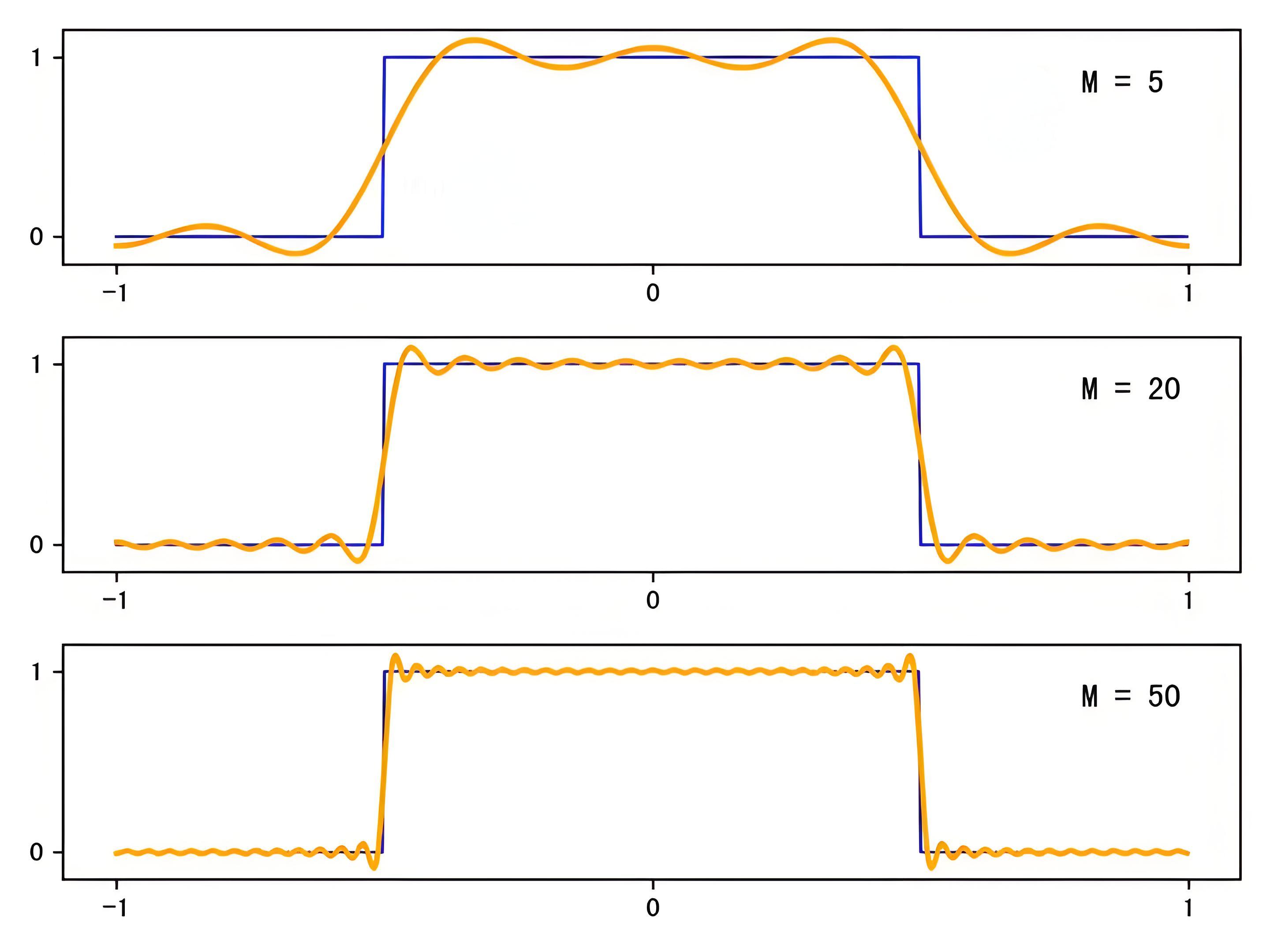}
    \caption{As $M$ increases, the peak of finite $M$-term Fourier series gradually approaches discontinuous points.}
    \label{gibbs}
\end{figure}
 
\subsubsection{Shape reconstruction without the boundary condition}

On the other hand, owing to Theorem \ref{Asymp-farfield},
we define the second indicator:
\begin{equation}\label{Tz-tl}
    \begin{aligned}
        T(z):=\frac{1}{l}\sum\limits_{-\hx\in\Theta_l}\frac{|T_{\hx}(z)|}{\sup\limits_{y\in\mathbb G}|T_{\hx}(y)|}
    \end{aligned}
\end{equation}
with
\ben
T_{\hx}(z):=\int_{k_-}^{k^+}u^{\infty}(\hx,-\hx,k)e^{-2ik\hx\cdot z}k^{(n-3)/2}dk, \quad z\in \mathbb G.
\enn
The behavior of $T_{\hx}(z)$ is illuminated in the following proposition.
\begin{proposition}\label{prop-line}
For a fixed $(\hx, -\hx)\in \mathcal{A}_2$, denote $y^+=\mathcal{G}(-\hx,\hx)$, we have
\ben
T_{\hx}(z) = 
&\kappa^{-1/2}(y^+)R^\la(\hx,-\hx)\left((k^+)^{\frac{n-1}{2}}-(k_-)^{\frac{n-1}{2}}\right) + O\left(\ln{\frac{k^+}{k_-}}\right), &\text{if }(z-y^+)\cdot\hx=0
\enn Moreover, we have
\ben
T_{\hx}(z) = O\left(\frac{k_-^{-\frac{3-n}{2}}}{|(z-y^+)\cdot\hx|}\right) + O\left(\ln{\frac{k^+}{k_-}}\right)
\enn
for any $z$ such that $(z-y^+)\cdot\hx\neq 0$.
\end{proposition}
\begin{proof}
    Inserting Theorem \ref{Asymp-farfield} into $T_{\hx}$ in \eqref{Tz-tl}, we derive
\ben
\Bigg|T_{\hx}(z)- \int_{k_-}^{k^+}k^{\frac{n-3}{2}} \kappa^{-1/2}(y^+)R^\la(\hx,-\hx)e^{2ik(y^+-z)\cdot\hx}dk\Bigg| =  O\left(\ln{\frac{k^+}{k_-}}\right).
\enn
The results follow by straightforward calculations if $(z-y^\pm)\cdot\hx=0$.\\
When $(z-y^+)\cdot\hx\neq 0$, note that
\ben
\begin{aligned}
\int_{k_-}^{k^+}e^{2ik(y^+-z)\cdot\hx}dk &= \frac{1}{2i(y^+-z)\cdot\hx}e^{2ik(y^+-z)\cdot\hx}\Big|^{k^+}_{k_-} \\
&= O\left(\frac{1}{|(z-y^+)\cdot\hx|}\right)
\end{aligned}
\enn
for dimension $n=3$ and that
\ben
\begin{aligned}
\int_{k_-}^{k^+}k^{-\frac{1}{2}}e^{2ik(y^+-z)\cdot\hx}dk 
&= \left.\frac{k^{-\frac{1}{2}}e^{2ik(y^+-z)\cdot\hx}}{2i(y^+-z)\cdot\hx}\right|^{k^+}_{k_-} +\frac{1}{2} 
\int_{k_-}^{k^+}\frac{e^{2ik(y^+-z)\cdot\hx}}{2i(y^+-z)\cdot\hx}k^{-\frac{3}{2}}dk\\
&= O\left(\frac{k_-^{-\frac{1}{2}}}{|(z-y^+)\cdot\hx|}\right) +O\left(\int_{k_-}^{k^+}\frac{k^{-\frac{3}{2}}}{(y^+-z)\cdot\hx}dk\right)\\
&=O\left(\frac{k_-^{-\frac{1}{2}}}{|(z-y^+)\cdot\hx|}\right)
\end{aligned}
\enn
for dimension $n=2$. The proof is complete.
\end{proof}

In conclusion, the indicator $T_{\hx}(z)$ has a large value near the line $\{z\in\mathbb R^2|(z-y^{+})\cdot\hx=0\}$, which passes through $y^+\in \pa D$, and decays when $z$ is away from this line. We divide $T_{\hx}(z)$ by $\sup\limits_{y\in\mathbb G}|T_{\hx}(y)|$ to counteract the effect of $\kappa^{-1/2}(y^+)R^\la(\hx,-\hx) $. The boundary $\pa D$ is consequently shown from the superposition of these lines, i.e., $T(z)$. 

Numerically, we plot the following indicator function:
\begin{align}\label{Tz-num-tl}
  \mathcal{T}(z)=  \frac{1}{l}\sum\limits_{-\hx\in\Theta_l}\frac{|\mathcal{T}_{\hx}(z)|}{\sup\limits_{y\in\mathbb G}|\mathcal{T}_{\hx}(y)|},
\end{align}
where
\begin{align*}
     \mathcal{T}_{\hx}(z)&:=\sum\limits_{m=0}^Mu^{\infty,\delta}(\hx,-\hx,k_m)e^{2ik_m\hx\cdot z}k_m^{-1/2}, \quad z\in \mathbb G.
\end{align*}
Obviously, $\mathcal{T}(z)\approx T(z)$. Therefore, $\mathcal{T}(z)$ is also expected to catch the boundary $\partial D$. Notably, the indicator $\mathcal{T}(z)$ requires less data than the indicator $\mathcal{I}(z)$, but the cost is that the boundary conditions cannot be determined.

\subsection{Summary of reconstruction algorithm}
Based on the methods proposed in the previous two subsections, we formulate Algorithm \ref{Alg-BandD} for identifying an obstacle using the noisy multi-frequency backscattering far field patterns.

\begin{algorithm}[htpb]
\label{Alg-BandD}
\caption{Reconstruct the boundary $\partial D$ and identify boundary condition}
\LinesNumbered 
\KwIn{Far field data set$\left\{u^{\infty,\delta}(\hx,\theta,k)\Big|(\hx,\theta)\in\mathcal{A}_2,\ k\in(k_-,k^+)\right\}$ with $Q_1=-I,\ Q_2=Q_{\alpha_1}$ and $Q_3=Q_{\alpha_2}$ such that $\Theta_l=R_{\alpha_1}\Theta_l=R_{\alpha_2}\Theta_l$. Here, $R_{\alpha_1}$, $R_{\alpha_2}$, $Q_{\alpha_1}$ and $Q_{\alpha_2}$ are defined in \eqref{QandR} with $\alpha_1=8$ and $\alpha_2=10$.} 
\KwOut{The boundary $\partial D$. The type of boundary condition. Furthermore,  for impedance type scatterer $D$, output $\lambda(\mathcal{G}(-\hx,\hx)),\ \hx\in\Theta_L$.}
Compute $\mathcal{L}(\hx,\alpha_j)$ in \eqref{sovle-lambda-num-noD} for $-\hx\in\Theta_l,j=1,2.$\\
Follow \eqref{judgement-num-DorN}, \eqref{sovlelambda-num-equation} and \eqref{reconstruct-lambda} to obtain $\Tilde{\gamma}(\hx)$ in \eqref{num-set-gamma} for $-\hx\in\Theta_l.$ For impedance type scatterer, calculate $\lambda(\mathcal{G}(-\hx,\hx))$.\\
Plot the indicator function $\mathcal{I}(z)$ to reconstruct $\partial D$.\\
Use \eqref{I+I-DN} to distinguish whether the boundary condition is the Dirichlet condition
or the Neumann condition.
\end{algorithm}

{\bf Remarks:}
\begin{itemize}
    \item One may use the indicator $\mathcal{T}$ given by \eqref{Tz-num-tl} directly to reconstruct the boundary $\pa D$.  Different to $\mathcal{I}$ given by \eqref{Iz-num-bj}, the behavior of the indicator $\mathcal{T}$ is independent of the physical property of the obstacles.    
    \item Given that $\pa D$ is known or reconstructed in a high resolution (e.g., by the direct sampling method with the indicator $\mathcal{T}$), we can now compute the impedance coefficient $\lambda$ by the formula \eqref{sovlelambda-num-NeD}. Note that the formula \eqref{sovlelambda-num-NeD} is much more robust to measurement noises but highly rely on the resolution of $\pa D$ since we have used the Gauss curvature.
    \item In step 2, we encounter an inherent ambiguity in distinguishing between Dirichlet and Neumann scatterers, as both boundary conditions satisfy the threshold criterion specified in \eqref{judgement-num-DorN}. This limitation is subsequently resolved in step 4 through the diagnostic behavior of our indicator function $\mathcal{I}(z)$, as mathematically characterized in \eqref{I+I-DN}.
\end{itemize}

\section{Numerical examples and discussions}
\label{NumExamples}
\setcounter{equation}{0}
In this section, a variety of numerical examples are presented in two dimensions to illustrate the applicability and the effectiveness of our sampling methods.
We consider the following two benchmark examples:
\be
\label{Egg}&\mbox{\rm Egg:}&\quad x(t)\ =\left(1.5\cos t, \frac{\sin t}{1+0.2\cos{t}} \right)^T,\quad -\pi\leq t\leq\pi,\\
\label{kite}&\mbox{\rm Kite:}&\quad x(t)\ =\ (\cos t+0.65\cos 2t-0.65, 1.5\sin t)^T,\quad -\pi\leq t\leq\pi.
\en
Note that the egg shaped domain is convex, while the kite shaped domain is concave.
In our simulations, if not stated otherwise, we will always consider equally distributed wave numbers with $\delta k=0.1$ and $64$ equally distributed directions of the incident wave in $\mathbb S^1$. The grids are equally distributed on the rectangle $[-3,3]^2$ with a sampling distance of $0.01$. 

For each example, we consider the following different boundary conditions:
\begin{itemize}
    \item Dirichlet condition,
    \item Neumann condition,
    \item $\lambda_1(t)=1+0.1\sin(t),\ -\pi<t<\pi$,
     \item $\lambda_2(t)=2+0.5\sin(t)+0.2\sin(5t),\ -\pi<t<\pi$.
\end{itemize}
\subsection{Reconstructions for the Egg Example}
Following Algorithm \ref{Alg-BandD}, we begin by distinguishing between the following two scenarios:
\begin{itemize}
    \item The scatterer $D$ is either sound-soft or sound-hard.
    \item The scatterer $D$ is of impedance type.
\end{itemize}

\begin{figure}[htbp]
   \centering
    \begin{tabular}{cccc}
        \subfigure[Dirichlet case.]{
            \label{judgement-D-Egg-20-50}
        \includegraphics[width=0.2\textwidth]{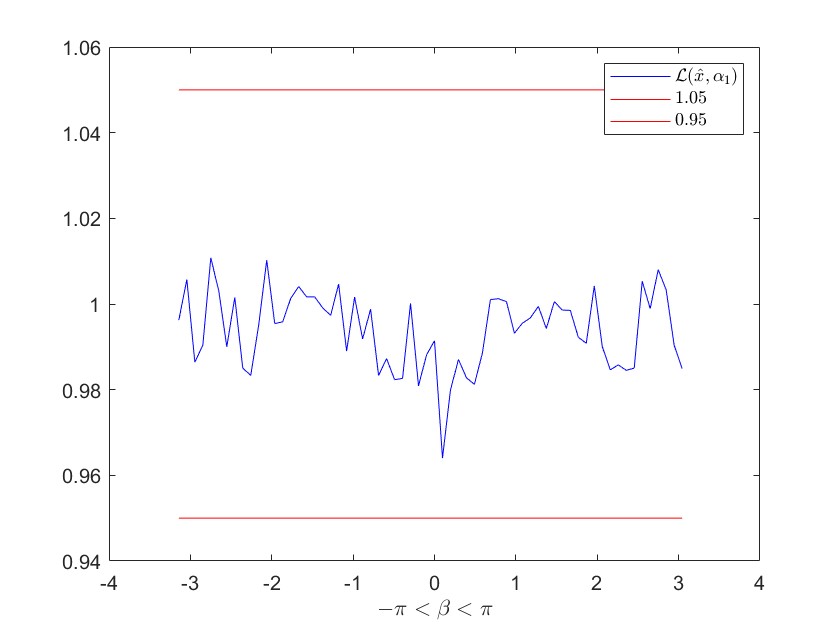}
        }&
        \subfigure[Neumann case.]{
            \label{judgement-N-Egg-20-50}
            \includegraphics[width=0.2\textwidth]{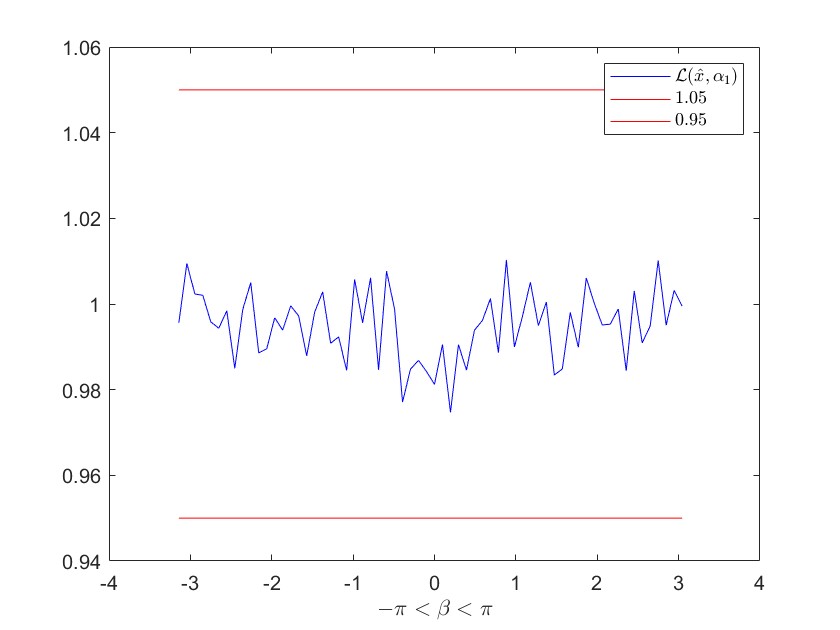}
        }&
        \subfigure[$\lambda_1$ case.]{
            \label{judgement-1S-Egg-20-50}
            \includegraphics[width=0.2\textwidth]{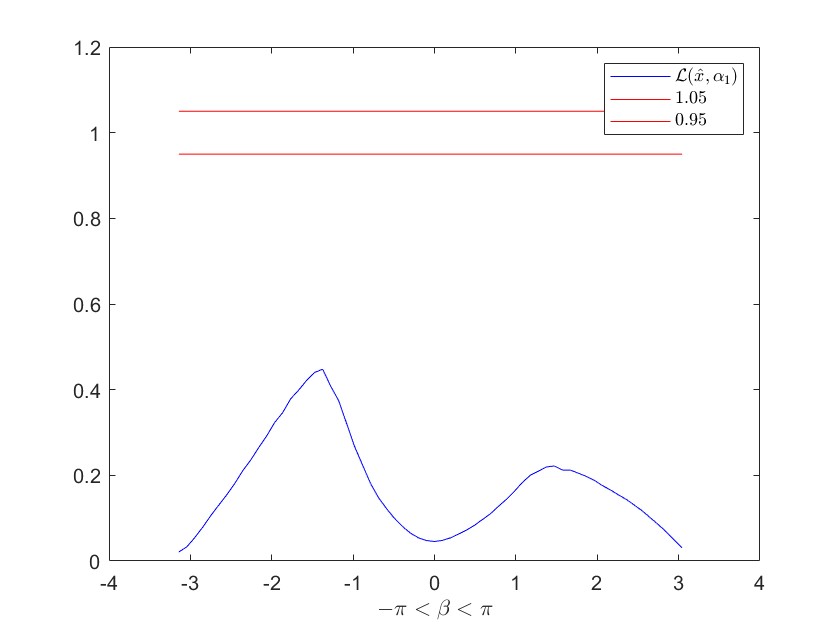}
        }&
        \subfigure[$\lambda_2$ case.]{
            \label{judgement-2SS-Egg-20-50}
            \includegraphics[width=0.2\textwidth]{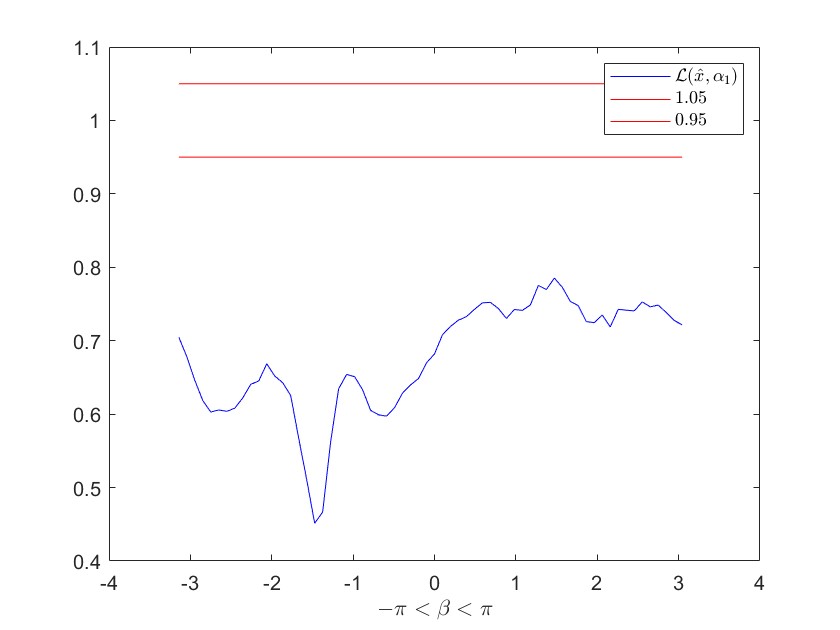}
        }
    \end{tabular}
    \caption{The values of $\mathcal{L}(\hx, \alpha_1)$ for four different boundary conditions with frequency band $[20, 50]$.}
    \label{judgement-Egg-20-50}
\end{figure} 
Taking $\hx=(\cos \beta,\sin \beta)^T,\ -\pi<\beta<\pi$. Based on the criteria \eqref{judgement-num-DorN}, Figure \ref{judgement-Egg-20-50} obviously show that the values of $\mathcal{L}(\hx, \alpha_1)$ indeed lie in the interval $(0.95, 1.05)$ for the Dirichlet and the Neumann cases. Furthermore,  Figure \ref{lambda-egg} shows the reconstruction of $\lambda(\mathcal{G}(-\hx,\hx))=\lambda(\beta)$ using \eqref{reconstruct-lambda}.
\begin{figure}[htbp]
   \centering
    \begin{tabular}{cc}
        \subfigure[Comparison for $\lambda_1$.]{
            \label{lambda1-egg}
        \includegraphics[width=0.3\textwidth]{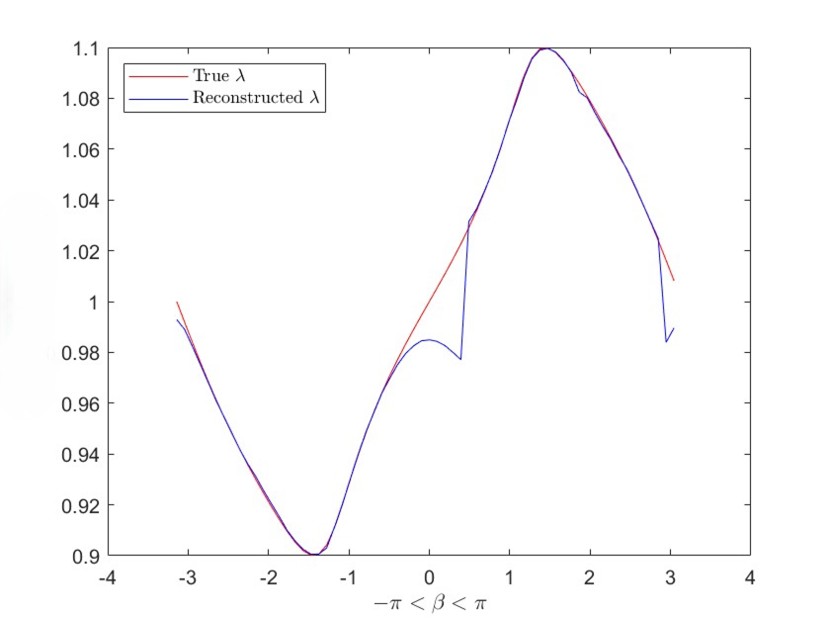}
        }&
        \subfigure[Comparison for  $\lambda_2$.]{
            \label{lambda2-egg}
            \includegraphics[width=0.3\textwidth]{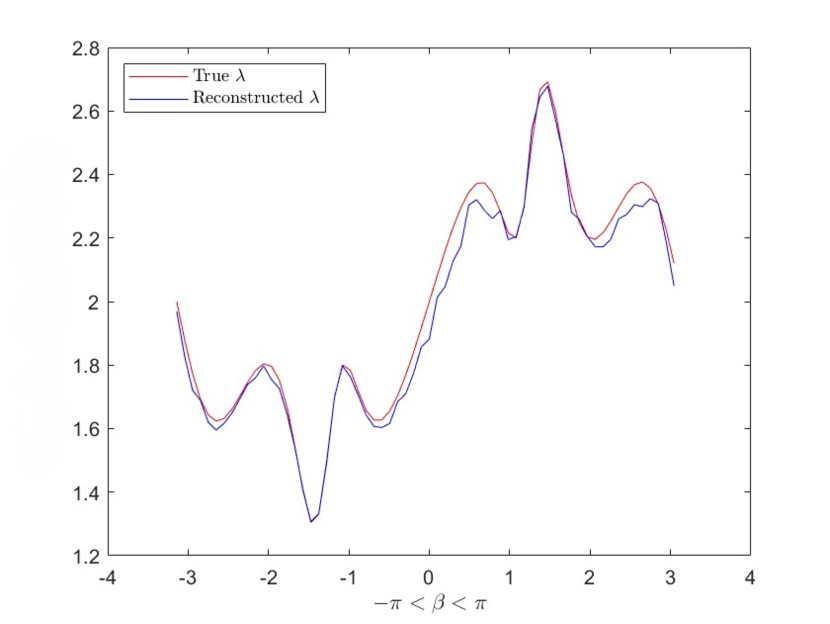}
        }
    \end{tabular}
    \caption{True and reconstructed $\lambda_1$ and $\lambda_2$ using \eqref{reconstruct-lambda} with frequency band $[20, 50]$.}
    \label{lambda-egg}
\end{figure}
The reconstruction of $\lambda$ is well-defined for values $\lambda\approx1$, with the exception of the specific case when $\lambda = 1$.
This is reasonable because $\mathcal{L}\to\infty$  as $\lambda\to1$ in \eqref{sovlelambda-num-equation}, and according to \eqref{stab-lambda-L}, it further makes the equation \eqref{sovlelambda-num-equation} so stable. However, $\mathcal{L}$ is always calculated as a finite value, which is naturally not good for reconstruction in the case of $\lambda=1$.
\begin{figure}[htbp]
   \centering
    \begin{tabular}{cccc}
        \subfigure[Dirichlet case.]{
            \label{RD-D-Egg-20-50}
        \includegraphics[width=0.2\textwidth]{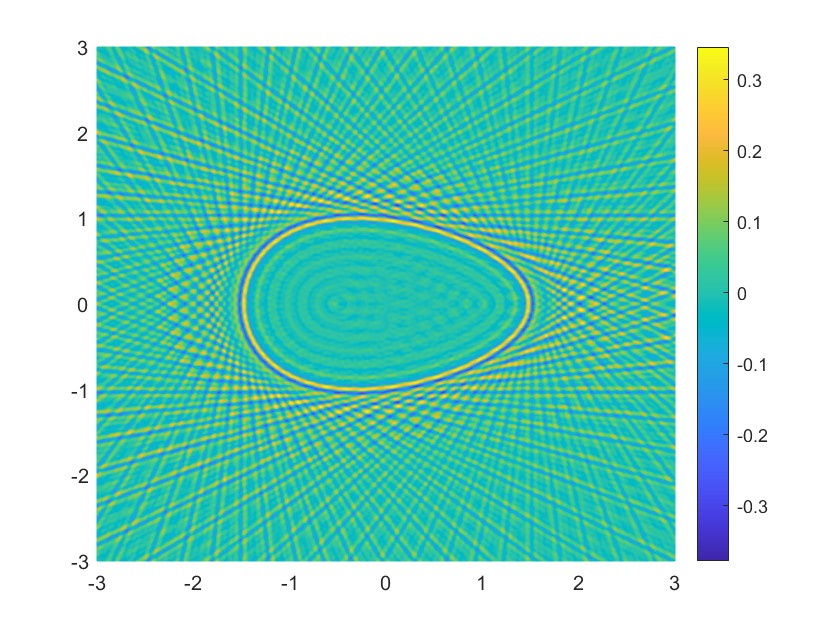}
        }&
                \subfigure[Neumann case.]{
            \label{RD-N-Egg-20-50}
        \includegraphics[width=0.2\textwidth]{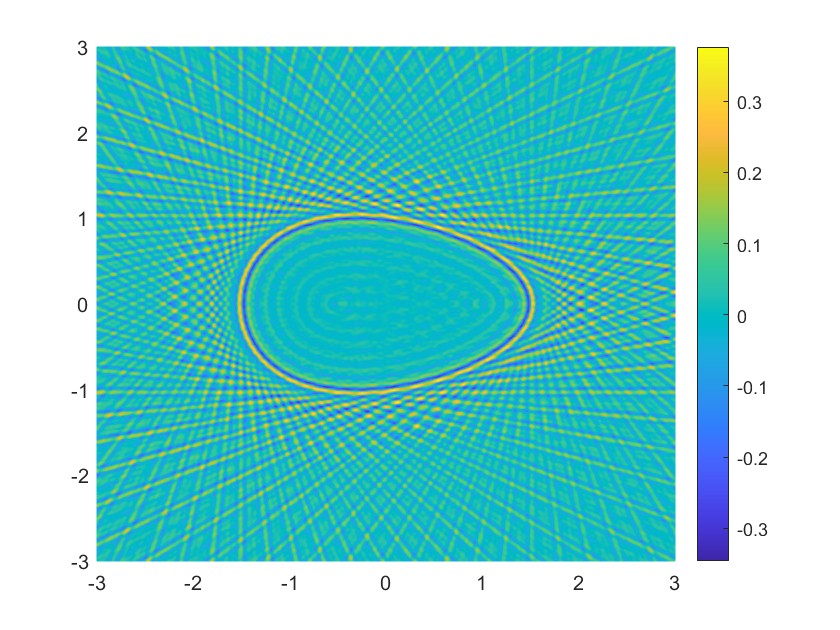}
        }&
                \subfigure[$\lambda_1$ case.]{
            \label{RD-1S-Egg-20-50}
        \includegraphics[width=0.2\textwidth]{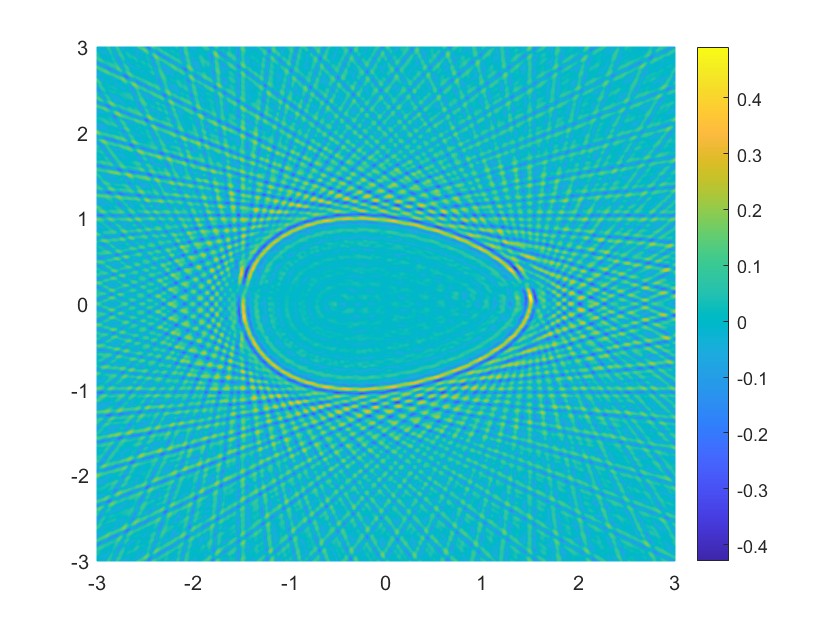}
        }&
        \subfigure[$\lambda_2$ case.]{
           \label{RD-2SS-Egg-20-50}
            \includegraphics[width=0.2\textwidth]{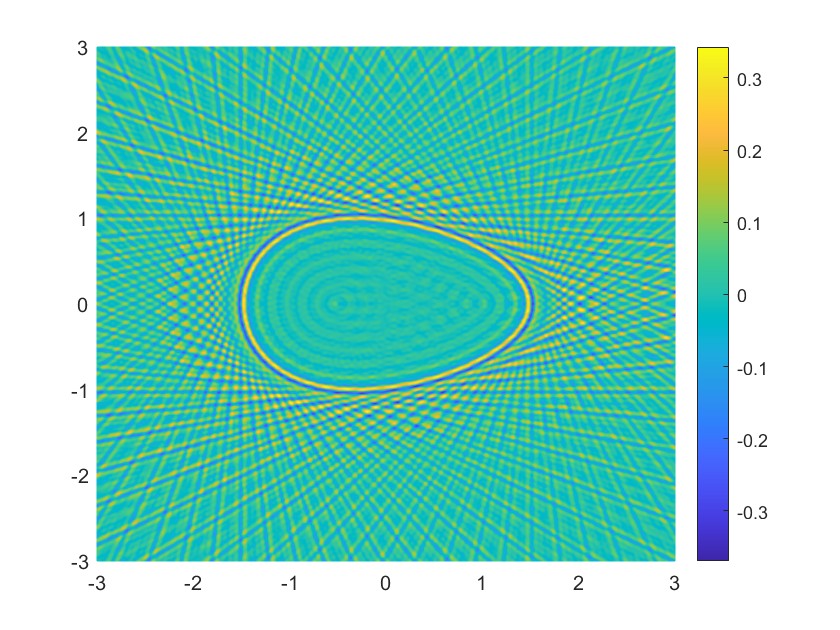}
        }\\
        \subfigure[Dirichlet case.]{
            \label{TRD-D-Egg-20-50}
            \includegraphics[width=0.2\textwidth]{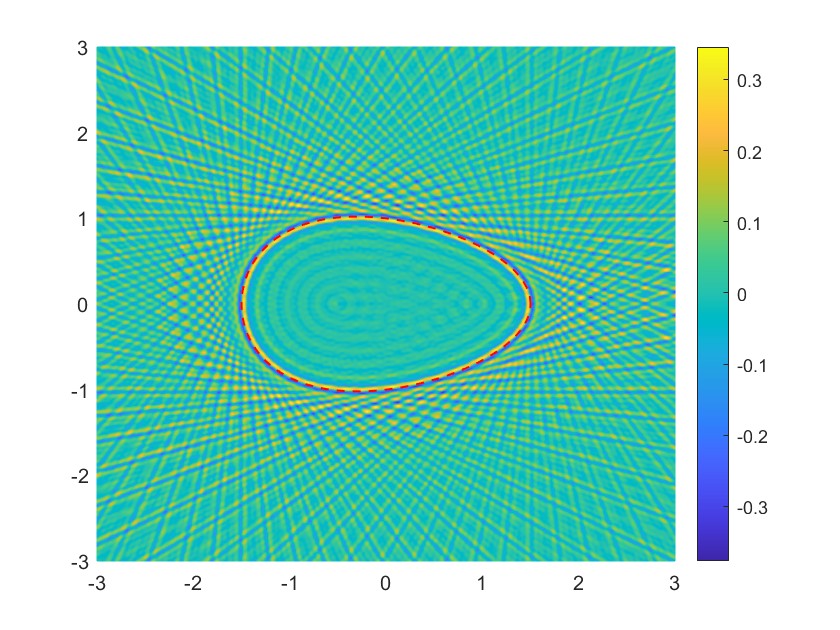}
        }&
                \subfigure[Neumann case.]{
            \label{TRD-N-Egg-20-50}
        \includegraphics[width=0.2\textwidth]{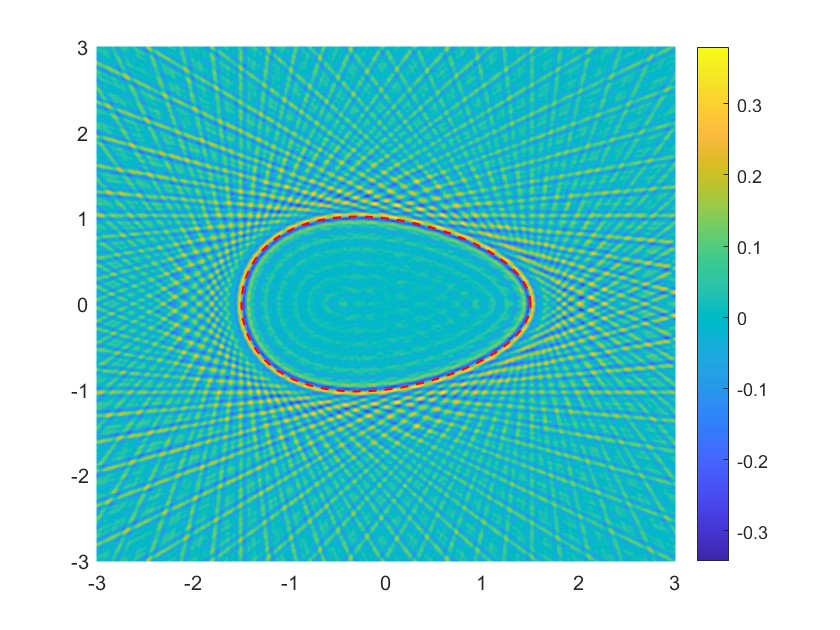}
        }&
                \subfigure[$\lambda_1$ case.]{
            \label{TRD-1S-Egg-20-50}
        \includegraphics[width=0.2\textwidth]{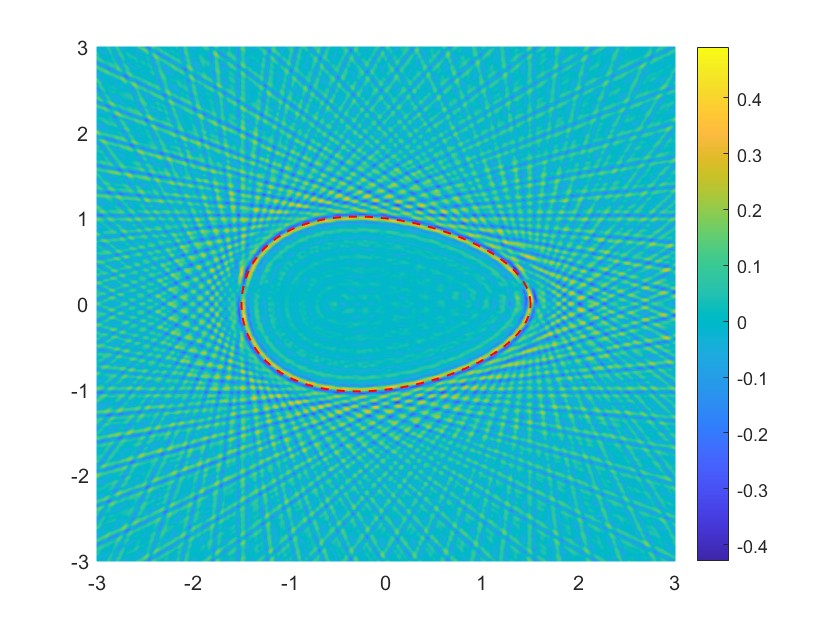}
        }&
        \subfigure[$\lambda_2$ case.]{
            \label{TRD-2SS-Egg-20-50}
            \includegraphics[width=0.2\textwidth]{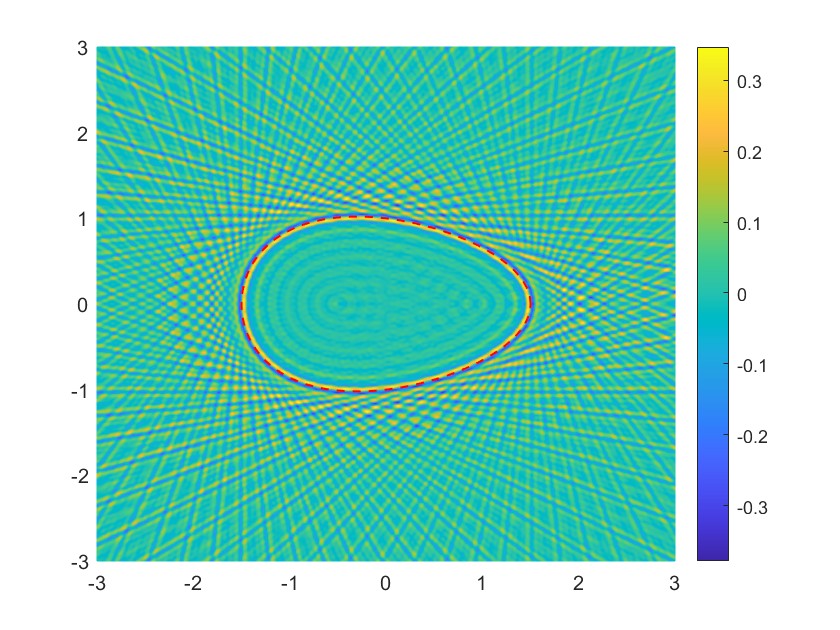}
        }
    \end{tabular}
    \caption{Reconstructions of $\partial D$ with $\mathcal{I}(z)$ in different case using frequency band $[20, 50]$. Top: Plot of $\mathcal{I}(z)$. Bottom: Comparison with true $\partial D$ (red dotted curve).}
    \label{RD-Egg-20-50}
\end{figure}
Figure \ref{RD-Egg-20-50} shows the the effectiveness of the indicator function $\mathcal{I}(z)$. Owing to the high-frequency data, these images seem somewhat not obvious. Using more low frequency data can improve this drawback for some examples. Although it is not compatible with the previous approximate theory, it is useful in practice. The details are shown in Figure \ref{RD-Egg-5-30}. Note that in the $\lambda_1$ case, the indicator function $\mathcal{I}(z)$ obtained significantly larger values at two points than elsewhere, because $\lambda_1\approx1$ near these points and $\lambda=1$ have indeed caused us trouble.
\begin{figure}[htbp]
   \centering
    \begin{tabular}{cccc}
       \subfigure[Dirichlet case.]{
            \label{RD-D-Egg-5-30}
      \includegraphics[width=0.2\textwidth]{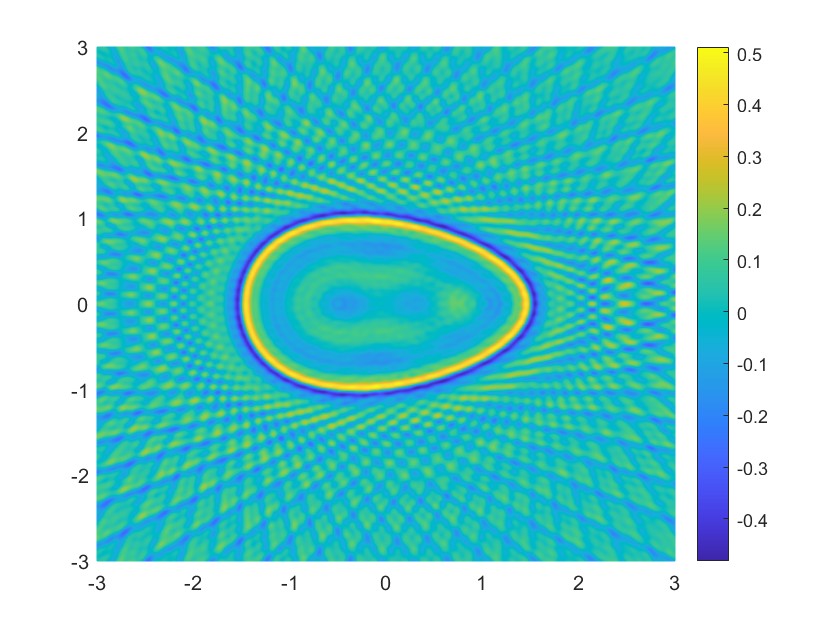}
        }&
                \subfigure[Neumann case.]{
           \label{RD-N-Egg-5-30}
        \includegraphics[width=0.2\textwidth]{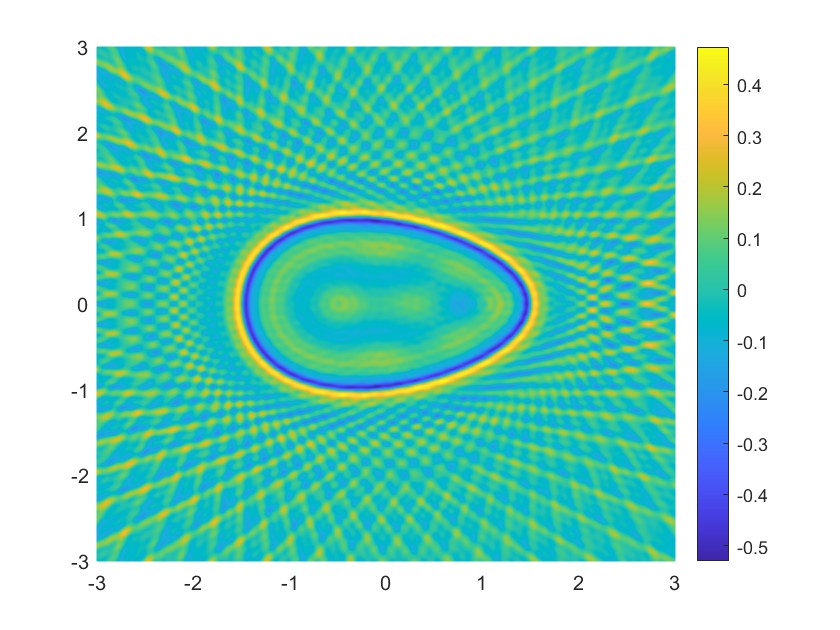}
        }&
               \subfigure[$\lambda_1$ case.]{
            \label{RD-1S-Egg-5-30}
       \includegraphics[width=0.2\textwidth]{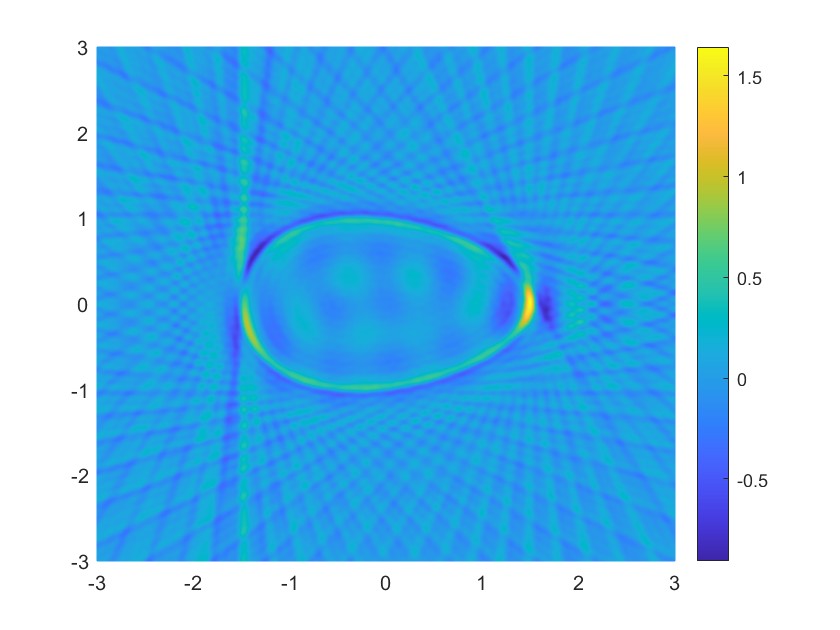}
       }&
       \subfigure[$\lambda_2$ case.]{
          \label{RD-2SS-Egg-5-30}
            \includegraphics[width=0.2\textwidth]{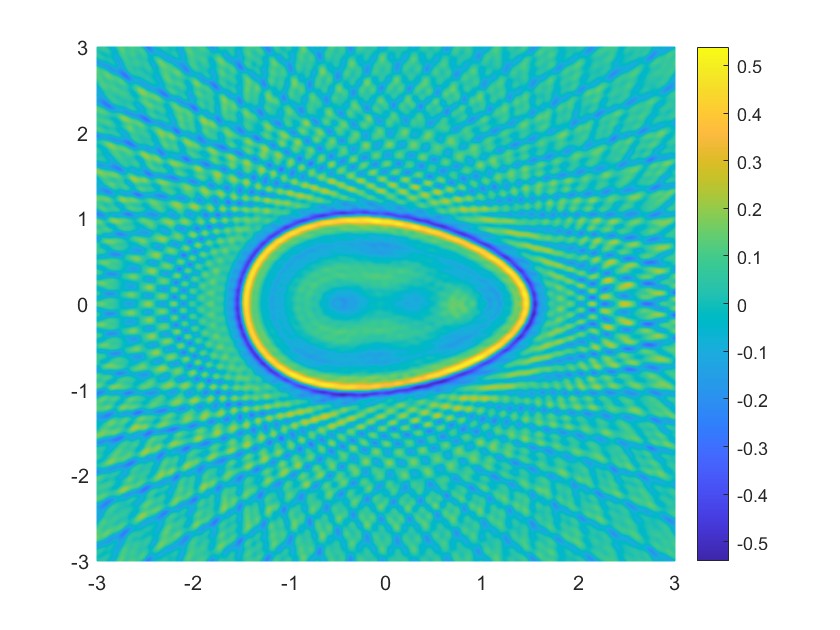}
        }\\
        \subfigure[Dirichlet case.]{
            \label{TRD-D-Egg-5-30}
            \includegraphics[width=0.2\textwidth]{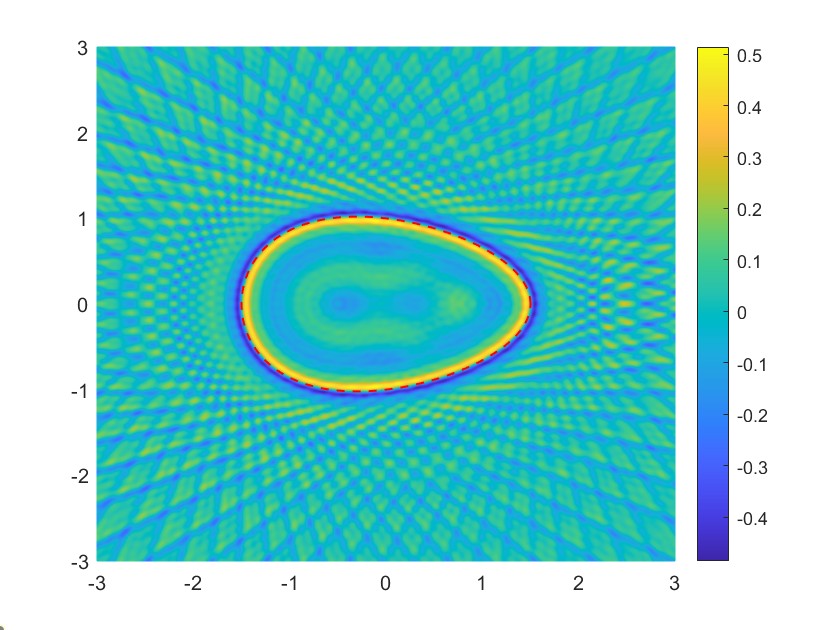}
        }&
                \subfigure[Neumann case.]{
            \label{TRD-N-Egg-5-30}
        \includegraphics[width=0.2\textwidth]{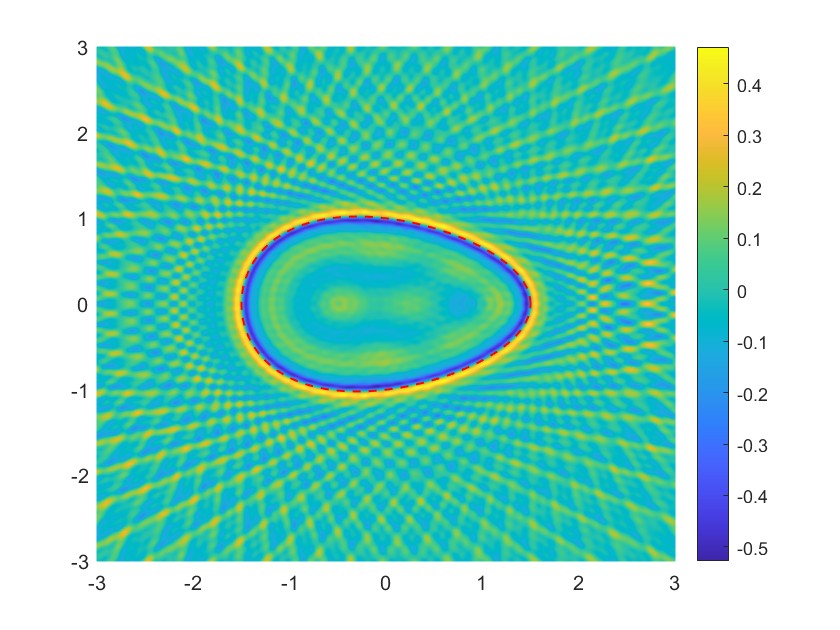}
        }&
                \subfigure[$\lambda_1$ case.]{
            \label{TRD-1S-Egg-5-30}
        \includegraphics[width=0.2\textwidth]{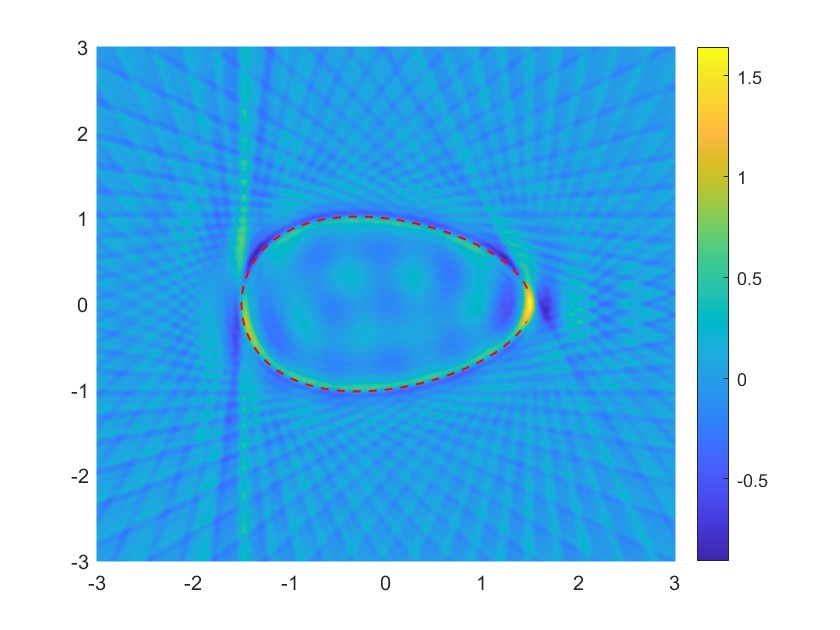}
        }&
        \subfigure[$\lambda_2$ case.]{
            \label{TRD-2SS-Egg-5-30}
            \includegraphics[width=0.2\textwidth]{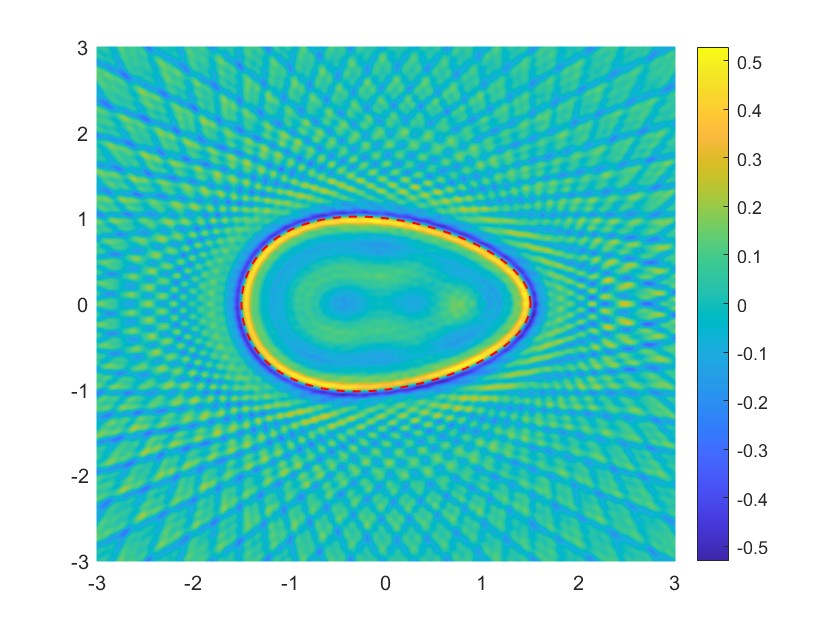}
        }
    \end{tabular}
    \caption{Reconstructions of $\partial D$ with $\mathcal{I}(z)$ in different case using frequency band $[5, 30]$. Top: Plot of $\mathcal{I}(z)$. Bottom: Comparison with true $\partial D$(red dotted curve).}
    \label{RD-Egg-5-30}
\end{figure}

Meanwhile, Figure \ref{RD-Egg-20-50} and Figure \ref{RD-Egg-5-30} demonstrate the effectiveness of the judgment criteria \eqref{I+I-DN}. Furthermore, from the Gibbs phenomenon, ignoring the approximation error of the scattering field and choosing a smaller $k_-$ and a larger $k^+$ can lead to more accurate judgment. Figure \ref{RD-D-Egg-1-30} and \ref{RD-N-Egg-1-30} confirm the validity of our inference. Figure \ref{RD-1S-Egg-1-30} and \ref{RD-2SS-Egg-1-30} use the reconstructed $\lambda$ from Figure \ref{lambda-egg}, which implies that if the reconstructed impedance coefficient $\lambda$ is not approximately $1$, we can also use low-frequency data to obtain a good reconstruction of $D$ or $\partial D$.
\begin{figure}[htbp]
   \centering
    \begin{tabular}{cccc}
        \subfigure[Dirichlet case.]{
            \label{RD-D-Egg-1-30}
        \includegraphics[width=0.2\textwidth]{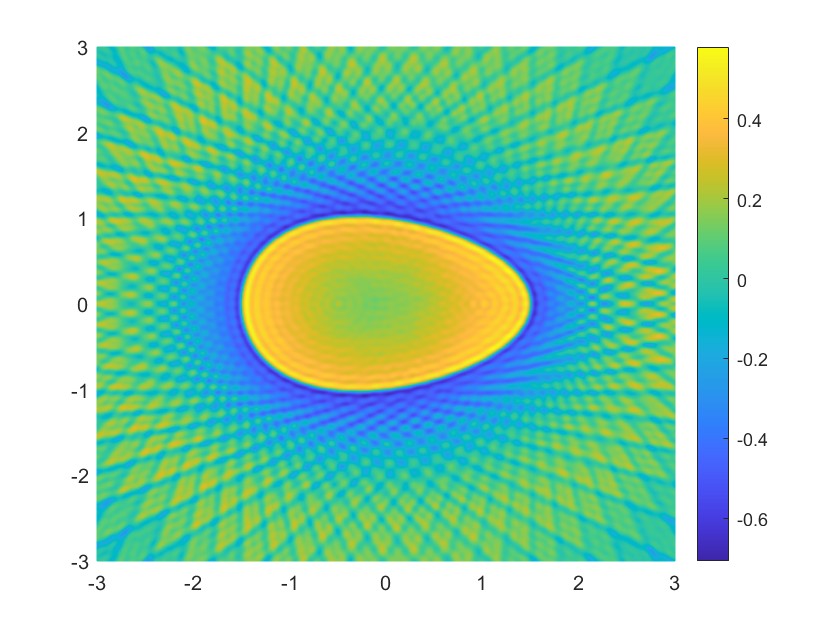}
        }&
        \subfigure[Neumann case.]{
            \label{RD-N-Egg-1-30}
            \includegraphics[width=0.2\textwidth]{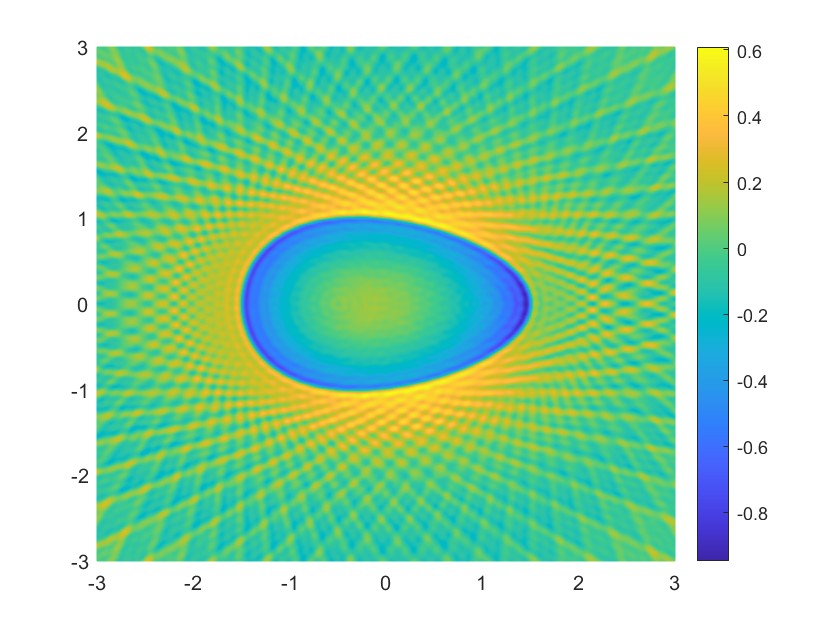}
        }&
        \subfigure[$\lambda_1$ case.]{
            \label{RD-1S-Egg-1-30}
            \includegraphics[width=0.2\textwidth]{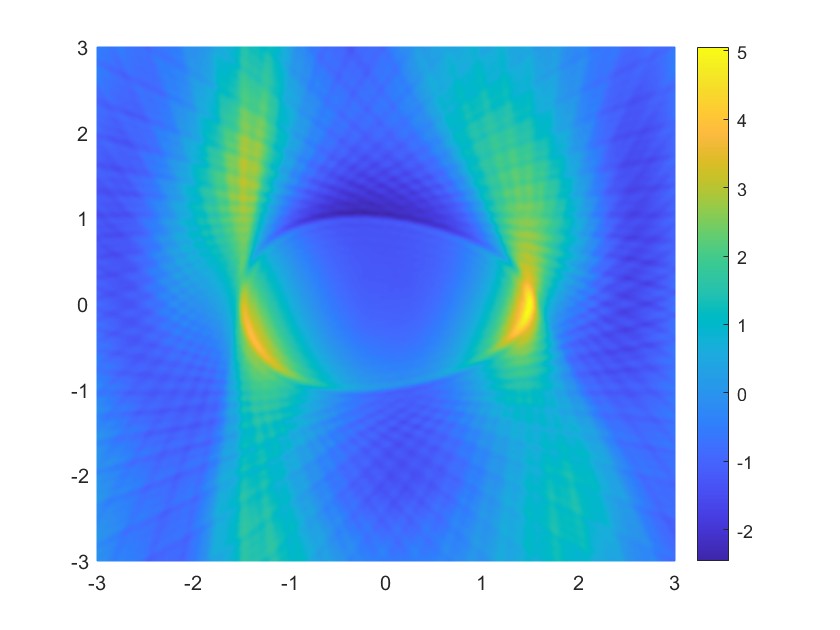}
        }&
        \subfigure[$\lambda_2$ case.]{
            \label{RD-2SS-Egg-1-30}
            \includegraphics[width=0.2\textwidth]{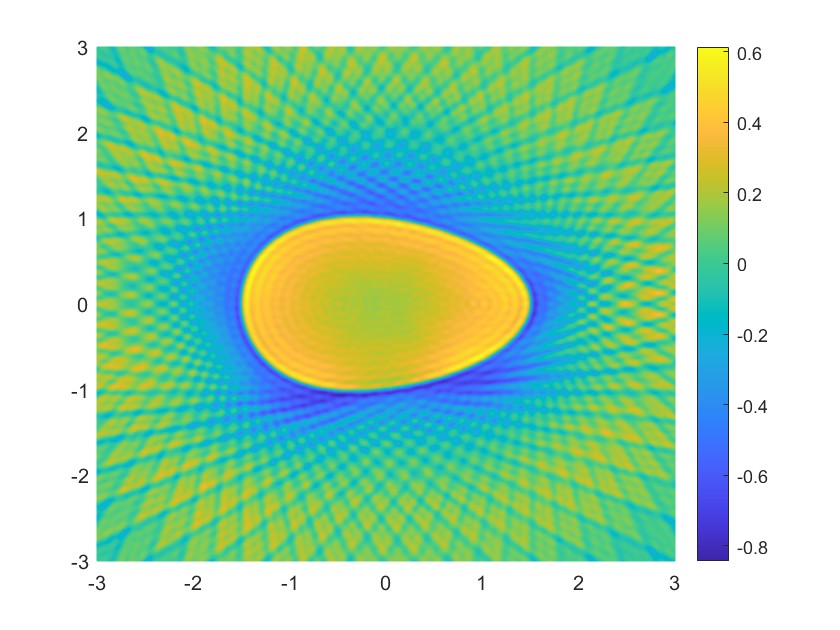}
        }
    \end{tabular}
    \caption{Reconstructions of $\partial D$ with $\mathcal{I}(z)$ in different case using frequency band $[1, 30]$.}
    \label{RD-Egg-1-30}
\end{figure} 

So far, we have completed the reconstruction for boundary conditions and boundary $\partial D$. Assuming that the boundary is exactly reconstructed,
Figure \ref{lambda-egg-30del} shows that with greater noise, the method \eqref{sovlelambda-num-NeD} can effectively reconstruct the impedance coefficient $\lambda$. 
The superior effectiveness of method \eqref{sovlelambda-num-NeD} stems from two key advantages over \eqref{reconstruct-lambda}. First, unlike \eqref{reconstruct-lambda} which suffers from pseudo-solution interference, \eqref{sovlelambda-num-NeD} completely eliminates this computational artifact. Second, \eqref{sovlelambda-num-NeD} incorporates significantly more prior information - specifically the given boundary conditions - than \eqref{reconstruct-lambda}, leading to more robust and accurate results.
\begin{figure}[htbp]
   \centering
    \begin{tabular}{cc}
        \subfigure[Comparison for $\lambda_1$.]{
            \label{lambda1-egg-30del}
        \includegraphics[width=0.3\textwidth]{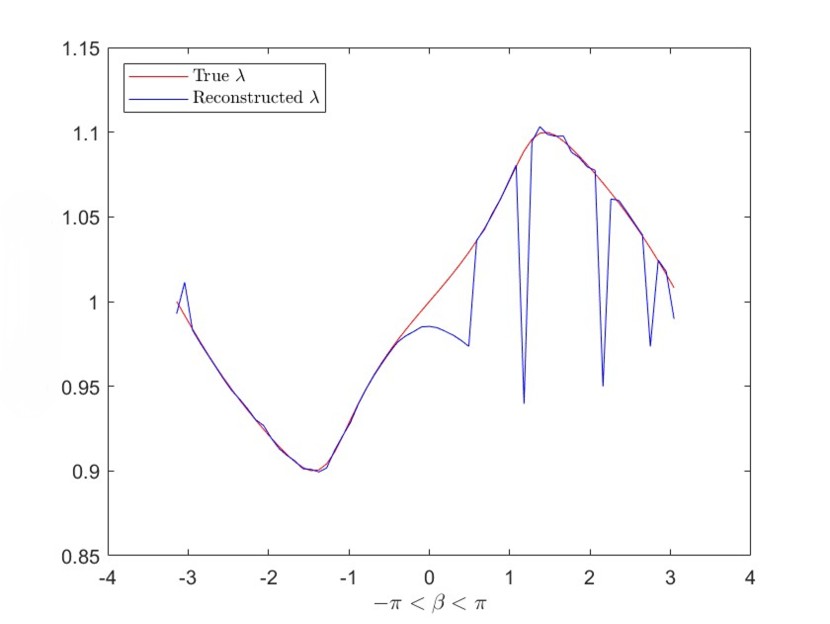}
        }&
        \subfigure[Comparison for  $\lambda_2$.]{
            \label{lambda2-egg-30del}
            \includegraphics[width=0.3\textwidth]{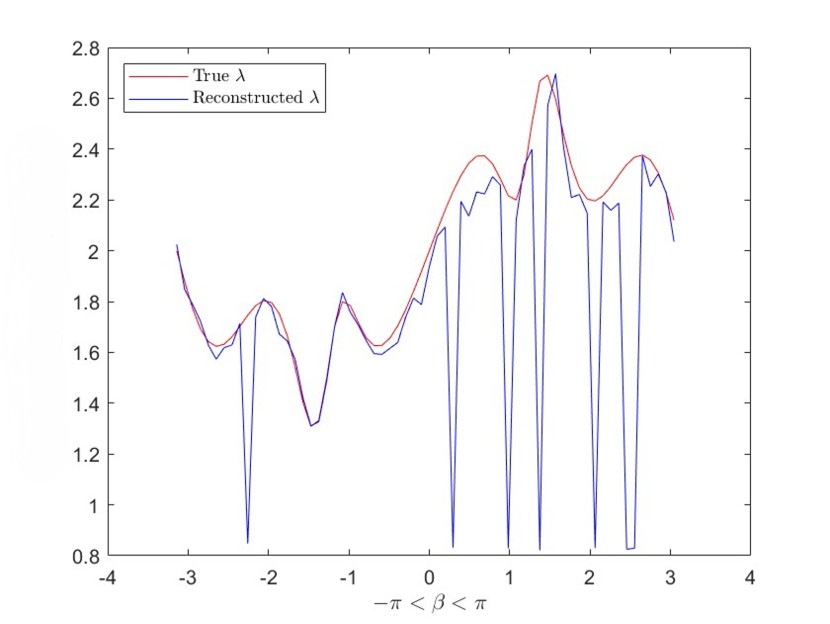}
        }\\
                \subfigure[Comparison for $\lambda_1$.]{
            \label{lambda1-egg-30del-NeD}
        \includegraphics[width=0.3\textwidth]{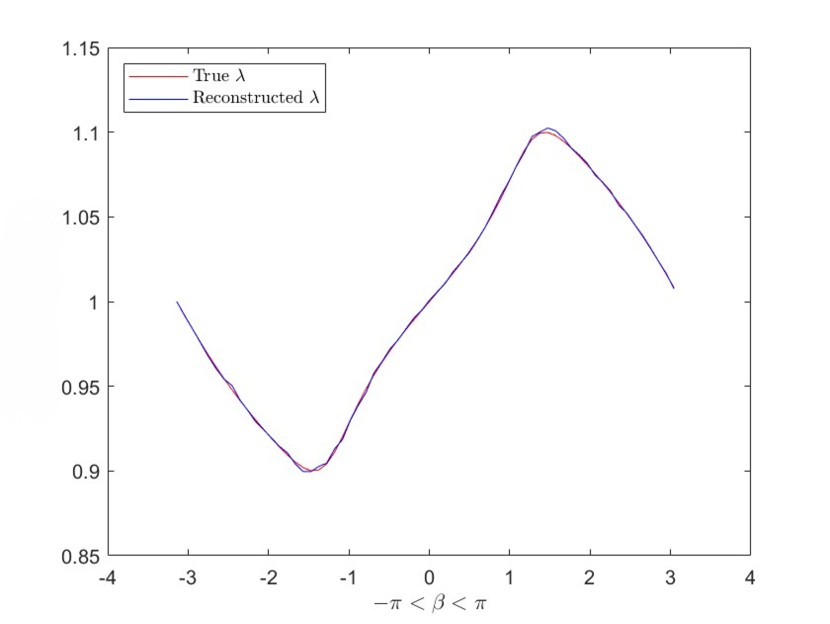}
        }&
        \subfigure[Comparison for  $\lambda_2$.]{
            \label{lambda2-egg-30del-NeD}
            \includegraphics[width=0.3\textwidth]{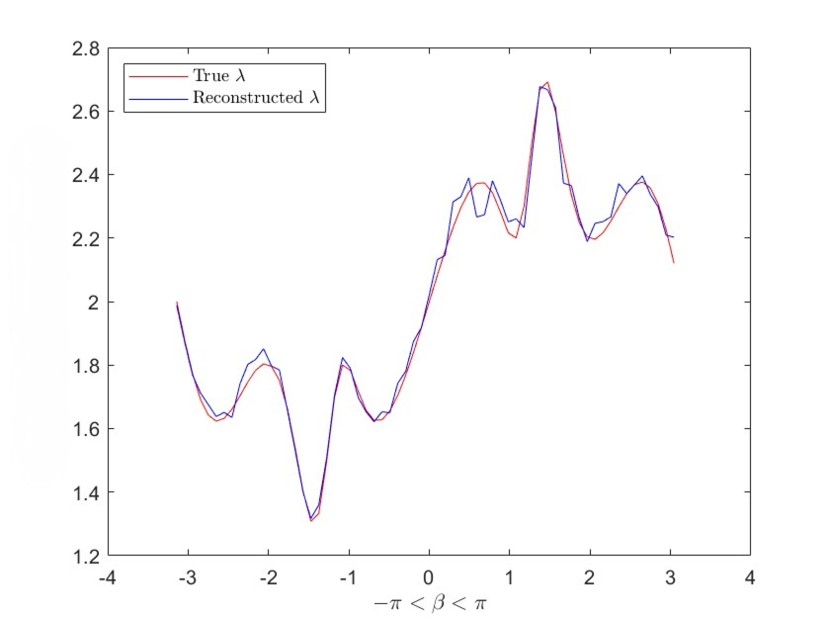}
        }
    \end{tabular}
    \caption{True and reconstructed $\lambda_1$ and $\lambda_2$ with $\delta=0.3$ using frequency band $[20, 50]$. Top: Reconstructed $\lambda$ by using \eqref{reconstruct-lambda}. Bottom: Reconstructed $\lambda$ by using \eqref{sovlelambda-num-NeD}.}
    \label{lambda-egg-30del}
\end{figure}

We end this example with Figure \ref{spa-Egg-20-50}. Figure \ref{spa-Egg-20-50} shows the reconstruction using the indicator function $\mathcal{T}$ with sparse directional data. Similar to the analysis of Proposition \ref{prop-line}, we capture these lines mentioned above. 
The reconstruction is improved with the increase of the observation direction number $SD$.
\begin{figure}[htbp]
   \centering
    \begin{tabular}{cccc}
        \subfigure[Dirichlet case.]{
            \label{spa1-D-Egg-20-50}
        \includegraphics[width=0.2\textwidth]{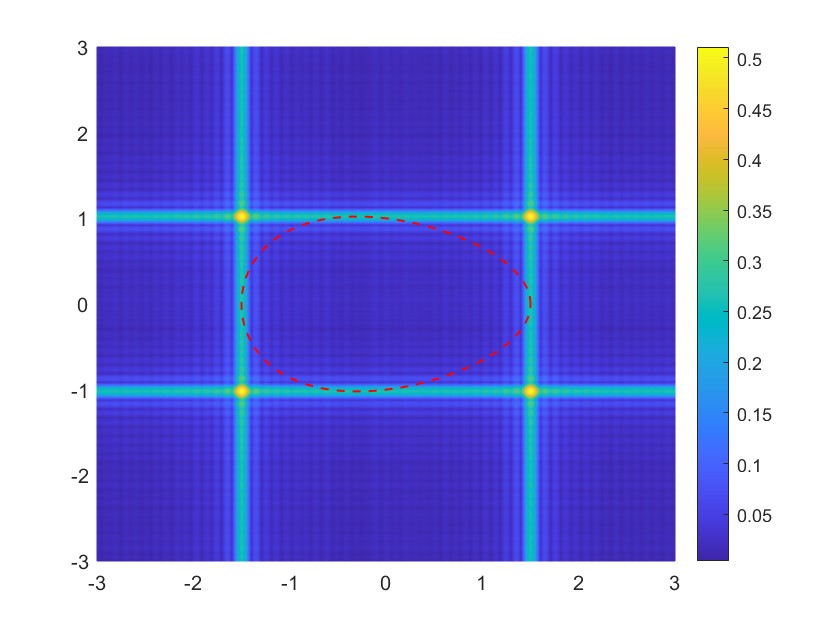}
        }&
                \subfigure[Neumann case.]{
            \label{spa1-N-Egg-20-50}
        \includegraphics[width=0.2\textwidth]{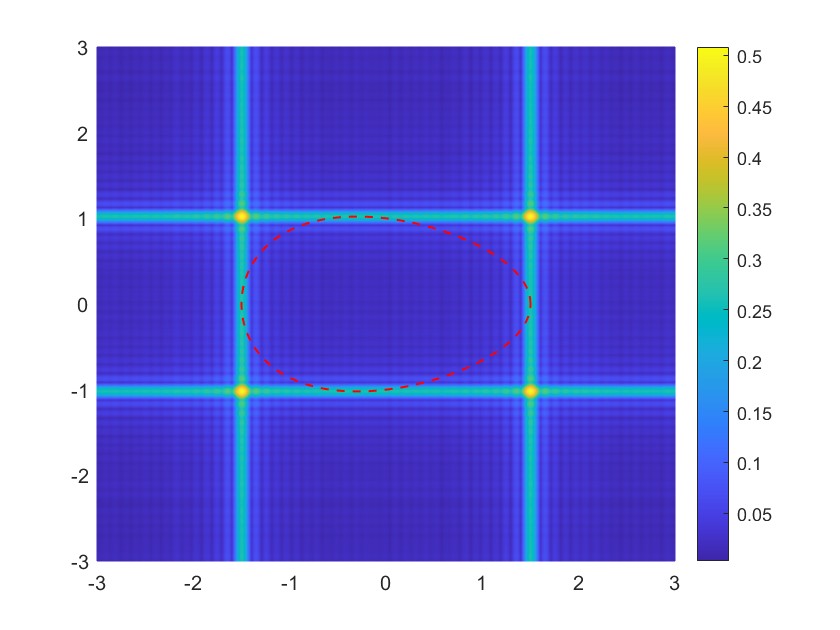}
        }&
                \subfigure[$\lambda_1$ case.]{
            \label{spa1-1S-Egg-20-50}
        \includegraphics[width=0.2\textwidth]{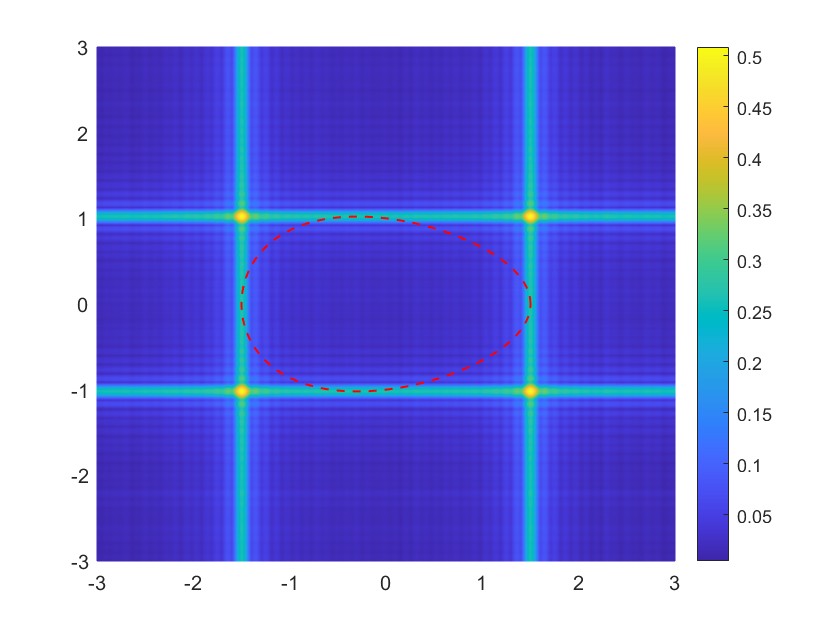}
        }&
        \subfigure[$\lambda_2$ case.]{
           \label{spa1-2SS-Egg-20-50}
            \includegraphics[width=0.2\textwidth]{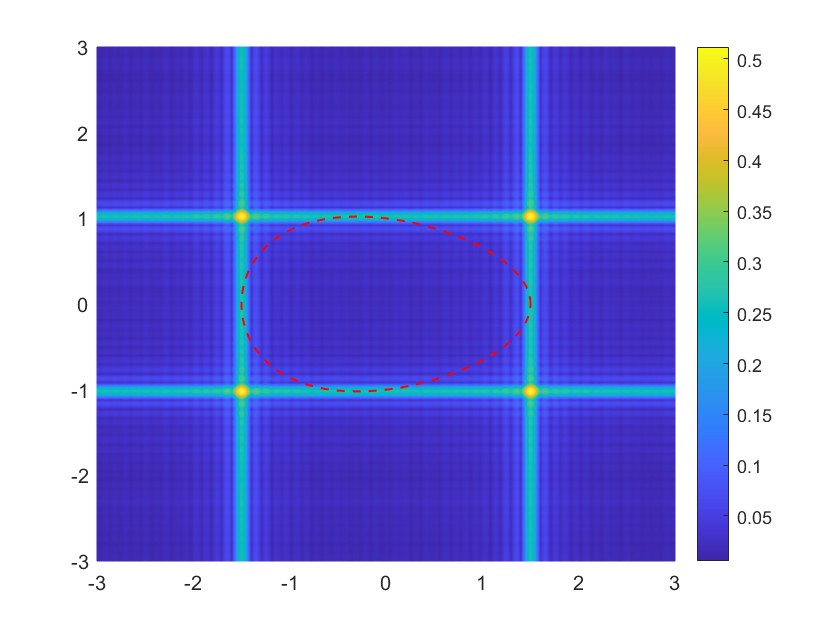}
        }\\
        \subfigure[Dirichlet case.]{
            \label{spa2-D-Egg-20-50}
            \includegraphics[width=0.2\textwidth]{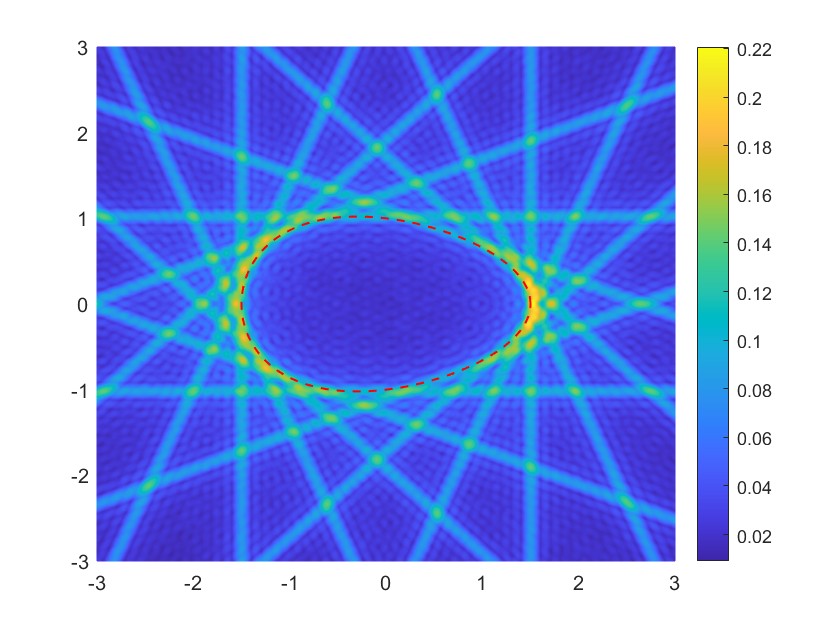}
        }&
                \subfigure[Neumann case.]{
            \label{spa2-N-Egg-20-50}
        \includegraphics[width=0.2\textwidth]{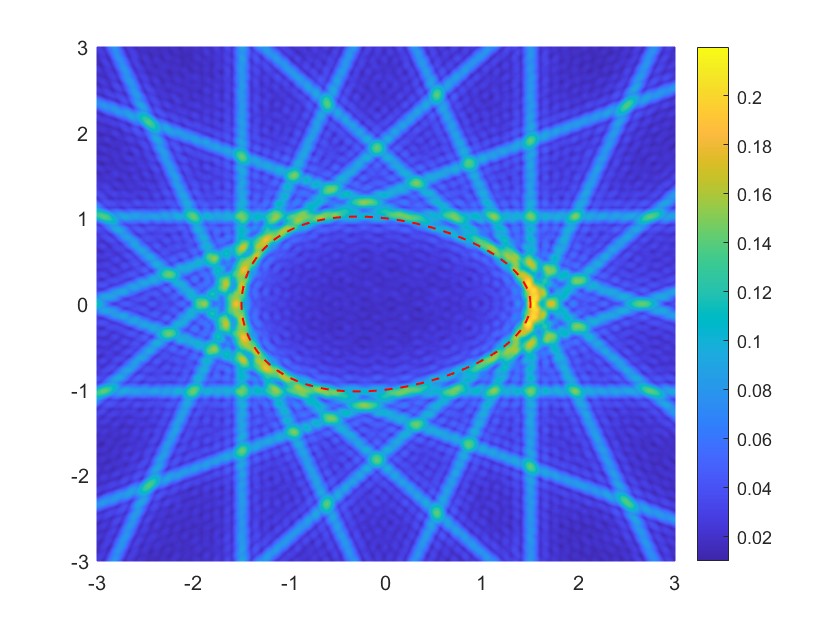}
        }&
                \subfigure[$\lambda_1$ case.]{
            \label{spa2-1S-Egg-20-50}
        \includegraphics[width=0.2\textwidth]{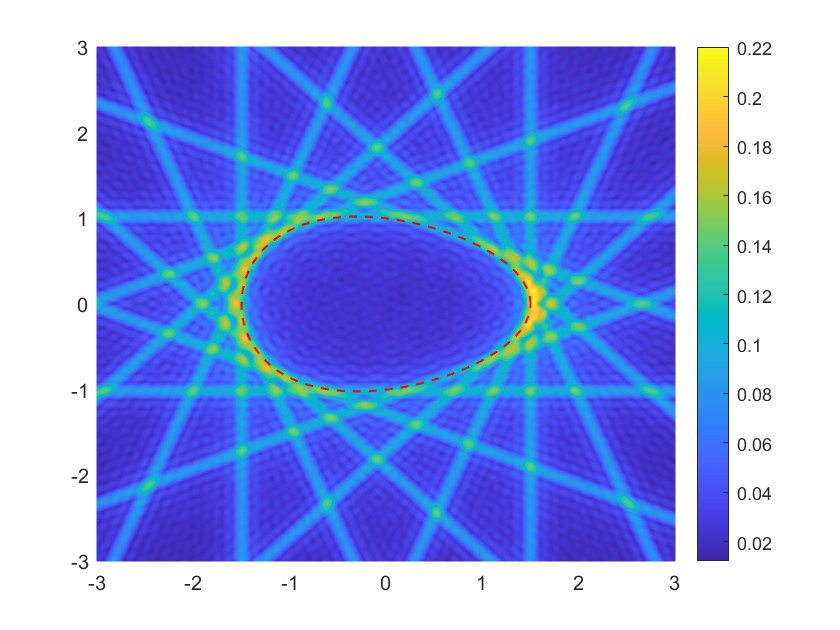}
        }&
        \subfigure[$\lambda_2$ case.]{
            \label{spa2-2SS-Egg-20-50}
            \includegraphics[width=0.2\textwidth]{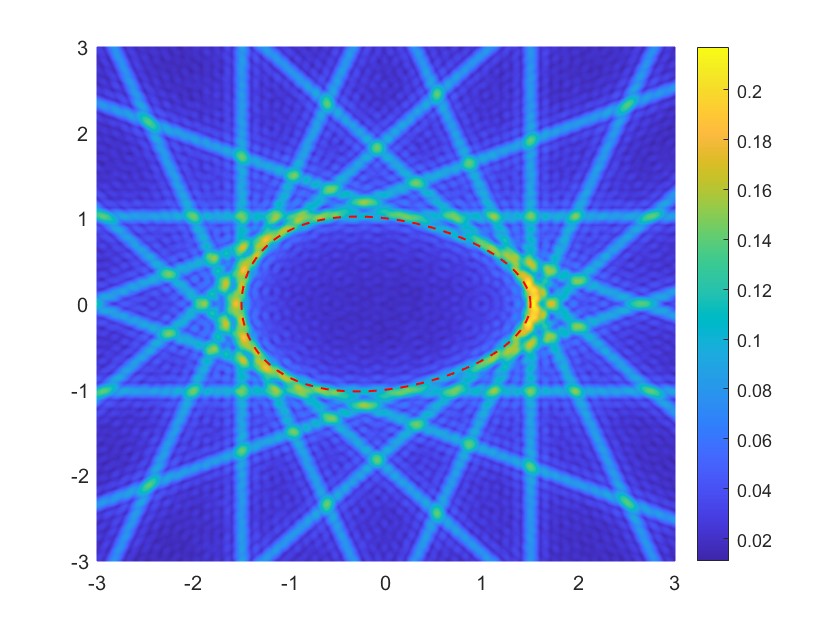}
        }\\
          \subfigure[Dirichlet case.]{
            \label{spa4-D-Egg-20-50}
            \includegraphics[width=0.2\textwidth]{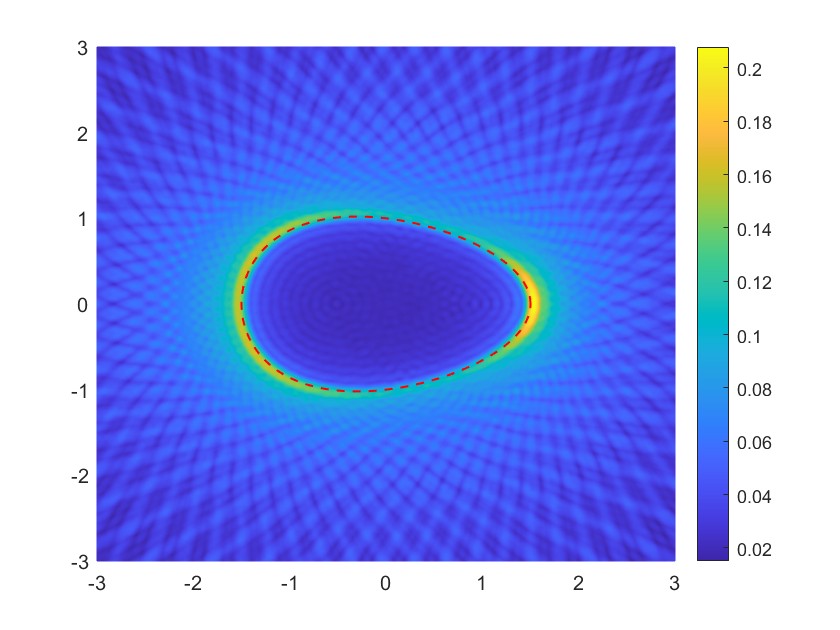}
        }&
                \subfigure[Neumann case.]{
            \label{spa4-N-Egg-20-50}
        \includegraphics[width=0.2\textwidth]{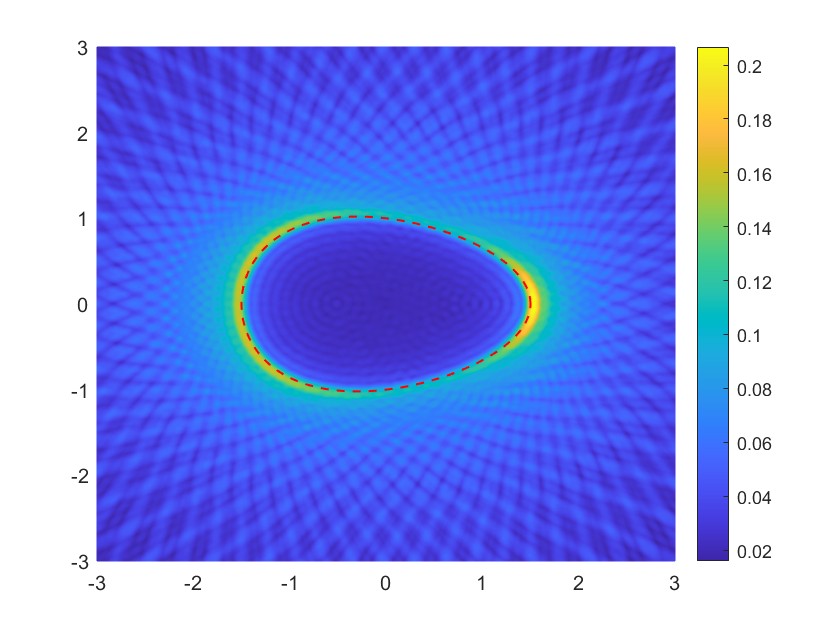}
        }&
                \subfigure[$\lambda_1$ case.]{
            \label{spa4-1S-Egg-20-50}
        \includegraphics[width=0.2\textwidth]{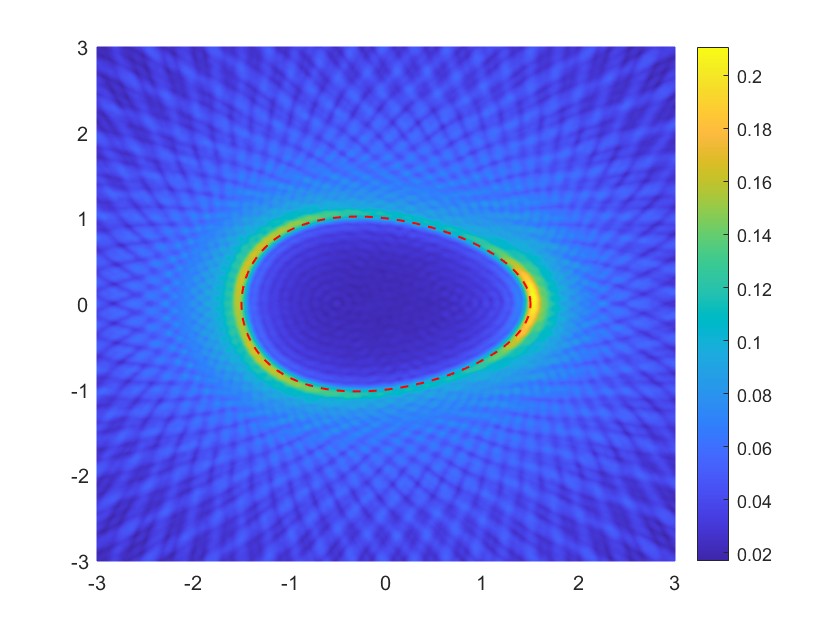}
        }&
        \subfigure[$\lambda_2$ case.]{
            \label{spa4-2SS-Egg-20-50}
            \includegraphics[width=0.2\textwidth]{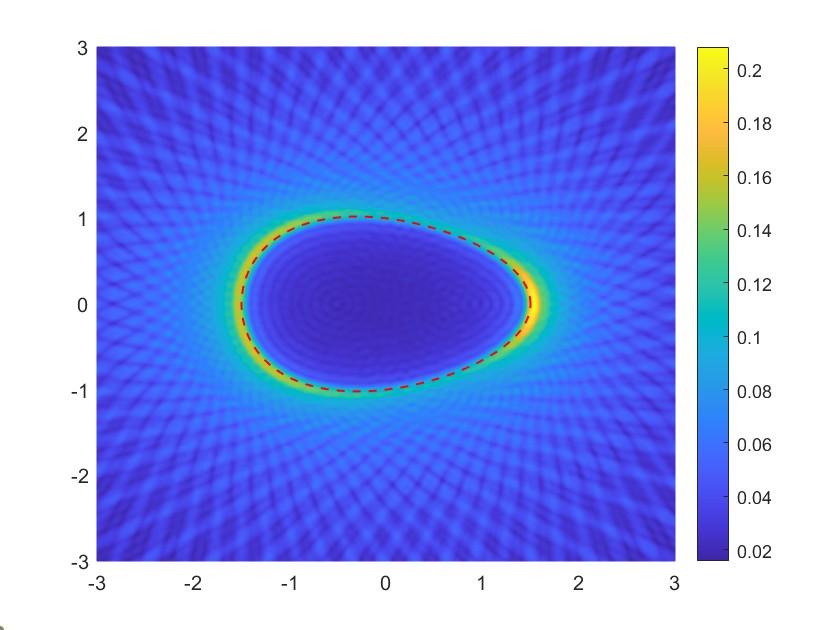}
        }
    \end{tabular}
    \caption{Reconstructions with $\mathcal{T}(z)$ in different cases  using $SD$ equally distributed directions data and  frequency band $[20, 50]$, the true boundary $\partial D$ is plotted by a red dotted curve. Top:  $SD=4$. Middle: $SD=16$. Bottom: $SD=64$.}
    \label{spa-Egg-20-50}
\end{figure}
\subsection{Reconstructions for the Kite Example}
 For nonconvex obstacles, owing to the invalidity of the Theorem \ref{thm-majda} for concave obstacles, we sometimes fail to determine whether their boundary conditions are Dirichlet or Neumann boundary conditions. Fortunately, changing the criteria \eqref{judgement-num-DorN} to the following criteria can effectively solve this problem,
 \begin{align}\label{judgement-num-DorN-concave}
     \min\limits_{\theta\in\Theta_l}|\mathcal{L}(\hx,\alpha_j)-1|<\frac{\delta}{2}=0.05.
 \end{align}
Figure \ref{jugement-Kite-20-50} verifies the validity of the new criteria \eqref{judgement-num-DorN-concave}.

 \begin{figure}[htbp]
   \centering
    \begin{tabular}{cccc}
        \subfigure[Dirichlet case.]{
            \label{jugement-D-Kite-20-50}
        \includegraphics[width=0.2\textwidth]{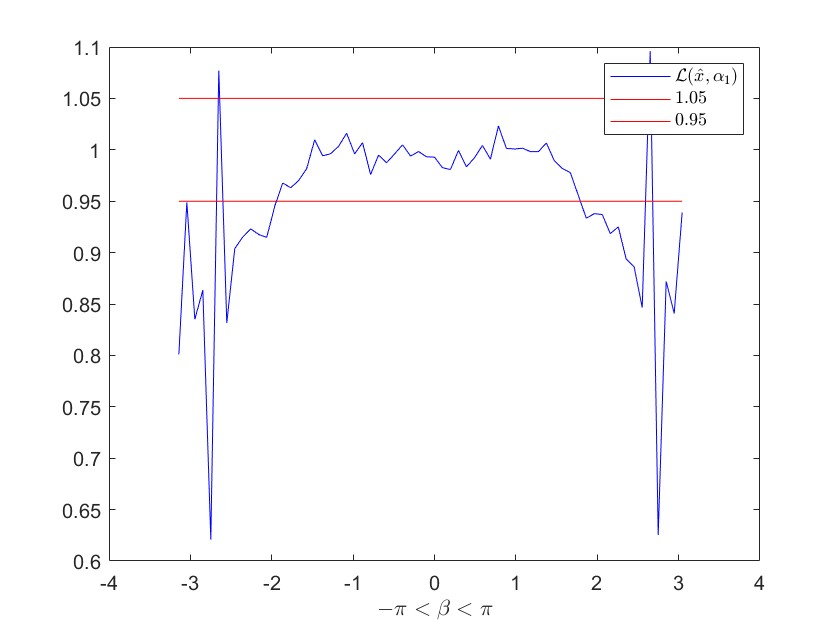}
        }&
        \subfigure[Neumann case.]{
            \label{jugement-N-Kite-20-50}
            \includegraphics[width=0.2\textwidth]{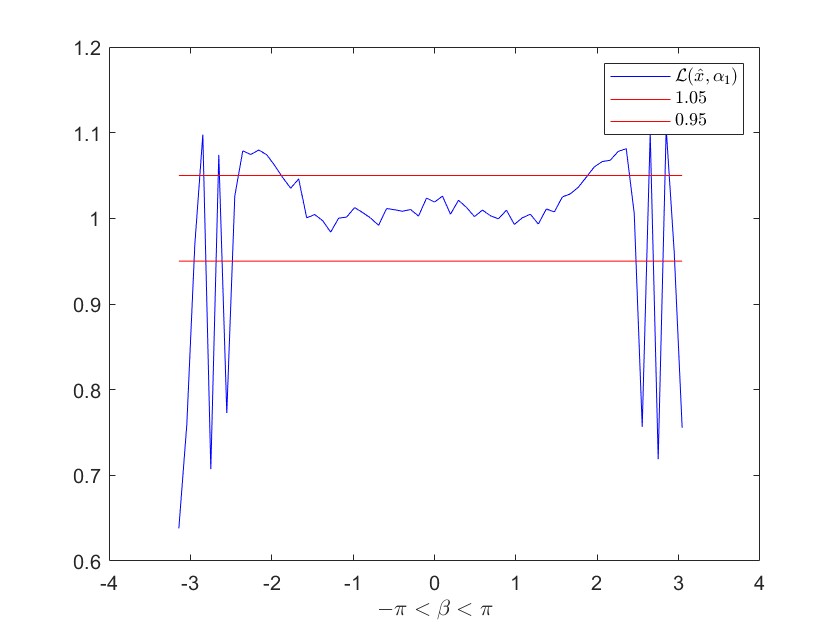}
        }&
        \subfigure[$\lambda_1$ case.]{
            \label{jugement-1S-Kite-20-50}
            \includegraphics[width=0.2\textwidth]{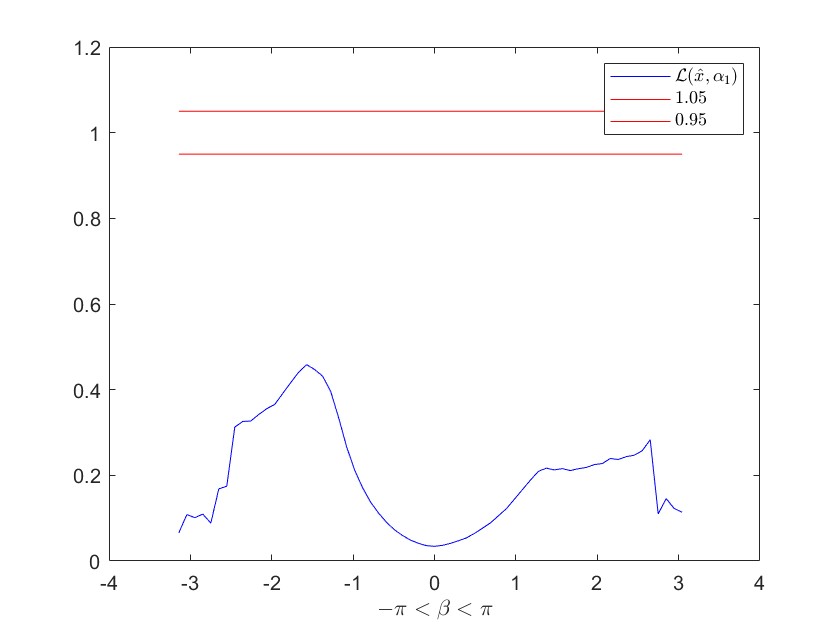}
        }&
        \subfigure[$\lambda_2$ case.]{
            \label{jugement-2SS-Kiteg-20-50}
            \includegraphics[width=0.2\textwidth]{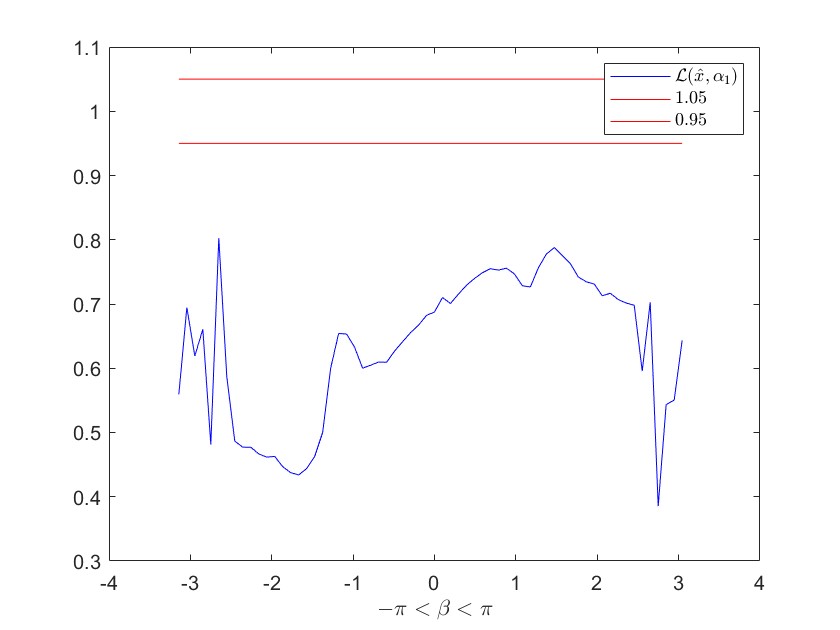}
        }
    \end{tabular}
    \caption{The values of $\mathcal{L}(\hx, \alpha_1)$ for four different boundary conditions with frequency band $[20, 50]$.}
    \label{jugement-Kite-20-50}
\end{figure} 

Figure \ref{lambda-kite}  shows the reconstruction of $\lambda(\mathcal{G}(-\hx,\hx))=\lambda(\beta),\ -\pi<\beta<\pi$ using \eqref{reconstruct-lambda}. Note that mapping $\mathcal{G}(\beta)$ is not injective with respect to $\beta$. Denote by $\Lambda_D\subset(-\pi,\pi)$ the component such that $\mathcal{G}$ is injective. We observe that the impedance coefficient $\lambda$ can be well reconstructed only in some subset of $\Lambda_D$.
\begin{figure}[htbp]
   \centering
    \begin{tabular}{cc}
        \subfigure[Comparison for $\lambda_1$.]{
            \label{lambda1-kite}
        \includegraphics[width=0.3\textwidth]{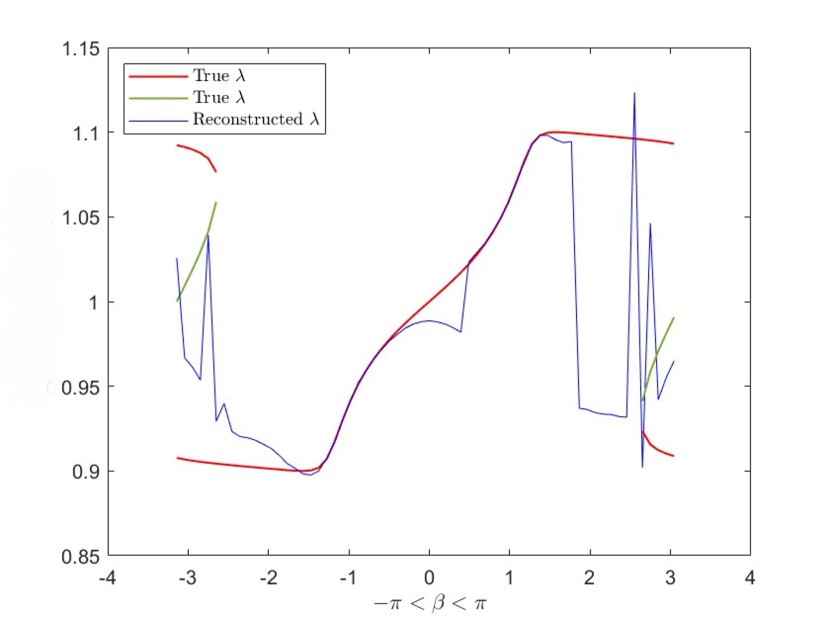}
        }&
        \subfigure[Comparison for  $\lambda_2$.]{
            \label{lambda2-kite}
            \includegraphics[width=0.3\textwidth]{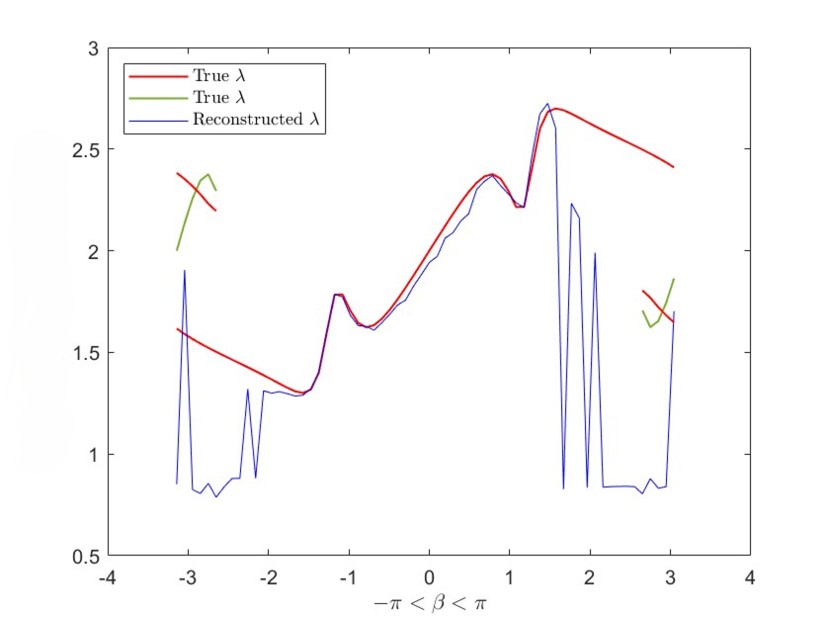}
        }
    \end{tabular}
    \caption{True and reconstructed $\lambda_1$ and $\lambda_2$ using \eqref{reconstruct-lambda} with frequency band $[20, 50]$.}
    \label{lambda-kite}
\end{figure}

Figures \ref{RD-Kite-20-50} and \ref{RD-Kite-5-30} illustrate that the indicator function $\mathcal{I}(z)$ and the criteria \eqref{I+I-DN} can be used directly in concave situations. Surprisingly, $\mathcal{I}(z)$ performs very well in the Dirichlet and Neumann situations.  Considering the poor reconstruction of $\lambda$ on $(-\pi,\pi)\backslash\Lambda_D$, such reconstruction of $\partial D$ for $\lambda_1$ and $\lambda_2$ case is acceptable. An interesting phenomenon worth discussing here is that for those $\lambda$ near the concave structure, our algorithm always obtains the wrong solution in \eqref{reconstruct-lambda}, note that for larger $\lambda>1$, $\lambda_{\alpha_j,\hx}^{(2)}\approx\sqrt{\hx\cdot\hx_j}$ is always the wrong solution. Thus, the wrong calculation $\gamma(\hx)^{-1}$ becomes large. It implies that the indicator function $\mathcal{I}(z)$ can properly reconstruct concave structures if the concave structure has a larger impedance coefficient $\lambda$, as shown in Figure \ref{RD-2SS-Kite-5-30} and \ref{TRD-2SS-Kite-5-30}.
\begin{figure}[htbp]
   \centering
    \begin{tabular}{cccc}
        \subfigure[Dirichlet case.]{
            \label{RD-D-Kite-20-50}
        \includegraphics[width=0.2\textwidth]{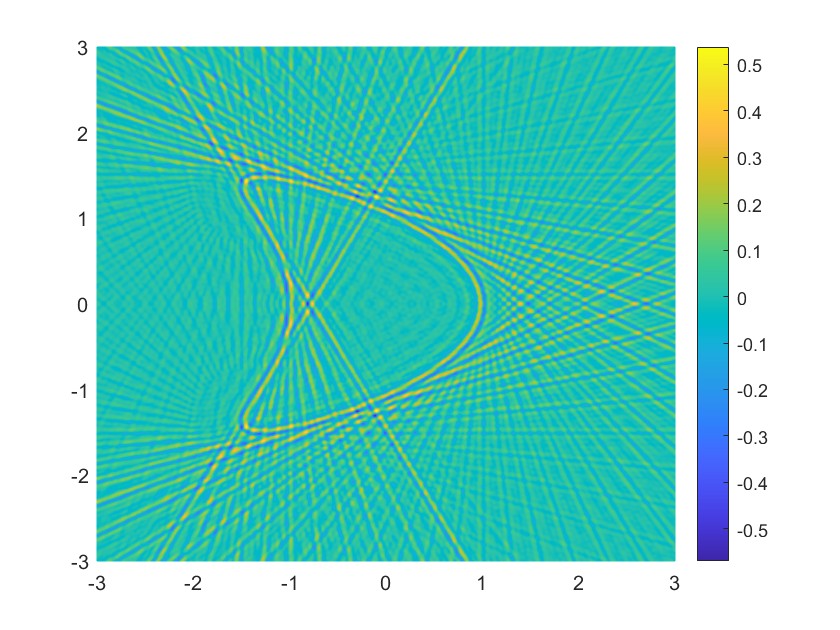}
        }&
                \subfigure[Neumann case.]{
            \label{RD-N-Kite-20-50}
        \includegraphics[width=0.2\textwidth]{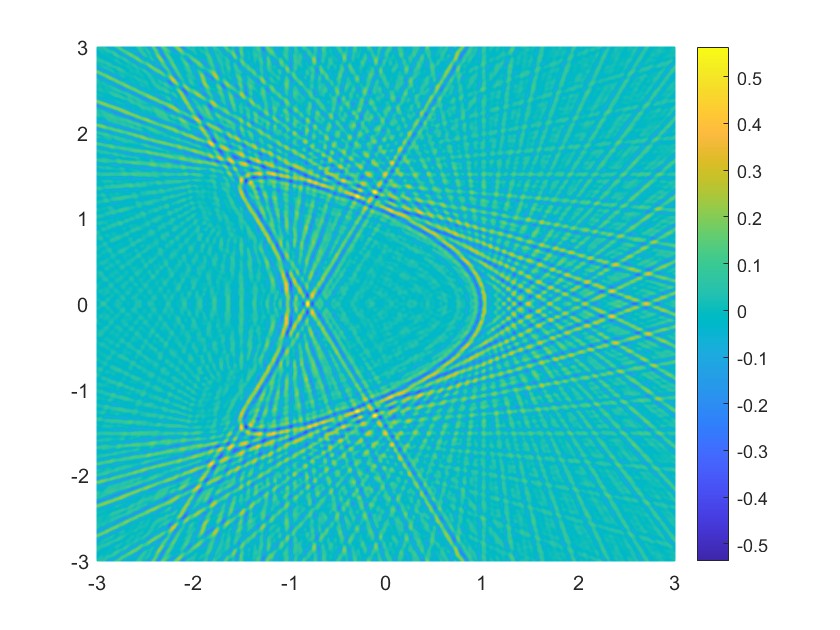}
        }&
                \subfigure[$\lambda_1$ case.]{
            \label{RD-1S-Kite-20-50}
        \includegraphics[width=0.2\textwidth]{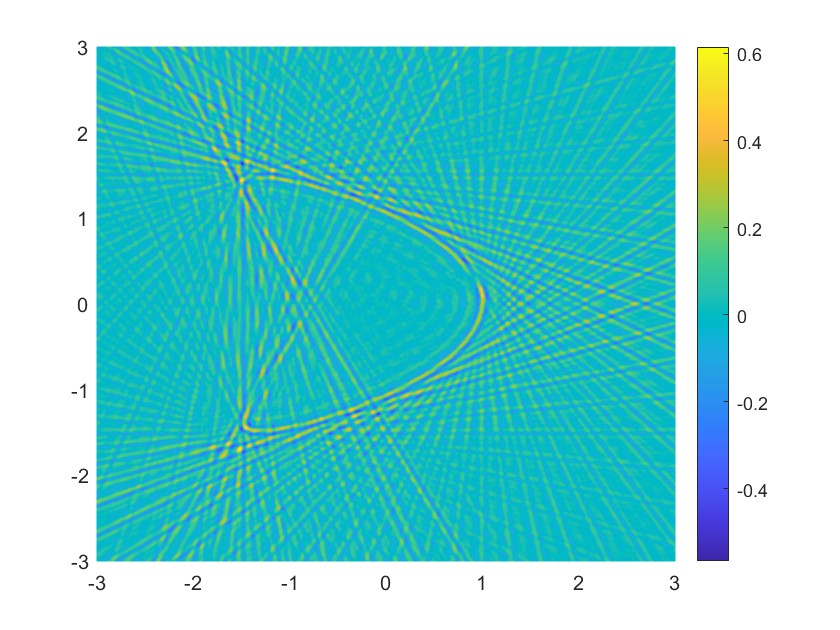}
        }&
        \subfigure[$\lambda_2$ case.]{
           \label{RD-2SS-Kite-20-50}
            \includegraphics[width=0.2\textwidth]{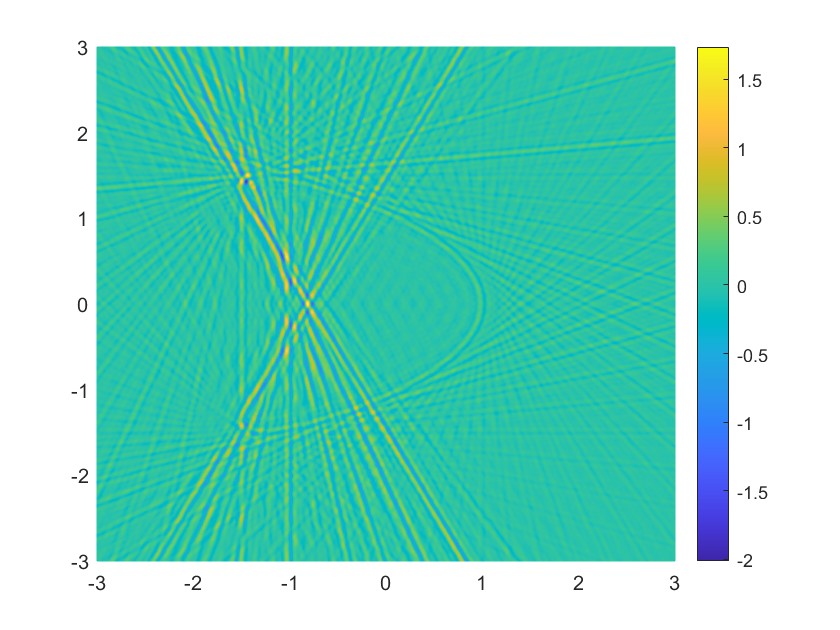}
        }\\
        \subfigure[Dirichlet case.]{
            \label{TRD-D-Kite-20-50}
            \includegraphics[width=0.2\textwidth]{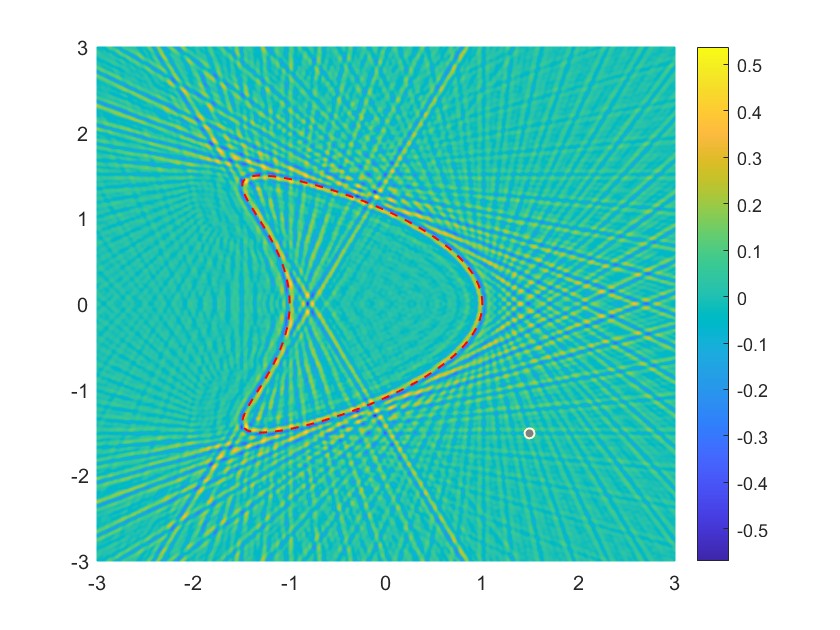}
        }&
                \subfigure[Neumann case.]{
            \label{TRD-N-Kite-20-50}
        \includegraphics[width=0.2\textwidth]{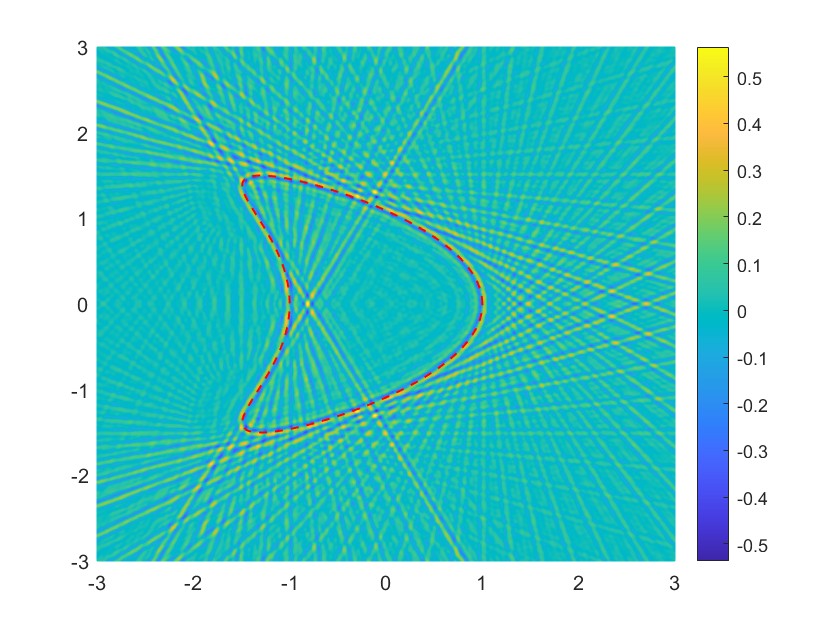}
        }&
                \subfigure[$\lambda_1$ case.]{
            \label{TRD-1S-Kite-20-50}
        \includegraphics[width=0.2\textwidth]{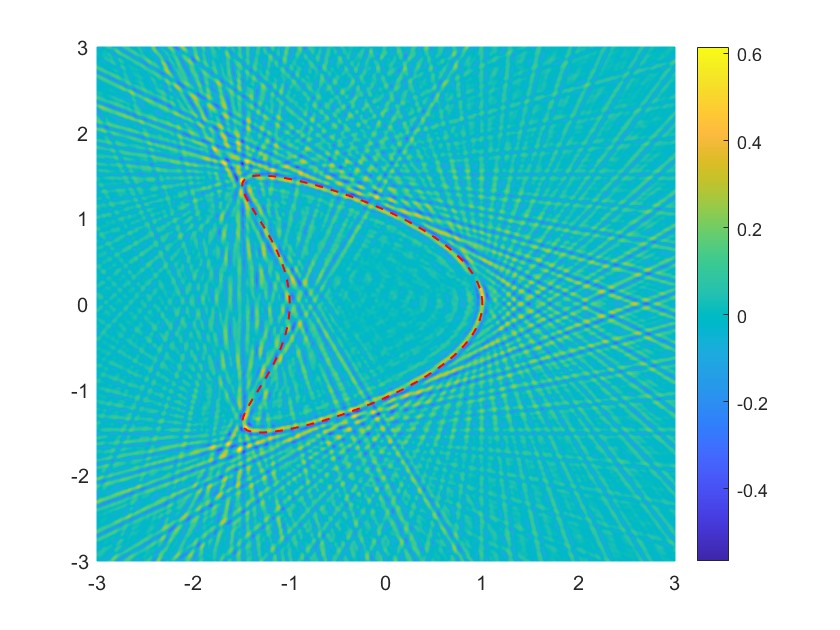}
        }&
        \subfigure[$\lambda_2$ case.]{
            \label{TRD-2SS-Kite-20-50}
            \includegraphics[width=0.2\textwidth]{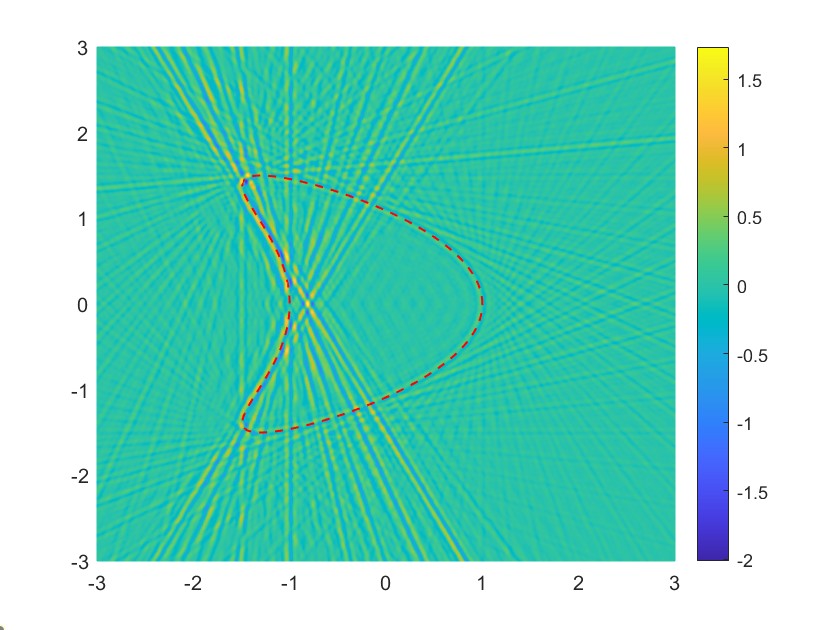}
        }
    \end{tabular}
    \caption{Reconstructions of $\partial D$ with $\mathcal{I}(z)$ in different case using frequency band $[20, 50]$. Top: Plot of $\mathcal{I}(z)$. Bottom: Comparison with true $\partial D$(red dotted curve).}
    \label{RD-Kite-20-50}
\end{figure}
\begin{figure}[htbp]
   \centering
    \begin{tabular}{cccc}
        \subfigure[Dirichlet case.]{
            \label{RD-D-Kite-5-30}
        \includegraphics[width=0.2\textwidth]{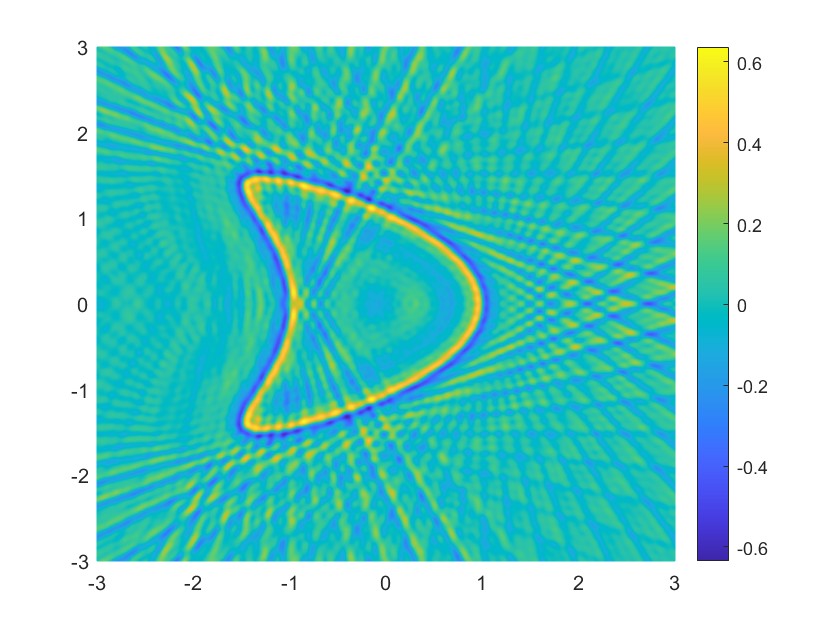}
        }&
                \subfigure[Neumann case.]{
            \label{RD-N-Kite-5-30}
        \includegraphics[width=0.2\textwidth]{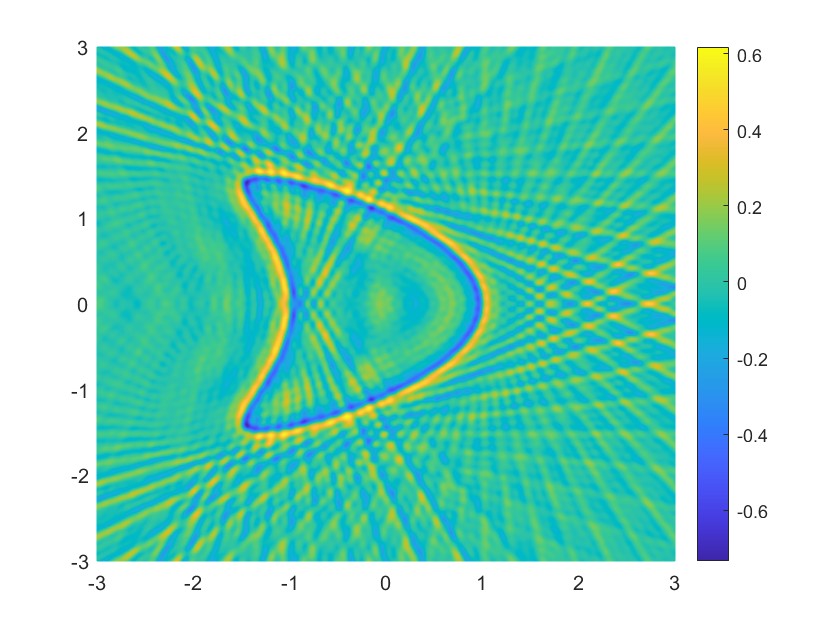}
        }&
                \subfigure[$\lambda_1$ case.]{
            \label{RD-1S-Kite-5-30}
        \includegraphics[width=0.2\textwidth]{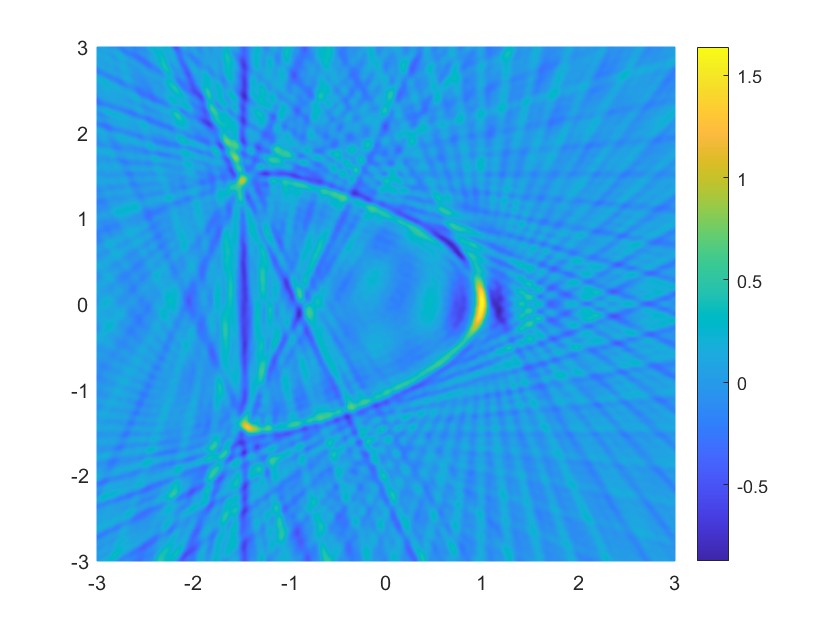}
        }&
        \subfigure[$\lambda_2$ case.]{
           \label{RD-2SS-Kite-5-30}
            \includegraphics[width=0.2\textwidth]{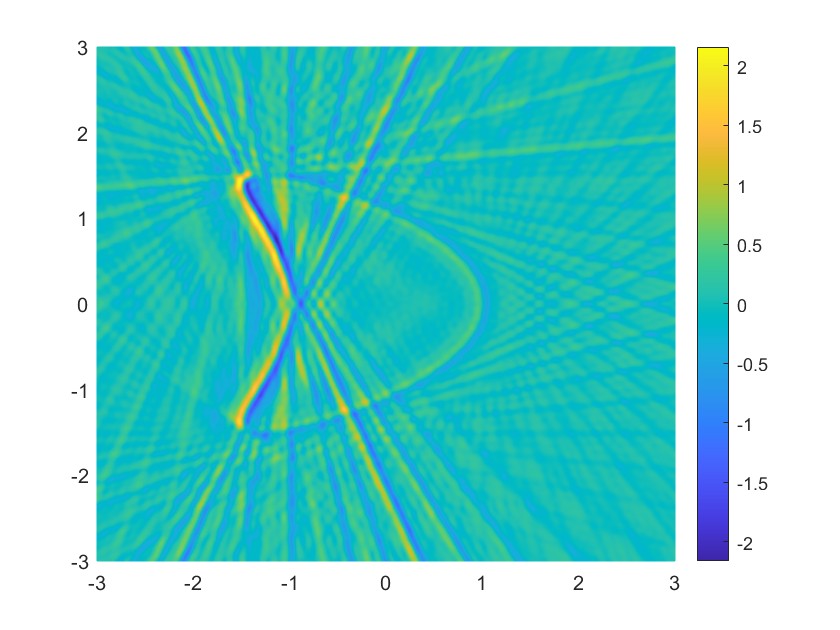}
        }\\
        \subfigure[Dirichlet case.]{
            \label{TRD-D-Kite-5-30}
            \includegraphics[width=0.2\textwidth]{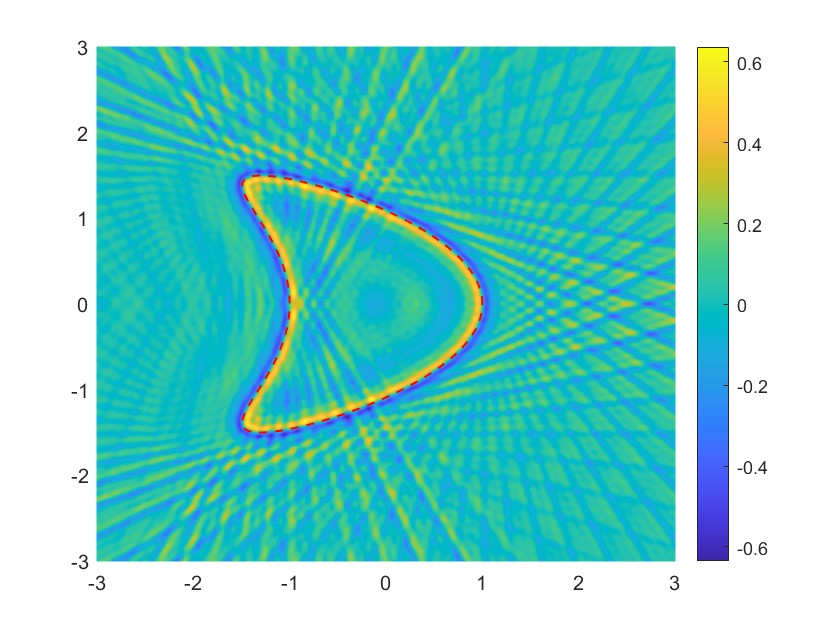}
        }&
                \subfigure[Neumann case.]{
            \label{TRD-N-Kite-5-30}
        \includegraphics[width=0.2\textwidth]{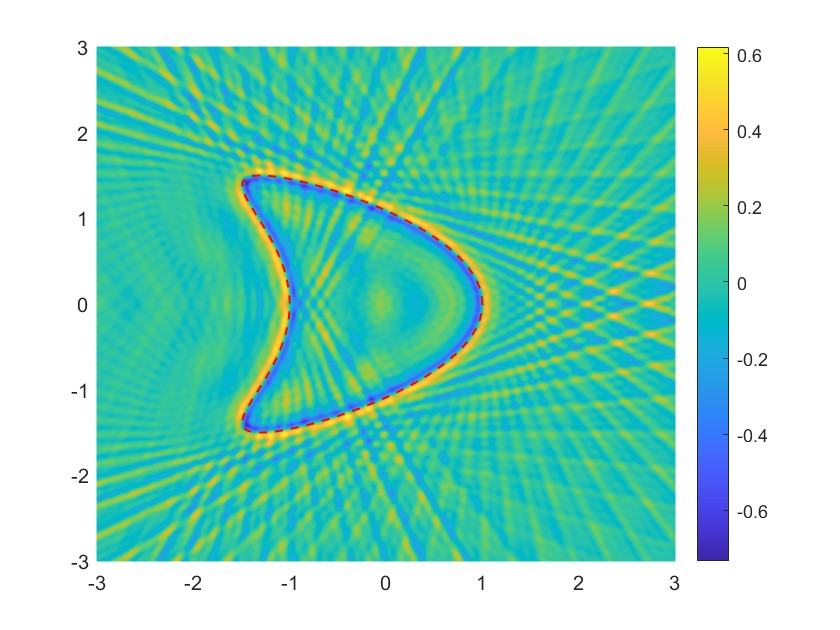}
        }&
                \subfigure[$\lambda_1$ case.]{
            \label{TRD-1S-Kite-5-30}
        \includegraphics[width=0.2\textwidth]{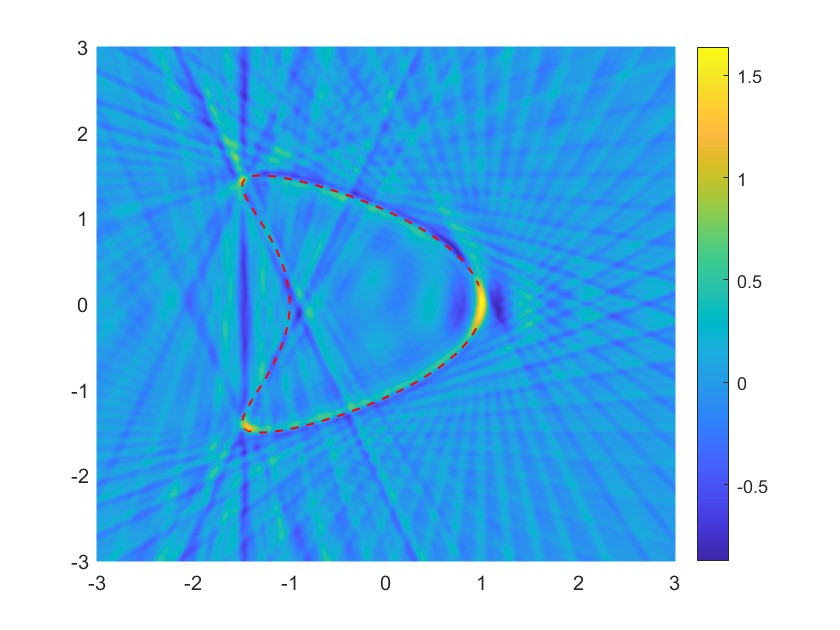}
        }&
        \subfigure[$\lambda_2$ case.]{
            \label{TRD-2SS-Kite-5-30}
            \includegraphics[width=0.2\textwidth]{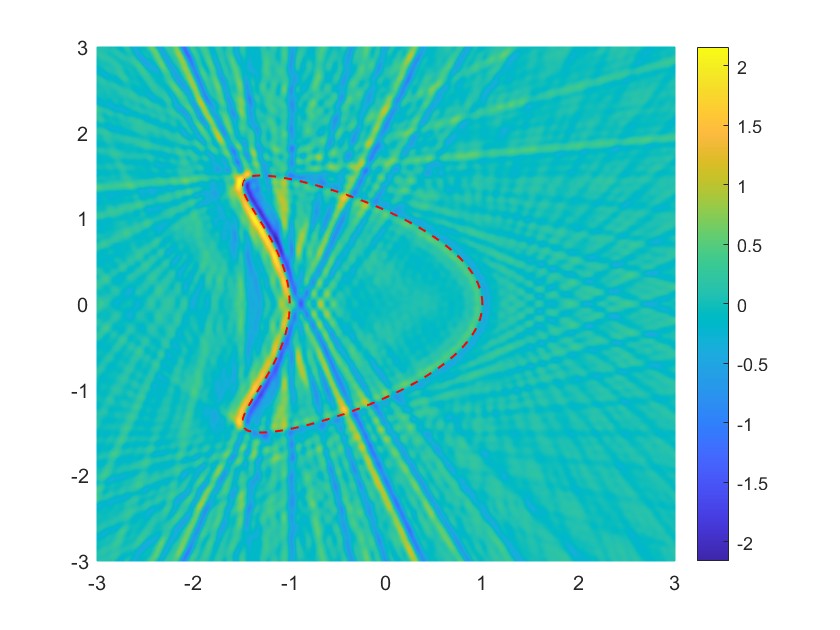}
        }
    \end{tabular}
    \caption{Reconstructions of $\partial D$ with $\mathcal{I}(z)$ in different case using frequency band $[5, 30]$. Top: Plot of $\mathcal{I}(z)$. Bottom: Comparison with true $\partial D$(red dotted curve).}
    \label{RD-Kite-5-30}
\end{figure}

Figure \ref{lambda-Kite-30del} demonstrates that we can obtain a better reconstruction for $\lambda$ via the method \eqref{sovlelambda-num-NeD} with known $\partial D$. Note that $\lambda$ received poor reconstruction near $\beta=-\pi$ because $\mathcal{G}$ is not injective here.
\begin{figure}[htbp]
   \centering
    \begin{tabular}{cc}
     \subfigure[Comparison for $\lambda_1$.]{
            \label{lambda1-Kite-30del-NeD}
        \includegraphics[width=0.3\textwidth]{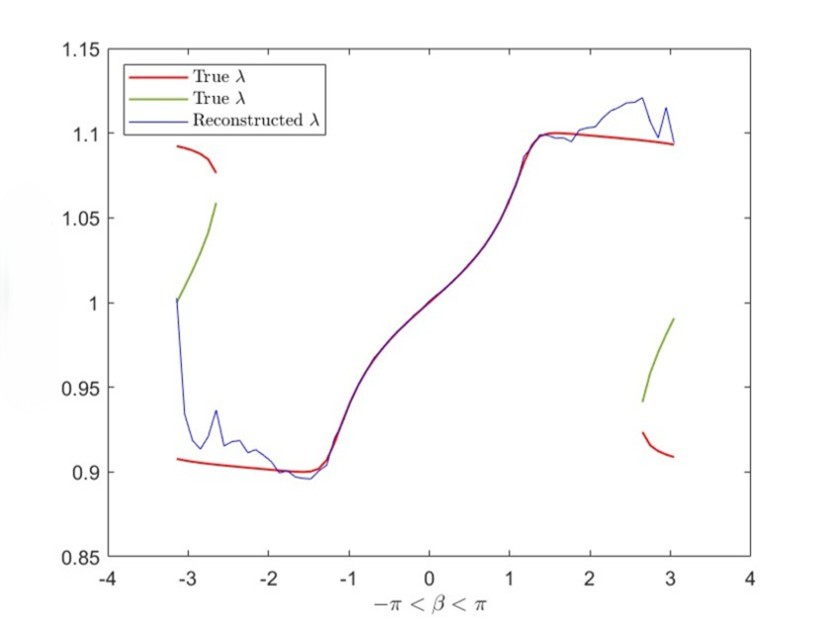}
        }&
        \subfigure[Comparison for  $\lambda_2$.]{
            \label{lambda2-Kite-30del-NeD}
            \includegraphics[width=0.3\textwidth]{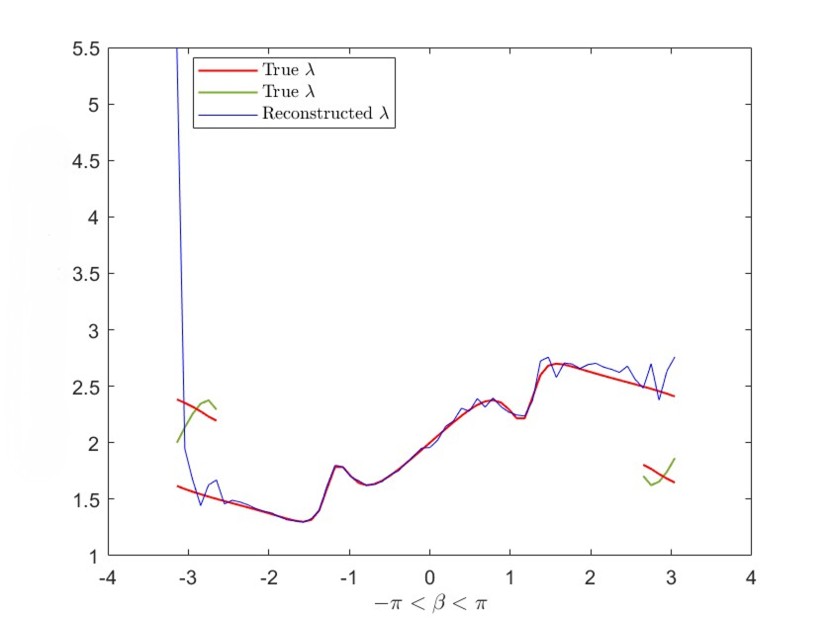}
        }
    \end{tabular}
    \caption{True $\lambda_1$ and $\lambda_2$ and reconstructed by  \eqref{sovlelambda-num-NeD} 
 with $\delta=0.3$ using frequency band $[20, 50]$.}
    \label{lambda-Kite-30del}
\end{figure}

Figure \ref{spa-Kite-20-50} shows that the indicator $\mathcal{T}$ can reconstruct concave obstacles and has fairly good performance. The indicator $\mathcal{T}$ is insensitive to changes in $\lambda$ but is susceptible to the influence of high-curvature structures.
\begin{figure}[htbp]
   \centering
    \begin{tabular}{cccc}
        \subfigure[Dirichlet case.]{
            \label{spa1-D-Kite-20-50}
        \includegraphics[width=0.2\textwidth]{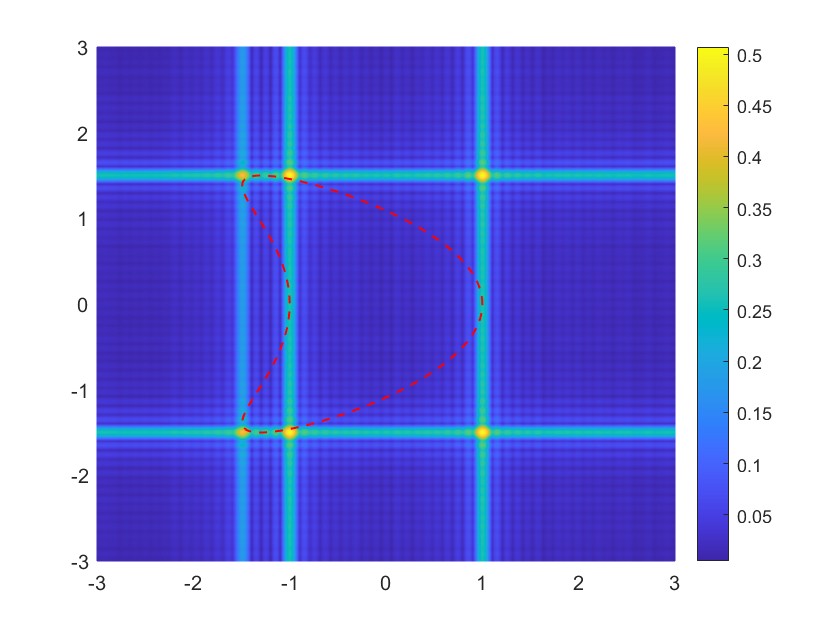}
        }&
                \subfigure[Neumann case.]{
            \label{spa1-N-Kite-20-50}
        \includegraphics[width=0.2\textwidth]{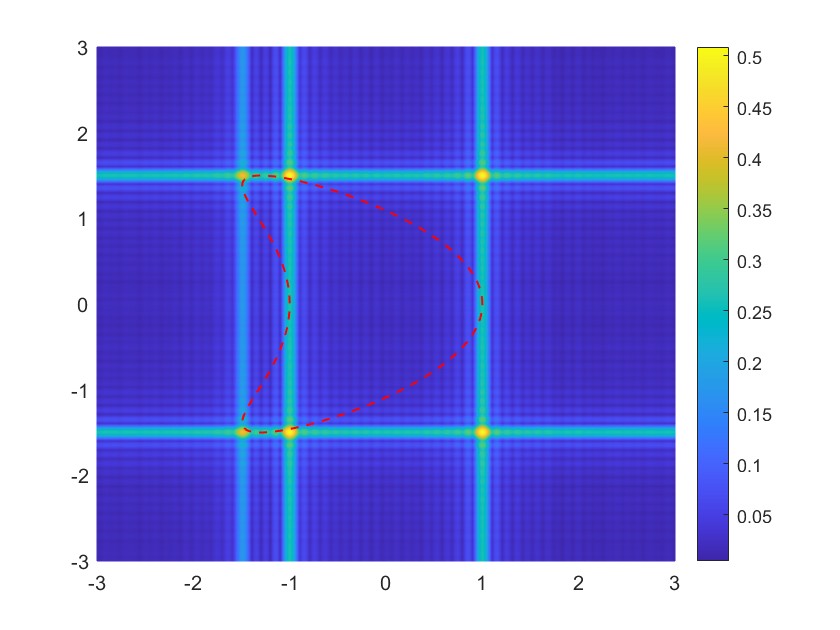}
        }&
                \subfigure[$\lambda_1$ case.]{
            \label{spa1-1S-Kite-20-50}
        \includegraphics[width=0.2\textwidth]{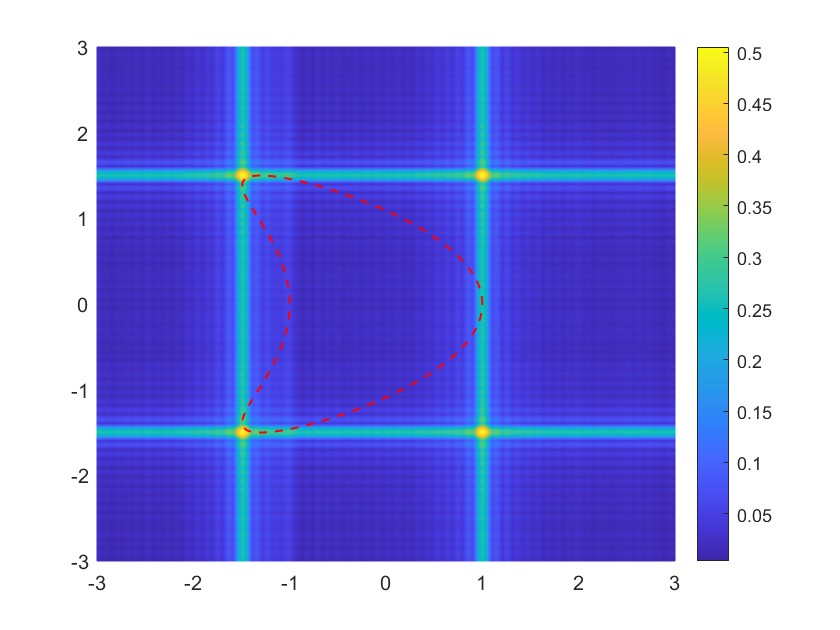}
        }&
        \subfigure[$\lambda_2$ case.]{
           \label{spa1-2SS-Kite-20-50}
            \includegraphics[width=0.2\textwidth]{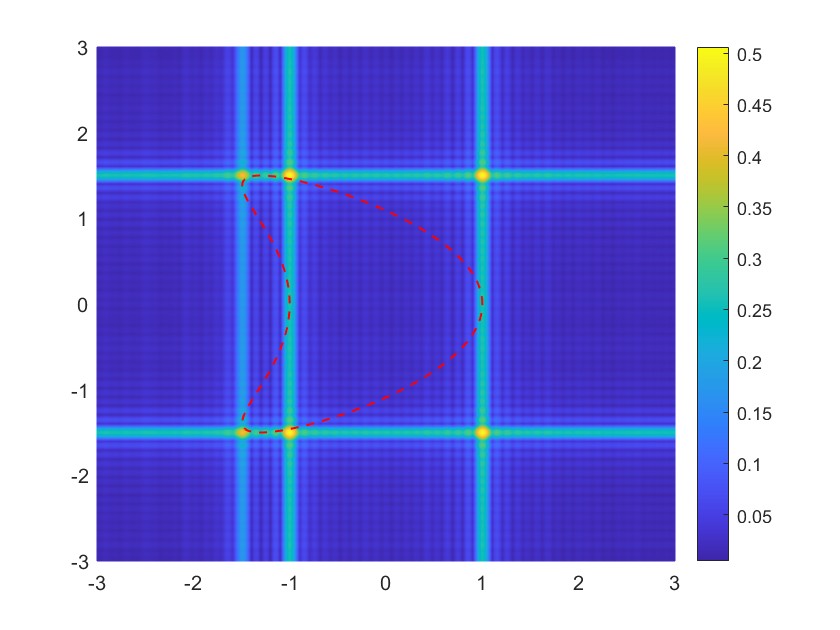}
        }\\
        \subfigure[Dirichlet case.]{
            \label{spa2-D-Kite-20-50}
            \includegraphics[width=0.2\textwidth]{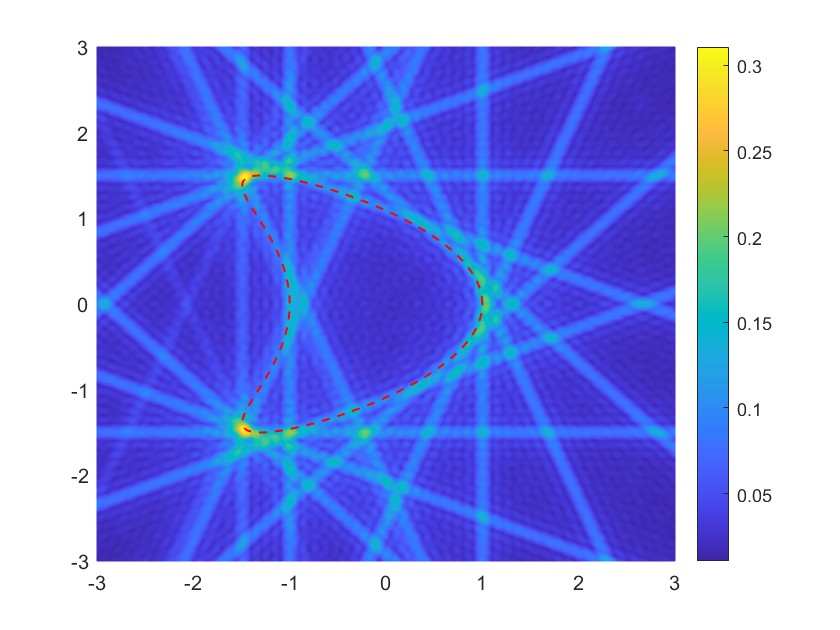}
        }&
                \subfigure[Neumann case.]{
            \label{spa2-N-Kite-20-50}
        \includegraphics[width=0.2\textwidth]{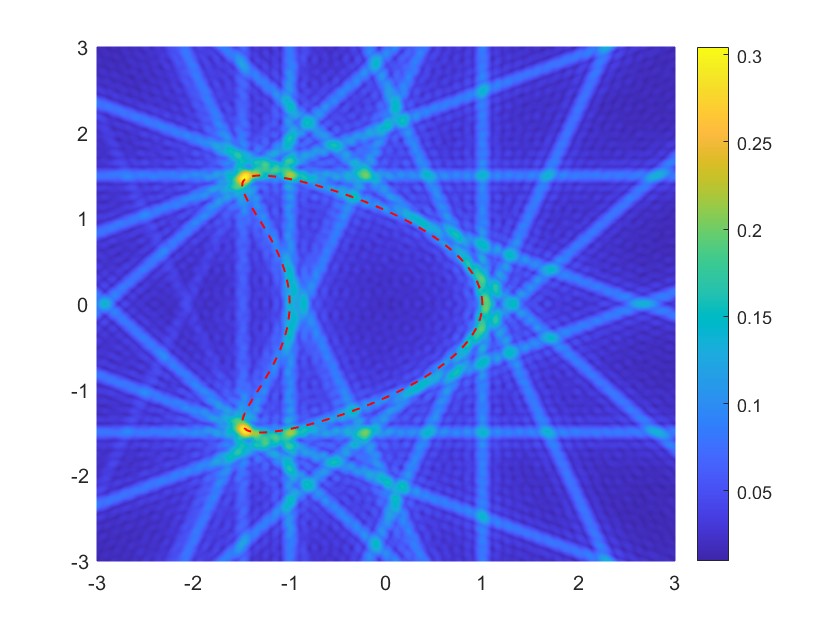}
        }&
                \subfigure[$\lambda_1$ case.]{
            \label{spa2-1S-Kite-20-50}
        \includegraphics[width=0.2\textwidth]{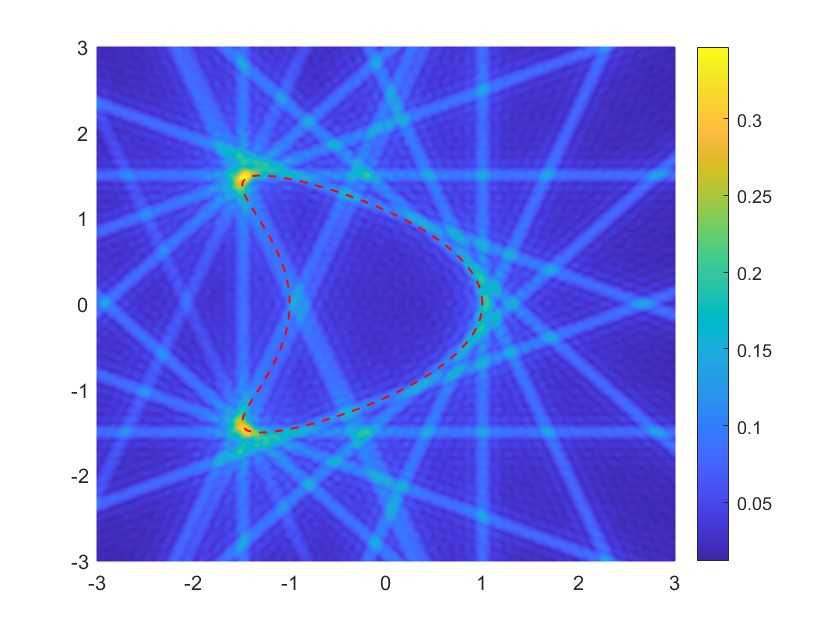}
        }&
        \subfigure[$\lambda_2$ case.]{
            \label{spa2-2SS-Kite-20-50}
            \includegraphics[width=0.2\textwidth]{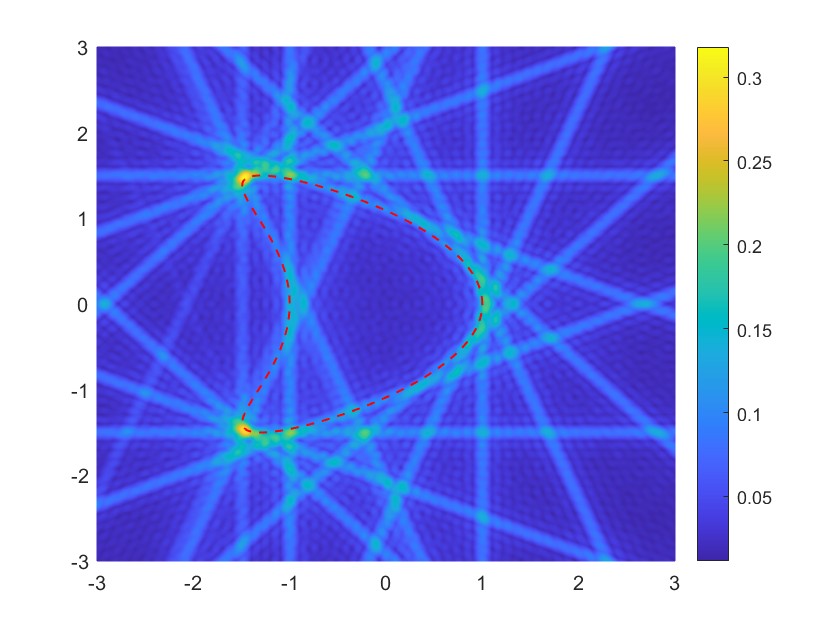}
        }\\
          \subfigure[Dirichlet case.]{
            \label{spa4-D-Kite-20-50}
            \includegraphics[width=0.2\textwidth]{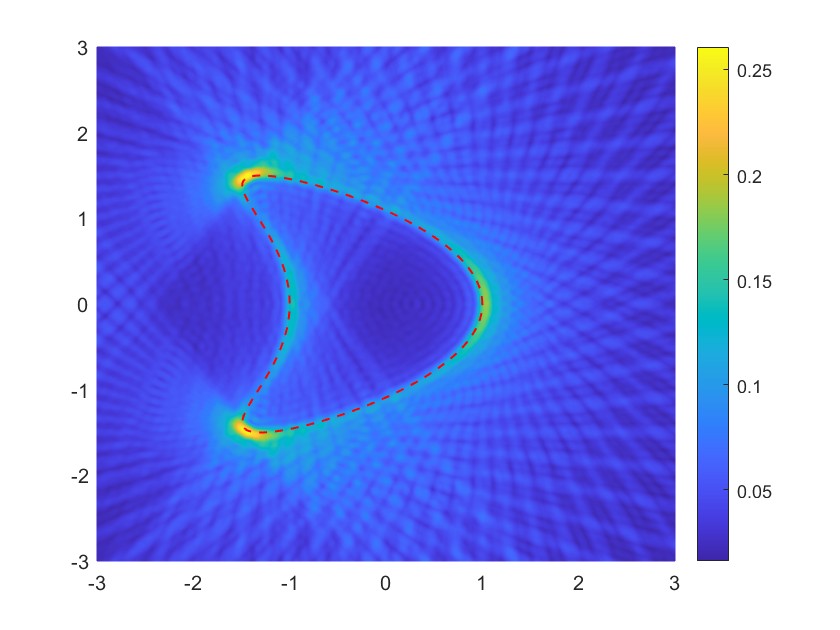}
        }&
                \subfigure[Neumann case.]{
            \label{spa4-N-Kite-20-50}
        \includegraphics[width=0.2\textwidth]{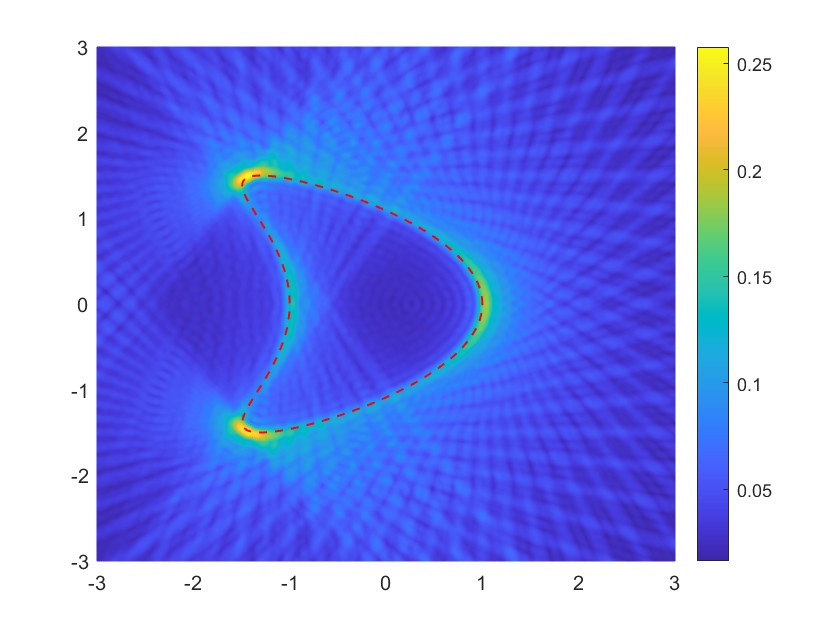}
        }&
                \subfigure[$\lambda_1$ case.]{
            \label{spa4-1S-Kite-20-50}
        \includegraphics[width=0.2\textwidth]{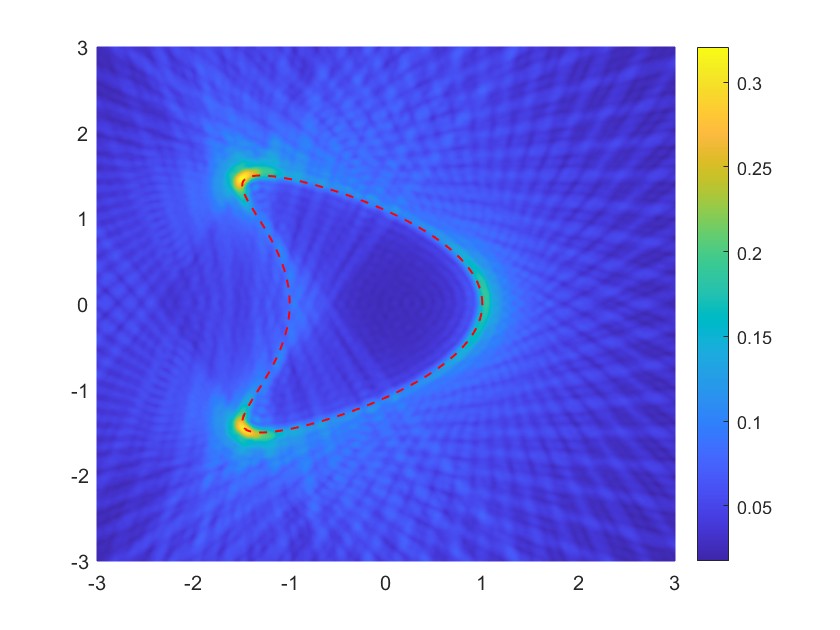}
        }&
        \subfigure[$\lambda_2$ case.]{
            \label{spa4-2SS-Kite-20-50}
            \includegraphics[width=0.2\textwidth]{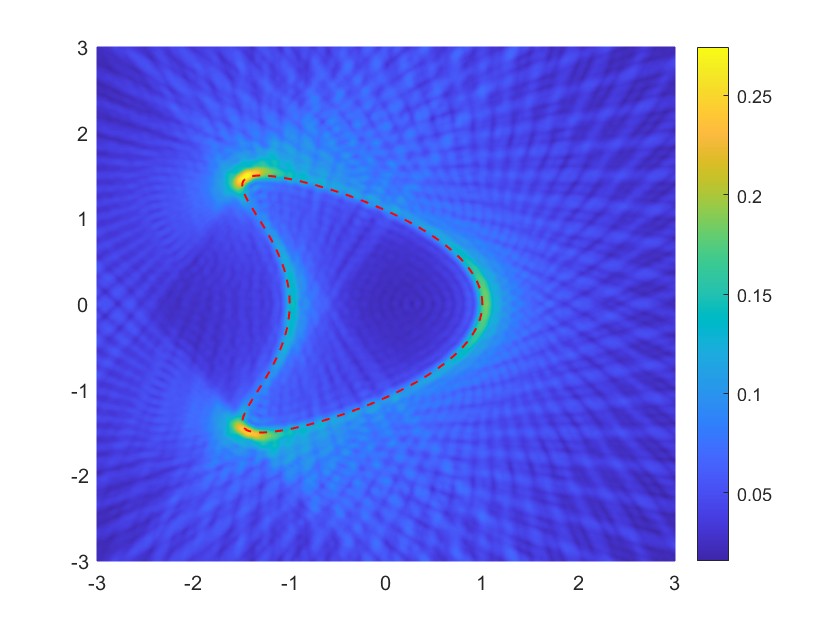}
        }
    \end{tabular}
    \caption{Reconstructions with $\mathcal{T}(z)$ in different cases using $SD$ equally distributed directions data and  frequency band $[20, 50]$, the true boundary $\partial D$ is plotted by a red dotted curve. Top:  $SD=4$. Middle: $SD=16$. Bottom: $SD=64$.}
    \label{spa-Kite-20-50}
\end{figure}

\section{Concluding remarks}

This paper presents a uniqueness result and a non-iterative numerical approach for identifying an obstacle from the multi-frequency backscattering far-field data. 

Our theoretical development builds upon Majda's seminal work, which we extend to cover impedance boundary conditions in two-dimensional scenarios. As a consequence, we derive a high-wavenumber asymptotic expansion of the far-field pattern for strictly convex obstacles. The result covers arbitrary boundary conditions in both two and three dimensions. Most significantly, under the convexity assumption, our analysis shows that backscattering far-field data alone suffice to simultaneously determine both the boundary condition and the geometric shape of the obstacle. To the best of our knowledge, this is the first uniqueness result from the multi-frequency backscattering far-field data.

The inverse obstacle scattering problem exhibits significant ill-posedness 
and nonlinearity. To circumvent this fundamental challenge, we introduce a two-stage reconstruction strategy: first decoupling the impedance parameter estimation from geometric constraints (albeit requiring supplementary multi-directional backscattering data), followed by boundary reconstruction through our innovative sampling-based imaging technique. The numerical simulations demonstrate the exceptional performance of our proposed methodology, successfully reconstructing all key characteristics of the scatterer - including its spatial location, geometric configuration, and impedance parameter - with high accuracy. 

Notably, our algorithm maintains its effectiveness even when applied to concave obstacle geometries, a particularly challenging class of inverse problems. While these empirical results are highly promising, we acknowledge that the underlying theoretical foundations remain to be rigorously established. This critical theoretical aspect will constitute the focus of our subsequent research publications.

\section*{Acknowledgement}
 The research of X. Liu is supported by the National Key R\&D Program of China under grant 2024YFA1012303 and the NNSF of China under grant 12371430. 
 
\bibliographystyle{SIAM}

\begin{thebibliography}{99}

\bibitem{ArensJiLiu2020}
T. Arens, X. Ji and X. Liu, Inverse electromagnetic obstacle scattering problems with multi-frequency sparse backscattering far field data, {\it Inverse Probl. \bf36(10)}, (2020), 105007.





\bibitem{BeilKlib}
L. Beilina and M. V. Klibanov, A new approximate mathematical model for global convergence for acoefficient inverse problem with backscattering data, {\it J. Inverse Ill-Posed Probl. \bf20(4)}, (2012), 513–565.

\bibitem{Bojarski}
N. Bojarski, A survey of the physical optics inverse scattering identity, {\it IEEE Trans. Antennas Propag. \bf20(5)} (1982), 980-989.

\bibitem{CaconiColton2004}
F. Cakoni and D. Colton, The determination of the surface impedance of a partially coated obstacle from far field data, {\it SIAM J. Appl. Math. \bf64}, (2004), 709–723.

\bibitem{CaconiColtonMonk2010}
F. Caconi, D. Colton and P. Monk, The determination of boundary coefficients from far field measurements, {\it 	J. Integr. Equations Appl. \bf22(2)} (2010),  167-191.

\bibitem{ChanGLS2012} S. N. Chandler-Wilde, I. G. Graham, S. Langdon and E.A. Spence, Numerical asymptotic boundary integral methods in high-frequency acoustic scattering, {\it Acta Numer. \bf21} (2012), 89-305.




\bibitem{CCHuang}
J. Chen, Z. Chen and G. Huang, Reverse time migration for extended obstacles: acoustic Waves, {\it Inverse Probl. \bf29} (2013), 085005.

\bibitem{ChengLiuNaka2003}
J. Cheng, J. Liu and G. Nakamura, Recovery of the shape of an obstacle and the boundary impedance from the far field pattern, {\it Kyoto J. Math. \bf43(1)}, (2003), 165-186.

\bibitem{Christiansen2013}
T. J. Christiansen, Inverse obstacle problems with backscattering or generalized backscattering data in one or two directions, {\it Asymptot. Anal. \bf81}, (2013), 315-335.

\bibitem{Colton1982}
D. Colton, Stable methods for determining the surface impedance of an obstacle from low frequency far field data, {\it Appl. Anal. \bf14}, (1982), 61–70.

\bibitem{ColtonKirsch1981}
D. Colton and A. Kirsch, The determination of the surface impedance of an obstacle from measurements of the far field pattern, {\it SIAM J. Appl. Math. \bf41}, (1981), 8–15.

\bibitem{ColtonKirsch}
D. Colton and A. Kirsch, A simple method for solving inverse scattering problems in the resonance region, {\it Inverse Probl. \bf12}, (1996), 383-393.

\bibitem{CK}D. Colton and R. Kress, {\it Inverse Acoustic and Electromagnetic Scattering Theory} (Fourth Edition), Springer, Berlin, 2019.

\bibitem{ColtonMonk}D. Colton and P. Monk, Target identification of coated objects, {\it IEEE Trans. Antennas Propagat. \bf54}, (2006), 1232-1242.


\bibitem{DouLiuMengZhang}F. Dou, X. Liu, S. Meng and B. Zhang, Data completion algorithms and their applications in inverse acoustic scattering with limited-aperture backscattering data, {\it J. Comput. Phys. \bf469}, (2022), 111550.

\bibitem{EskinRalston3}
G. Eskin and J. Ralston, The inverse backscattering problem in three dimensions, {\it Commun. Math. Phys. \bf124} (1989), 169-215.

\bibitem{EskinRalston2}
G. Eskin and J. Ralston, Inverse backscattering problem in two dimensions, {\it Commun. Math. Phys. \bf138} (1991), 456-486.


\bibitem{GriesHS2013}
R. Griesmaier, N. Hyv{\"o}nen and O. Seiskari, A note on analyticity properties of far field patterns, {\it Inverse Probl. Imag. \bf7}, (2013), 491–498.






\bibitem{HeKindSini2009}
L. He, S. Kindermann and M. Sini, Reconstruction of shapes and impedance functions using few far field measurements, {\it J. Comput. Phys. \bf228}, (2009), 717–730.

\bibitem{LH-SPM}L. H\"{o}rmander, {\it The Analysis of Linear Partial Differential Operators I: Distribution Theory and Fourier Analysis} (2nd Edition), Springer, Berlin Heidelberg, 2003.

\bibitem{ItoJinZou}
K. Ito, B. Jin, and J. Zou, A direct sampling method to an inverse medium scattering problem, {\it Inverse Probl. \bf28} (2012), 025003.








\bibitem{Kirsch98}
A. Kirsch, Charaterization of the shape of a scattering obstacle using the spectral data of the far field operator, {\it Inverse Probl. \bf14} (1998),1489--1512.


\bibitem{KirschGrinberg}
A. Kirsch and N. Grinberg, {\it The Factorization Method for Inverse Problems}, Oxford University Press, 2008.


\bibitem{KR95}R. Kress, On the numerical solution of a hypersingular integral equation in scattering theory, {\it J. Comput. Appl. Math. \bf61} (1995), 345-360 . 

\bibitem{KressRun98}R. Kress  and W. Rundell, Inverse obstacle scattering using reduced data,
{\it SIAM J. Appl. Math. \bf59} (1998), 442-454.

\bibitem{KressRundell2001}
R. Kress and W. Rundell, Inverse scattering for shape and impedance, {\it Inverse Probl. \bf17}, (2001), 1075–1085.

\bibitem{KressRundell2018}
R. Kress and W. Rundell, Inverse scattering for shape and impedance revisited, {\it J. Integr. Equations Appl. \bf30}, (2018), 293-331.

\bibitem{Lagergren}
R. Lagergren, The back-scattering problem in three dimensions, {\it J Pseudo-Differ. Oper. \bf2}, (2011), 1–64.


\bibitem{Lee2014}
K. M. Lee, Inverse scattering problem from an impedance obstacle via two-steps method, {\it J. Comput. Phys. \bf274}, (2014), 182–190.


\bibitem{LiLiu2015} J. Li and H. Liu, Recovering a polyhedral obstacle by a few backscattering measurements, {\it J. Differ. Equations \bf259} (2015), 2101-2120.

\bibitem{LiLiuWang2017}
J. Li, H. Liu and Y. Wang, Recovering an electromagnetic obstacle by a few phaseless backscattering measurements, {\it Inverse Probl. \bf33}, (2017), 035011.

\bibitem{LiLiuZou}
J. Li, H. Liu and J. Zou, Locating multiple multiscale acoustic scatterers, {\it SIAM Multi. Model. Simul. \bf12} (2014), 927-952.

\bibitem{LiZou} J. Li and J. Zou,
A direct sampling method for inverse scattering using far field data, {\it Inverse Probl. Imag. \bf7} (2013), 757-775.



\bibitem{LiuNakaSini2007}
J. Liu, G. Nakamura and M. Sini, Reconstruction of the shape and surface impedance from acoustic scattering data for an arbitrary cylinder, {\it SIAM J. Appl. Math. \bf67}, (2007), 1124–1146.

\bibitem{LiuIP17} X. Liu,
A novel sampling method for multiple multiscale targets from scattering amplitudes at a fixed frequency, {\it Inverse Probl. \bf33} (2017), 085011.

\bibitem{LiuShi} X. Liu and Q. Shi, A quantitative sampling method for elastic and electromagnetic sources, {\it J. Comput. Phys. \bf 539} (2025), Paper No. 114251.

\bibitem{LiuSun}
X. Liu and J. Sun, Data recovery in inverse scattering: from limited-aperture to full-aperture, {\it J. Comput. Phys. \bf 386} (2019), 350-364.

\bibitem{Ludwig} 
D. Ludwig, Uniform asymptotic expansion of the field scattered by a convex object at high frequencies,  {\it 	Commun. Pure Appl. Math. \bf 20(1)} (1967), 187-203

\bibitem{Majda}
A. Majda, High frequency asymptotics for the scattering matrix and the inverse problem of acoustical scattering,
{\it Commun. Pure Appl. Math. \bf 29(3)} (1976), 261-291.

\bibitem{Mclean} W. Mclean, {\it Strongly Elliptic Systems and Boundary Integral Equation}, Cambridge University Press, Cambridge, 2000.

\bibitem{Morawetz}
C. Morawetz, Decay for solutions of the exterior problem for the wave equation, {\it 	Commun. Pure Appl. Math. \bf 28(2)} (1975), 229-264.

\bibitem{OlaPS}
P. Ola, L. P\"{a}iv\"{a}rinta and V. Serov, Recovering singularities from backscattering in two dimensions, {\it Commun. Part. Diff. Equat. \bf26}, (2001), 697–715.

\bibitem{Potthast2010} R. Potthast,
 A study on orthogonality sampling, {\it Inverse Probl. \bf26} (2010), 074075.

\bibitem{RakUnl}
Rakesh and G. Uhlmann, Uniqueness for the inverse backscattering problem for angularly controlled potentials, {\it Inverse Probl. \bf30}, (2014), 065005.

\bibitem{Serov}
V. Serov, Inverse fixed angle scattering and backscattering problems in two dimensions, {\it Inverse Probl. \bf24}, (2008), 065002.

\bibitem{Serrenho2006}
P. Serranho, A hybrid method for inverse scattering for shape and impedance, {\it Inverse Probl. \bf22}, (2006), 663–680.

\bibitem{Shin2016}
J. Shin, Inverse obstacle backscattering problems with phaseless data, {\it Eur. J. Appl. Math. \bf27}, (2016), 111–130.



\bibitem{Spence}
E. A. Spence, Wavenumber-Explicit Bounds in Time-Harmonic Acoustic Scattering, {\it SIAM J. Math. Anal. \bf46(4)}, (2014) 2987–3024.

\bibitem{StefUhl}
P. Stefanov and G. Uhlmann, Inverse backscattering for the acoustic equation, {\it SIAM J. Math. Anal. \bf28} (1997), 1191-1204.


\bibitem{Taylor} 
M. Taylor, Grazing rays and reflection of singularities of solutions to wave equations, {\it Commun. Pure Appl. Math. \bf29(1)}, (1976), 1-37.

\bibitem{Uhlmann2001}
G. Uhlmann, A time-dependent approach to the inverse backscattering problem, {\it Inverse Probl. \bf17}, (2001), 703–716.

\bibitem{Wang}
J.-N. Wang, Inverse backscattering problem for the acoustic equation in even dimensions, {\it J. Math. Anal. Appl. \bf220}, (1998), 676–696.


\end{thebibliography}

\end{document}